%% file: Stein.tex
\documentclass{article}

\usepackage{enumitem}   
\usepackage{xr}

\RequirePackage[OT1]{fontenc}
\RequirePackage{amsthm,amsmath,amssymb}
\RequirePackage[numbers]{natbib}
\RequirePackage[colorlinks,citecolor=blue,urlcolor=blue]{hyperref}

\newcommand{\beq}{\begin{equation}}
\newcommand{\eeq}{\end{equation}}
\newcommand\numberthis{\addtocounter{equation}{1}\tag{\theequation}}

\newcommand{\PP}{\mathbb{P}}
\newcommand{\RR}{\mathbb{R}}
\newcommand{\NN}{\mathbb{N}}
\newcommand{\EE}{\mathbb{E}}
\newcommand{\ZZ}{\mathbb{Z}}
\newcommand{\stationary}{\mu}

\newcommand{\test}{\phi}
\newcommand{\step}{s}

\newcommand{\radius}{r}
\newcommand{\Dim}{d}
\newcommand{\timestep}{h} 
\newcommand{\acc}{\epsilon} 
\newcommand{\increment}{\xi} 
\newcommand{\Lip}{L} 

\theoremstyle{plain}
\newtheorem{theorem}{Theorem}
\theoremstyle{remark}
\newtheorem{remark}[theorem]{Remark}
\theoremstyle{plain}
\newtheorem{assumption}[theorem]{Assumption}
\theoremstyle{plain}
\newtheorem{corollary}[theorem]{Corollary}
\theoremstyle{plain}
\newtheorem{lemma}[theorem]{Lemma}
\theoremstyle{plain}
\newtheorem{proposition}[theorem]{Proposition}
\theoremstyle{definition}

\begin{document}

\author{Thomas Bonis\\
	LTCI, T\'el\'ecom Paristech, Universit\'e Paris-Saclay, Paris, France \\
	thomas.bonis@telecom-paristech.fr}

\date{}
	\title{Rates in the Central Limit Theorem and diffusion approximation via Stein's Method}

	\maketitle



  
   



\begin{abstract}
We present a way to use Stein's method in order to bound the Wasserstein distance of order $2$ between two measures $\nu$ and $\mu$ supported on $\RR^d$ under the assumption that $\mu$ is the reversible measure of a diffusion process. In order to apply our result, we only require to have access to a stochastic process $(X_t)_{t \geq 0}$ such that $X_t$ is drawn from $\nu$ for any $t > 0$. We then show that, whenever $\mu$ is the Gaussian measure $\gamma$, one can use a slightly different approach to bound the Wasserstein distances of order $p \geq 1$ between $\nu$ and $\gamma$ under an additional exchangeability assumption on the stochastic process $(X_t)_{t \geq 0}$. Using our results, we are able to obtain convergence rates for the multi-dimensional Central Limit Theorem in terms of Wasserstein distances of order $p \geq 2$. Our results can also provide bounds for steady-state diffusion approximation, allowing us to tackle two problems appearing in the field of data analysis by giving a quantitative convergence result for invariant measures of random walks on random geometric graphs and by providing quantitative guarantees for a Monte Carlo sampling algorithm. 
\end{abstract}



\input{main.tex}

 \input{appendix.tex}

\section*{Acknowledgements}
The author would like to thank J\'er\^ome Dedecker, Michel Ledoux and Yvik Swan for fruitful discussions as well as Chi Tran and Fr\'ed\'eric Chazal for their help during the redaction of the paper. The author is also grateful to anonymous reviewers for their many comments and suggestions. The author was supported by the French D\'el\'egation G\'en\'erale de l'Armement (DGA) and by ANR project TopData ANR-13-BS01-0008.
\bibliographystyle{abbrvnat}
\bibliography{Bibliography}

\end{document}

%% file: main.tex
\section{Introduction}
Stein's method is a general approach to bound distances between two measures and was first introduced by Stein \citep{OriginalStein} 
to provide quantitative bounds for normal approximation. Stein's approach relies on the following observation: the one-dimensional Gaussian measure $\gamma$ is the only measure on $\RR$ such that, for any $\test \in \mathcal{C}^\infty_c(\RR)$, where $\mathcal{C}^\infty_c(E)$ denotes the space of functions defined on the space $E$ with compact support and derivatives of all orders, 
\[
\int_{\RR} \left(-x \test(x) + \test'(x)\right) d\gamma(x) = 0.
\]
Hence, if another probability measure $\nu$ satisfies
\[
\int_{\RR} \left(-x \test(x) + \test'(x)\right) d\nu(x) \approx 0
\]
on a sufficiently large set of functions $\test$, 
then $\nu$ should be close to $\gamma$. Stein's approach was later generalized to the multidimensional Gaussian measure in \citep{Gotze} and more recently to the infinite-dimensional measure of a Brownian motion in \citep{Decreuse1, Decreuse2}. Stein's method was also extended to non-Gaussian target measures such as the Poisson measure \citep{ChenPoisson} as well as more general target measures by Barbour \citep{Barbour}. In Barbour's approach, the target measure $\mu$ is a measure of $\RR^d$ assumed to be the invariant measure of a diffusion process with infinitesimal generators of the form $\mathcal{L}_\mu = b.\nabla + <a, Hess >_{HS}$, where $<.,.>_{HS}$ denotes the Hilbert-Schmidt scalar product on matrices. Under such assumptions, one has 
\[
\forall \test \in \mathcal{C}^\infty_c(\RR^d), \int_{\RR^d} \mathcal{L}_\mu \test \, d\mu = 0
\]
and, similarly to the Gaussian case, 
if another probability measure $\nu$ supported on $\RR^d$ satisfies 
\[
\int_{\RR^d} \mathcal{L}_\mu \test \, d\nu \approx 0
\]
for a sufficiently large set of functions, 
then $\nu$ should be close to $\mu$. 
More precisely, in Barbour's approach one first solves the Stein equation 
\beq
\label{eq:SteinEquation}
u  - \int_{\RR^d} u \, d\mu = \mathcal{L}_\mu f_u
\eeq
for $u$ belonging to a specific set of functions $\mathcal{U}$. 
From here, taking the integral over $\nu$ yields 
\[
\int_{\RR^d} u d\nu - \int_{\RR^d} u \, d\mu = \int_{\RR^d} \mathcal{L}_\mu f_u \, d\nu.
\]
Thus, if 
\[
\sup_{u \in \mathcal{U}} \left|\int_{\RR^d} \mathcal{L}_\mu f_u  \, d\nu \right| \leq \epsilon,
\]
then 
\[
\sup_{u \in \mathcal{U}} \left|\int_{\RR^d} u \, d\nu - \int_{\RR^d} u \, d\mu \right| \leq \epsilon. 
\]
By choosing an appropriate set $\mathcal{U}$, one is then able to obtain a bound on various distances between $\mu$ and $\nu$:
\begin{itemize}
\item if $\mathcal{U} = \{ u : \RR^d \rightarrow \RR \mid \|u\|_\infty \leq 1 \}$, then $\sup_{u \in \mathcal{U}} |\int_{\RR^d} u \, d\nu - \int_{\RR^d} u \, d\mu|$ is the total variation distance between $\mu$ and $\nu$;
\item in dimension $1$, if $\mathcal{U} = \{  u : \RR \rightarrow \RR \mid \exists t \in \RR, u(x) = 1_{x \leq t}  \}$, then $\sup_{u \in \mathcal{U}} |\int_{\RR}  u \, d\nu - \int_{\RR^d} u \, d\mu|$ is the Kolmogorov distance between $\mu$ and $\nu$;
\item if $\mathcal{U} = \{  u : \RR^d \rightarrow \RR \mid \forall x,y \in \RR^d,  |u(y) - u(x)| \leq \|y-x\|  \}$, where $\|.\|$ denotes the Euclidean norm, then $\sup_{u \in \mathcal{U}} |\int_{\RR^d} u \, d\nu - \int_{\RR^d} u \, d\mu|$ is the Wasserstein distance of order $1$, also called Kolmogorov-Rubinstein distance, between $\mu$ and $\nu$. 
\end{itemize}
However, solving Equation~(\ref{eq:SteinEquation}) usually involves computations depending on the target measure $\mu$ which can be difficult to carry out in the multidimensional setting and recent results using this approach usually deal with specific one-dimensional target measures \citep{Braverman1,Braverman3,Braverman2} or assume the target measure $\mu$ satisfies restrictive assumption: for instance \citep{log-concave} solved Equation~(\ref{eq:SteinEquation}) for log-concave target measures $\mu$ on $\RR^d$ with $\mathcal{L}_\mu = \nabla \log f . \nabla + \Delta$, where $f$ is the density of $\mu$. Another downside of this approach is that it can only be used to provide bounds in terms of a distance of the form $\sup_{u \in \mathcal{U}} \left|\int_{\RR^d} u \, d\nu - \int_{\RR^d} u \, d\mu \right|$ for a given set of functions $\mathcal{U}$. It is thus impossible to use this method to obtain bounds for distances such as Wasserstein distances of order $p > 1$ defined by 
\[
W^p_p(\mu,\nu) = \inf_{\pi}  \int_{\RR^d \times \RR^d} \|x-y\|^p \pi(dx,dy),
\]  
where $\pi$ has marginals $\mu$ and $\nu$ and $\|.\|$ is the Euclidean norm. 

Recently, \citep{Stein} proposed a new approach to bound the Wasserstein distance of order $2$ between the multidimensional Gaussian measure and a measure $\nu$. To use their approach, one requires the existence of a fucntion $\tau_\nu$, called a Stein kernel, which satisfies the following integration by parts formula 
\[
\forall \test \in C^\infty_c(\RR^d), \int_{\RR^d} \left( -x . \nabla \test(x) + <\tau_\nu(x), \nabla^2 \test(x)>_{HS} \right) d\nu(x) = 0,
\]
where $<.,.>$ is the Hilbert-Schmidt scalar product. Then, the Wasserstein distance between the Gaussian measure and $\nu$ can be bounded by a discrepancy between $\tau_\nu$ and the identity matrix $I_d$. This method can also be used to provide bounds for Wasserstein distances of any order $p \geq 1$ under a stronger definition of the Stein kernel. 
Moreover, this approach can also be applied to more general target measures $\mu$ assumed to be reversible measures of generators of the form $\mathcal{L}_\mu : \test \rightarrow b.\nabla \test + <a, Hess \, \test>_{HS}$ satisfying some technical assumptions. However, this approach still suffers from a couple issues: assumptions required on the target measures are quite restrictive and computing a Stein kernel for $\nu$, if such a function exists in the first place, can be difficult. We wish to adapt this approach by relaxing the assumptions required on the target measures and by replacing the Stein kernel with another operator $\mathcal{L}_\nu$ such that 
\[
\forall \test \in C^\infty_c(\RR^d), \int_{\RR^d} \mathcal{L}_\nu \test d \nu = 0,
\]
in which case we say $\nu$ is invariant under $\mathcal{L}_\nu$. One can then expect $\mu$ and $\nu$ to be close if $\mathcal{L}_\mu$ and $\mathcal{L}_\nu$ are similar. 
There are many ways to obtain operators $\mathcal{L}_\nu$ under which $\nu$ is invariant; in fact, such operators have been extensively used in Stein's method in order to study properties of the solution of Equation~(\ref{eq:SteinEquation}). For instance, the original approach of Stein \citep{OriginalStein} used pairs of random variables $(X,X')$ drawn from $\nu$ such that $(X,X')$ and $(X', X)$ follow the same law. Given such a pair of random variables $(X,X')$, which is called an exchangeable pair, $\nu$ is invariant under the operator $\mathcal{L}_\nu$ given by 
\[
\forall \test \in \mathcal{C}^\infty_c(\RR^d), \forall x \in \RR^d, \mathcal{L}_\nu \test (x) =  \frac{1}{s} \EE[(X' - X) (\test(X') + \test(X)) \mid X = x],
\]
where $s > 0$ is a rescaling factor. This operator $\mathcal{L}_\nu$ can then be compared to $\mathcal{L}_\mu$ using a Taylor expansion. This exchangeable pairs approach was then extended to the multidimensional setting in \citep{reinert2009}.
Similarly, couplings such as the zero-bias coupling or the size-bias coupling \citep{zerobias,sizebias} used to apply Stein's method also consist in finding particular operators $\mathcal{L}_\nu$ under which $\nu$ is invariant. However, exchangeable pairs are often easier to obtain than other type of couplings: for instance, multidimensional zero-bias couplings computed in \citep{HigherOrderZeroBias} require having access to an exchangeable pair in the first place. 
In fact, in some cases, one does not even require the exchangeability condition to hold to apply Stein's method: in dimension one,
\citep{Rollin} used two random variables $X,X'$, both drawn from $\nu$ but not necessarily exchangeables, to construct operators of the form 
\[
\forall \test \in \mathcal{C}^\infty_c(\RR), \forall x \in \RR^d, \mathcal{L}_\nu \test (x) = \frac{1}{s} \EE\left[\int_{0}^{X'} \test(y) dy  - \int_{0}^{X} \test(y) dy \mid X = x\right],
\]
with $s > 0$, to bound the distance between $\nu$ and the Gaussian measure. 
Unfortunately, this construction is restricted to the one-dimensional setting as there is no notion of primitive function in higher dimension. 
However, even in the multidimensional setting, it is still possible to use pairs of random variables $(X,X')$ to define an operator under which $\nu$ is invariant by taking 
\[
\forall \test \in \mathcal{C}^\infty_c(\RR), \forall x \in \RR^d, \mathcal{L}_\nu \test (x) = \frac{1}{s} \EE[\test(X') - \test(X) \mid X = x]. 
\]
In this work, we adapt the approach from \citep{Stein} to such operators. However, as we will see in this paper, using a single pair of random variable $(X,X')$ is not sufficient to obtain a meaningful bound.
Instead, we use a stochastic process $(X_t)_{t \geq 0}$ such that $X_t$ is drawn from $\nu$ for any $t \geq 0$ from which we derive a family of operators $((\mathcal{L}_\nu)_t)_{t \geq 0}$ under which $\nu$ is invariant by taking  
\beq
\label{eq:operator}
\forall t > 0, \forall \test \in \mathcal{C}^\infty_c(\RR^d), \forall x \in \RR^d, (\mathcal{L}_\nu)_t \test (x) = \frac{1}{s} \EE[\test(X_t) - \test(X_0) \mid X_0 = x],
\eeq
with $s > 0$. 
Using this family of operators in Theorem~\ref{thm:mainGauss}, we are able to bound the Wasserstein distance of order $2$ between $\nu$ and the Gaussian measure $\gamma$ before extending this result to more general target measures in Theorem~\ref{thm:main2}. In order to apply this latter result, the assumptions we require on the target measure $\mu$ are rather weak albeit technical: $\mu$ should be the reversible measure of a diffusion process $(Y_t)_{t \geq 0}$ with infinitesimal generator of the form $\mathcal{L}_\mu = b.\nabla + <a, Hess>_{HS}$ such that the measure of $(Y_t)_{t \geq 0}$ converges exponentially fast to $\mu$. 
Using a more direct approach, we also provide bounds on Wasserstein distances of any order $p \geq 1$ for one-dimensional normal approximation in Theorem~\ref{thm:WpGauss} and for multidimensional normal approximation in Theorem~\ref{thm:WpGaussexch}. This latter result requires an extra exchangeability assumption on the stochastic process $(X_t)_{t \geq 0}$ in order to use a family of operators $((\mathcal{L}_\nu)_t)_{t \geq 0}$ such that, for any $t > 0, \test \in \mathcal{C}^\infty_c(\RR^d), x \in \RR^d$,
\[
(\mathcal{L}_\nu)_t \test (x) = \frac{1}{s} \EE[(X_t - X_0)(\test(X_t) + \test(X_0)) \mid X_0 = x],
\]
where $s > 0$ is a rescaling factor. Let us note that, while we mostly focus on operators defined as in Equation~(\ref{eq:operator}), most of our results can be adapted to Stein kernels of to other type of couplings.  

As a first application of our results, we provide convergence rates for the Central Limit Theorem. 
\begin{theorem}
\label{thm:CTL}
Let $n > 1$ and 
let $X_1,\dots, X_n$ be independent random variables in $\RR^d$ with $\EE[X_1] = 0$ and $\EE[X_1 X_1^T] = I_d$. Let $\nu_n$ be the measure of $S_n = \frac{1}{\sqrt{n}}\sum_{i=1}^n X_i$ and let $\gamma$ be the $d$-dimensional Gaussian measure. For any $p \geq 2$, there exists $C_p > 0$ such that, if $\sum_{i=1}^n \EE[\|X_i\|^{p+q}] < \infty$ for some $q \in [0,2]$, then, taking $m, m' \in (0, \min(4,p+q) - 2]$, we have 
\begin{multline*}
W_p(\nu_n, \gamma) \leq C_p n^{-1/2} \Bigg( n^{-q/2p} \left(\sum_{i=1}^n \EE[\|X_i\|^{p+q}]\right)^{1/p} + 
 n^{-m/4} \left(\sum_{i=1}^n \EE[\|X_i\|^{2+m}]\right)^{1/2} 
 \\+ \sum_{i=1}^n \begin{cases}
 n^{-1/2 -m'/2}  |1-m'|^{-1} \|\EE[X_i X_i^T \|X_i\|^{m'}]\|  \text{ if $m' < 1$} \\
 n^{-1} \|\EE[X_i X_i^T \|X_i\|]\| \log(n) \text{ if $m' = 1$} \\
 n^{-1} |1-m'|^{-1} d^{1/2 - 1/(2m')} \|\EE[X_i X_i^T \|X_i\|^{m'}]\|^{1/m'}    \text{ if $m' > 1$}
 \end{cases}
 \Bigg).
\end{multline*}
\end{theorem}
Overall, this result extends multiple one-dimensional bounds and improves on existing multidimensional ones for identically distributed random variables. 
\begin{itemize}
\item For $p = 2, q=m=m'=2$, since $\EE[\|X_1\|^4]^{1/2} \leq d^{1/4} \|\EE[X_1 X_1^T \|X_1\|^2]\|^{1/2}$, our result gives 
\beq
W_2(\nu_n, \gamma) \leq C_2 n^{-1/2} d^{1/4} \|\EE[X_1 X_1^T \|X_1\|^2]\|^{1/2}.
\eeq
 This bound generalizes the one-dimensional result 
 \[
 W_2(\nu_n, \gamma) \leq C n^{-1/2} \EE[X_1^4]
\]
 obtained in \citep{rio2009}. We also improve on the multidimensional bound of \citep{Zhai} which requires $\|X_1\|$ to be almost surely bounded and suffers from an additional $\log n$ factor. \\
\item For $p>2, q = m = m'=2$, we obtain 
\beq
W_p(\nu_n, \gamma) \leq C_p  n^{-1/2}(\EE[\|X_1\|^{p+2}]^{1/p} + d^{1/4} \|\EE[X_1 X_1^T \|X_1\|^2]\|^{1/2}),
\eeq
generalizing the one-dimensional bounds 
$W_p \leq C_p n^{-1/2} \EE[|X_1|^{p+2}]^{1/p}$ obtained in \cite{Bobkov2016}. 
\item For $p>2, q=0, m = m' = \min(4,p+q) - 2$, we get 
\beq
W_p(\nu_n, \gamma) \leq C_p n^{-1/2+1/p} (\EE[\|X_1\|^p]^{1/p} + o(1)),
\eeq
extending a one-dimensional result from \cite{Sakhanenko}.
\end{itemize}
While our bounds are optimal with respect to the dependency in $n$, they can still be improved. Indeed, for $p,q = 2$, our bound scales at least linearly with respect to the dimension. Yet, if all the coordinates of the variables $X_1,\dots,X_n$ are independent, then one can use the one-dimensional result to obtain a bound which scales with the square root of the dimension. Hence, it is likely that tighter bounds can be obtained under additional assumptions. For instance, if the measures of the random variables satisfy a Poincar\'e inequality for a constant $C>0$, it is possible to use an approach based on the Stein kernel to bound the Wasserstein distance of order $2$ between the measure of $S_n$ and the Gaussian measure by $n^{-1/2} \sqrt{d (C - 1)}$ \citep{Fathi}, improving on our result whenever the constant $C$ is small with respect to the square root of the dimension. 

As another application of our results, we study the problem of steady-state diffusion approximation. In this problem one considers a Markov chain approximating a diffusion process and wants to compare the invariant measure of the Markov chain to the invariant measure of the approximated diffusion process. This topic is at the center of a recent series of papers \citep{Braverman1,Braverman3,Braverman2} in which Stein's method was used to obtain quantitative results. However, the computations used to derive these results are specific to the Markov chains and diffusion processes considered. In this paper, we give a general result for steady-state diffusion approximation in Corollary~\ref{cor:rw} in which we bound the Wasserstein distance of order $2$ between the invariant measure of the Markov chain and the reversible measure of the approximated diffusion process. 
Let us note that, since our approach does not require any exchangeability assumption, our result can be applied to the invariant measure of any Markov chains. On the other hand, using an approach based on exchangeable pairs would require the Markov chain considered to be reversible.
Using our result, we are able to tackle a couple of problems relative to the field of data analysis and involving invariant measures of non-reversible Markov chains. We start by providing quantitative bounds for the convergence of invariant measures of random walks on random neighbourhood graphs. We then evaluate the complexity of a Monte-Carlo algorithm for approximate sampling called Langevin Monte Carlo. 

The paper is organized as follows. In Section~\ref{sec:notations}, we introduce the notations used in this paper. In Section~\ref{sec:setting} we present the general bounds for Wasserstein distances obtained in this paper. In Section~\ref{sec:diff}, we present the steady-state diffusion approximation result along with its applications on random walks on random neighbourhood graphs and on Monte Carlo sampling. 
Then, in Section~\ref{sec:approach}, we present the main arguments we use to apply Stein's method and obtain bounds on the Wasserstein distance of order $2$ before presenting the approach followed to obtain bounds on Wasserstein distances of any order $p$ for normal approximation in Section~\ref{sec:Wp}. The necessary computations to derive Theorem~\ref{thm:CTL} from these bounds are then provided in Section~\ref{subsec:CTLcomp}. Section~\ref{sec:corrw} contains the proof of our steady-state diffusion approximation result. The computations required to apply this last result 
 to study the convergence of the invariant measures of random walks on random neighbourhood graphs and to bound the complexity of the Langevin Monte Carlo algorithm are proved in Sections~\ref{subsec:knn} and \ref{subsec:LMCproof} respectively. 
Finally, Sections~\ref{sec:technical} and \ref{sec:approx} contains the technical results and approximation arguments used to derive the general Wasserstein bounds. 
\section{Notations and definitions}
\label{sec:notations}
Let $d$ be a positive integer. 
For any $k \in \NN$, we denote by $(\RR^d)^{\otimes k}$ the set of elements of the form $(x_j)_{j \in \{1,\dots,d\}^k} \in \RR^{d^k}$.
For $x \in \RR^d$ and $k \in \NN$, we denote by $x^{\otimes k}$ the element of $(\RR^d)^{\otimes k}$ such that 
\[
\forall j \in \{1,\dots,d\}^k, (x^{\otimes k})_{j} = \prod_{i=1}^k x_{j_i}.
\]
For any $x,y \in (\RR^d)^{\otimes k}$ and any symmetric positive-definite $d \times d$ matrix $A$, we pose 
\[
<x,y>_A = \sum_{l ,j \in \{1,\dots,d\}^k} x_l y_j \prod_{i=1}^k A_{l_i,j_i},
\]
and, by extension, 
\[
\|x\|^2_A = <x,x>_A.
\]
We denote by $<.,.>$ the traditional Hilbert-Schmidt scalar product, corresponding to $<.,.>_{I_d}$, and by $\|.\|$ the associated norm. 
Finally, for any $k \in \NN$ and any $x \in (\RR^d)^{\otimes k}$, we pose 
\beq
\label{eq:hnorm}
\|x\|_H = \sum_{j \in \{1,\dots,d\}} \sum_{i \in \{1,\dots,d\}^{k-1}} x_{j, i_1,\dots,i_{k-1}} \prod_{n=1}^d \left(\sum_{l=1}^{k-1} \delta_{i_l,n}\right)!. 
\eeq

For any two spaces $E,F \subset (\RR^d)^{\otimes k}$, we denote by $C^k(E,F)$ the set of functions from $E$ to $F$ with partial derivatives of order $k \in \NN$ and we denote by 
$C^k_c(E,F)$ the set of such functions with compact support. For any function $\test\in C^\infty( \RR^d ,\RR)$ and any $x \in \RR^d$, we denote by $\nabla^k \test \in (\RR^{d})^{\otimes k}$ the $k$-th gradient of $\test$:
\[
\forall j \in \{1,\dots,d\}^k, (\nabla^k \test (x))_j = \frac{\partial^k \test}{\partial x_{j_1} \dots \partial x_{j_k} }(x).
\]

Consider a connected and open set $E \subset \RR^d$ and a matrix-valued function $a : E \rightarrow (\RR^d)^{\otimes 2}$ such that $a$ is positive-definite on all of $E$. For any $x\in E$, $a(x)$ admits an inverse matrix $a^{-1}(x)$.
We denote by $d_a$ the metric on $E$ induced by $a$ defined by 
\[
\forall x,y \in E, d_a(x,y) = \inf_{\gamma} \int_0^1 \|\gamma'(t)\|_{a^{-1}(\gamma(t))} dt, 
\]
where the infimum is taken over all curves $\gamma \in C^1( [0,1], E)$ such that $\gamma(0) = x$ and $\gamma(1)= y$. 

Given two probability measures $\mu$ and $\nu$ on $E$ and $p \geq 1$, we denote by $W_{p,a}$ the Wasserstein distance of order $p$ with respect to the metric $d_a$:
\[
W_{p,a}^p =  \inf_{\pi}  \int_{E \times E} d_a(x,y)^p \pi(dx,dy),
\]  
where $\pi$ has marginals $\mu$ and $\nu$. Finally, we denote $W_{p,I_d}$ by $W_p$. 
\section{Bounds for Wasserstein distances}
\label{sec:setting}
\subsection{Gaussian case}
Let $d$ be a positive integer, $\gamma$ be the $d$-dimensional Gaussian measure and let $\nu$ be a probability measure on $\RR^d$ with finite second moment. 
Let us denote by $\mathcal{L}_\gamma$ the operator acting on $C^\infty_c(\RR^d, \RR)$ such that 
\[
\forall \test \in C^\infty(\RR^d, \RR), x \in \RR^d, \mathcal{L}_\gamma \test (x) = -x . \nabla \test(x) + <I_d, \nabla^2 \test(x)>.
\]
As we have mentioned in the introduction, $\gamma$ is the only measure satisfying 
\[
\forall \test \in C_c^\infty(\RR^d, \RR), \int_{\RR^d} \mathcal{L}_\gamma \test(x) d\gamma(x) = 0.  
\]
Now, let $\nu$ be another measure and let $(X_t)_{t \geq 0}$ be a stochastic process such that $X_t$ is drawn from $\nu$ for any $t \geq 0$. Let $s > 0$, for $t > 0$ we pose 
\[
\forall \test \in C^\infty(\RR^d, \RR), \forall x \in \RR^d, (\mathcal{L}_\nu)_{t} \test (x) = \frac{1}{s} \EE[\test(X_t) - \test(X_0) \mid X_0 = x].
\]
Then, for any $t > 0$ and any $\test \in C^\infty_c(\RR^d, \RR)$, we have 
\[
\int_{\RR^d}  (\mathcal{L}_\nu)_{t} \test (x) d\nu(x) = \frac{1}{s}\EE[\EE[\test(X_t) - \test(X_0) \mid X_0]] = \frac{1}{s} \EE[\test(X_t) - \test(X_0)] = 0. 
\]
Using a Taylor expansion, we have, for any $\test \in C^\infty_c(\RR^d, \RR)$, any $t > 0$ and any $x \in \RR^d$,
\begin{multline*}
(\mathcal{L}_\nu)_{t} \test (x) = \frac{1}{s} \EE\bigg[\left<X_t - X_0, \nabla \test(X_0)\right>  \\
+ \frac{1}{2} \left<(X_t - X_0)^{\otimes 2}, \nabla^2 \test (X_0) \right> + O\left(\|X_t - X_0\|^3 \right) \mid X_0 = x\bigg].
\end{multline*}
Hence, if 
\begin{itemize}
\item $\EE\left[\frac{X_t- X_0}{s} \mid X_0\right]$ is close to $-X_0$;
\item $\EE\left[\frac{1}{2s} (X_t - X_0)^{\otimes 2} \mid X_0 \right]$ is close to $I_d$ and 
\item $\frac{\|X_t - X_0\|^3}{s}$ is small
\end{itemize}
for all $t > 0$, then $\mathcal{L}_\nu$ should be close to $\mathcal{L}_\gamma$ and thus $\nu$ should be close close to $\gamma$. 
In this paper, we turn this intuition into an actual bound for the Wasserstein distance of order $2$ between $\nu$ and $\gamma$. 

\begin{theorem}
\label{thm:mainGauss}
Let $\nu$ be a probability  measure on $\RR^d$ with finite second moment and let $(X_t)_{t \geq 0}$ be a stochastic process such that $X_t$ is drawn from $\nu$ for any $t > 0$. 
Suppose that
\beq
\label{eq:condition}
\forall \epsilon > 0, \exists \xi, M > 0, \forall t \in [\epsilon, \epsilon^{-1}],  \EE\left[e^{\frac{(1+\xi) \|X_t - X_0\|^2}{e^{2t}-1}}\right] \leq M. 
\eeq
Then, for any $s > 0$, 
\[
W_2(\nu,\gamma) \leq \int_0^\infty e^{-t} \EE[S(t)]^{1/2} dt, 
\]
where 
\begin{align*}
S(t) = &   \left\|\EE\left[\frac{X_t-X_0}{s} + X_0\mid X_0\right] \right\|^2  \\
& + \frac{1}{e^{2t} - 1}\left\|\EE\left[\frac{(X_t-X_0)^{\otimes 2}}{2s} - I_d \mid X_0\right] \right\|^2 \\
& + \sum_{k=3}^{\infty} \frac{1}{ (s k!)^2 (e^{2t} - 1)^{k-1} } \|\EE[(X_t-X_0)^{\otimes k}\mid X_0]\|_H^2.
\end{align*}
\end{theorem}


\begin{remark} 
If we were using a single pair of random variables $(X,X')$ rather than the stochastic process $(X_t)_{t \geq 0}$ (that is, if $X_t = X'$ for any $t > 0$), then the function $S$ defined in the Theorem would not be integrable for small values of $t$ unless $X' = X$ in which case one would obtain a trivial bound for $W_2(\nu,\gamma)$. 
\end{remark}

\begin{remark}
\label{rem:condition}
A condition such as Equation~(\ref{eq:condition}) appears in all our bounds derived using Stein's method. In practice, when applying these results to obtain Theorem~\ref{thm:CTL} and Propositions~\ref{pro:randomgraph} and \ref{pro:LMC}, we use processes $(X_t)_{t \geq 0}$ such that $\forall t > 0, \|X_t - X_0\| \leq C \sqrt{t}$, for some $C > 0$, thus verifying such a condition. 
\end{remark}

When dealing with the Gaussian measures, we are also able to bound the Wasserstein distances of order $p \geq 1$ with a similar approach whenever $d = 1$. 

\begin{theorem}
\label{thm:WpGauss}
Let $p \geq 1$.
Let $\nu$ be a probability measure on $\RR$ with finite moment of order $p$. Let $(X_t)_{t \geq 0}$ be a stochastic process such that $X_t$ is drawn from $\nu$ for any $t > 0$. 
Suppose that
\[
\forall \epsilon > 0, \exists \xi, M > 0, \forall t \in [\epsilon, \epsilon^{-1}],  \EE\left[e^{\frac{ p (1 + \xi) \max(1, p-1) |X_t - X_0|^2}{2(e^{2t}-1)}}\right] \leq M. 
\]
Then, for any $s > 0$,
\[
W_p(\nu,\gamma) \leq \int_0^\infty e^{-t} \EE[S_p(t)^{p/2}]^{1/p} dt,
\]
where 
\begin{align*}
S_p(t) =  &  \EE\left[\frac{X_t-X_0}{s} +X_0 \mid X_0\right]^2 \\
& + \frac{\max(1, p-1)}{e^{2t} -1}  \EE\left[\frac{(X_t-X_0)^{2}}{2 s}  - 1 \mid X_0 \right]^2 \\
& +  \sum_{k=3}^{\infty} \frac{\max(1, p-1)^{k-1}}{s^2 k k! (e^{2t}-1)^{k-1} }  \EE[(X_t-X_0)^{k}\mid X_0]^2.
\end{align*}
\end{theorem}

This last result can also be extended to the multidimensional setting at the cost of an additional exchangeability assumption. 

\begin{theorem}
\label{thm:WpGaussexch}
Let $p \geq 1$.
Let $\nu$ be a probability measure on $\RR^d$ with finite moment of order $p$. Let $(X_t)_{t \geq 0}$ be a stochastic process such that $X_0$ is drawn from $\nu$ and such that the pairs $(X_0,X_t)$ and $(X_t, X_0)$ follow the same law for any $t > 0$. 
Suppose that, for any $\epsilon > 0$,
\[
\exists \xi, M > 0, \forall t \in [\epsilon, \epsilon^{-1}],  \EE\left[\|X_t - X_0\|^{p(1+\xi)} e^{\frac{ p(1+\xi) \max(1, p-1) \|X_t - X_0\|^2}{2(e^{2t}-1)}}\right] \leq M. 
\]
Then, for any $s > 0$, 
\[
W_p(\nu,\gamma) \leq \int_0^\infty e^{-t} \EE[S_p(t)^{p/2}]^{1/p} dt,
\]
where, 
\begin{align*}
S_p(t)  & = \left\|\EE\left[\frac{X_t-X_0}{s}+X_0 \mid X_0\right]\right\|^2 \\
& +  \frac{\max(1, p-1)}{e^{2t}-1}  \left\|\EE\left[\frac{(X_t-X_0)^{\otimes 2}}{2 s} - I_d \mid X_0 \right]\right\|^2 \\
& +  \sum_{k=3}^{\infty}  \frac{\max(1, p-1)^{k-1}}{4 (s (k-1)!)^2 (e^{2t}-1)^{k-1}}   \left\|\EE[(X_t-X_0)^{\otimes k} \mid X_0] \right\|_H^2.
\end{align*}
\end{theorem}
\subsection{General case}
\label{sec:general}

Let $d$ be a positive integer, $E'$ be a domain of $\RR^d$ and let $E \subset E'$ be a convex domain such that $0 \in E$. Consider two functions $a \in C^\infty(E', (\RR^d)^{\otimes 2})$ and $b \in C^\infty(E, \RR^d)$ such that 
$a(x)$ is symmetric and positive-definite for any $x \in E'$ and $(E', d_a)$ is a complete metric space. 

Let $\mu$ be a measure on $E$ and let us assume this measure is the reversible measure of a Markov process $(P_t)_{t \geq 0}$ with infinitesimal generator $\mathcal{L}_\mu$ where 
\[
\forall \test \in C^\infty_c(E, \RR), x \in E, \mathcal{L}_\mu f(x) = b(x). \nabla f(x) + <a(x), \nabla^2 f(x)>.
\]
We denote by $\Gamma_1$ the carr\'e du champ operator defined by
\[
\forall f,g \in C^\infty(E, \RR), \Gamma_1(f,g) = <\nabla f, \nabla g>_a
\]
and  by $\Gamma_2$ the operator defined for any $f,g \in C^\infty(E, \RR)$ by 
\beq
\label{eq:gamma2}
\Gamma_2(f,g) = 
\frac{1}{2} \left[ \mathcal{L}_\mu (\Gamma_1(f,g)) - \Gamma_1(\mathcal{L}_\mu f, g) - \Gamma_1( f, \mathcal{L}_\mu g) \right].
\eeq
Moreover, we assume that there exists $\rho \in \RR$ such that 
\[
\forall \test \in C^\infty(E, \RR), \Gamma_2(\test,\test) \geq \rho \Gamma_1(\test,\test).
\]
In the framework of \citep{Markov}, the structure $(E, \mu, \Gamma_1)$ is called a Markov triple and, when the previous assumption is verified, one says that the Markov triple verifies $CD(\rho, \infty)$ condition.
Under this assumption, 
 the semigroup $(P_t)_{t \geq 0}$ enjoys many regularizing properties which will prove crucial in our approach, see Proposition~\ref{pro:curvdim}. 
 We also assume that, for any measure $\eta$ such that $d \eta = h d \mu$, the measure $\eta_t$ with $d\eta_t = P_t h \, d\mu$ converges exponentially fast to $\mu$. More precisely, we assume there exists $c  \geq 1, \kappa > 0$ such that 
\[
\forall t > 0, W_{2,a}(\nu_t,\mu) \leq c e^{-\kappa t}  W_{2,a}(\nu,\mu). 
\]
While such an exponential convergence to $\mu$ is verified whenever $\rho > 0$ with $\kappa = \rho$ and $c = 1$ (see Theorem 9.7.2 \citep{Markov}), it can also be obtained under weaker assumptions. For example,  
if $a = I_d$ and $b = -\nabla V$, where $V \in C^\infty(\RR^d, \RR)$, this property is satisfied whenever $V$ is strongly convex outside a bounded set $C$ 
with bounded first and second order derivatives 
on $C$ \citep{GGB}. An extension of this result for more general functions $a$ is proposed in Theorem 2.1 \citep{Fwang}.

Let us summarize the assumptions we have made on $\mu$ so far. 
\begin{assumption}
\label{ass:main}
\begin{itemize}
\item[(i)] $E'$ is a domain of $\RR^d$ and $E \subset E'$ is a convex domain such that $0 \in E$.
\item[(ii)] $b \in C^\infty(E, \RR^d)$ and $a \in C^\infty(E', (\RR^d)^{\otimes 2})$ is symmetric positive definite on all of $E'$ and $(E', d_a)$ is a complete metric space.
\item[(iii)] $\mu$ is a probability measure such that $d_a(.,0) \in L_2(\mu)$ and $\|b\|_{a^{-1}} \in L_2(\mu)$. Furthermore, $\mu$ is a reversible measure of the operator $\mathcal{L}_\mu . = <b, \nabla .> + <a, \nabla^2 .>$.
\item[(iv)] The operator $\mathcal{L}_\mu$ is the infinitesimal generator of the Markov semigroup $(P_t)_{t \geq 0}$.
\item[(v)] There exists $\rho \in \RR$ such that 
\[
\forall \test \in C^\infty(E, \RR), \Gamma_2(\test,\test) \geq \rho \Gamma_1(\test,\test). 
\]
\item[(vi)] There exists $c \geq 1,\kappa > 0$ such that, for any probability measure $\eta$ verifying $d_a(.,0) \in L_2(\eta)$ and such that $d\eta = h d\mu$, 
 \[
 W_{2,a}(\eta_t,\mu) \leq c e^{-\kappa t}  W_{2,a}(\eta,\mu),
 \]
 where $\eta_t$ has measure $P_t h$.
\end{itemize}
\end{assumption}
Along with these assumptions on $\mu$, we require several assumptions on $\nu$ and on the stochastic process $(X_t)_{t \geq 0}$. Before stating these assumptions, let us introduce  a set of functions $(f_k)_{k \in \NN^\star}$ where, for any $k \in \NN^\star$, 
\[
f_k(t) = \begin{cases} 
 e^{- \rho t \max(1,k/2)} \left( \frac{2\rho d}{e^{2\rho t/(k-1)}-1} \right)^{(k-1)/2}  \textit{ if } \rho \neq 0 \\
\left(\frac{d (k-1)}{t}\right)^{(k-1)/2} \textit{ if } \rho = 0 
\end{cases}.
\]
\begin{assumption}
\label{ass:mainnu}
\item[(i)] $\nu$ is a probability measure on $E$ such that $d_a(.,0) \in L_2(\nu)$ and $\|b\|_{a^{-1}} \in L_2(\nu)$.
\item[(ii)] $(X_t)_{t \geq 0}$ is a stochastic process such that $X_t$ is drawn from $\nu$ for any $t > 0$. 
\item[(iii)]  
$
\forall t_0, t_1, \exists \xi, M > 0, \forall t \in [t_0, t_1], \EE[d_a(X_0, X_t)^{2 + \xi}] \leq M. 
$
\item[(iv)] 
$
\forall t_0, t_1, \exists \xi, M > 0, \forall t \in [t_0, t_1], \EE\left[\left(\sum_{k=1}^\infty \frac{f_k(t)}{k!} \left\|X_t-X_0\right\|^{k}_{a^{-1}(X_0)} \right)^{2 + \xi} \right] \leq M. 
$
\end{assumption}
Under these assumptions, we obtain the following bound. 
\begin{theorem}
\label{thm:main2}
Suppose Assumptions~\ref{ass:main} and \ref{ass:mainnu} are verified. 
Let $s> 0 $ and let 
\begin{align*}
 S(t) = & f_1(t) \left\|\EE\left[\frac{X_t-X_0}{s} - b(X_0) \mid X_0\right] \right\|_{a^{-1}(X_0)}\\
& +  f_2(t) \left\|\EE\left[\frac{(X_t-X_0)^{\otimes 2}}{2s} - a(X_0)\mid X_0\right] \right\|_{a^{-1}(X_0)}\\
& + \sum_{k=3}^{\infty} \frac{f_k(t)}{s k!} \left\|\EE[(X_t-X_0)^{\otimes k}\mid X_0]\right\|_{a^{-1}(X_0)}.
\end{align*}
Then, for any $T > 0$, we have 
\[
 W_{2}(\nu,\mu) \leq \frac{\int_0^T \EE[S(t)^2]^{1/2} dt}{ 1 - c e^{- \kappa T}}.
\]
\end{theorem}
\section{Invariant measures and diffusion approximation}
\label{sec:diff}
Let $(M_n)_{n \in \NN}$ be a Markov chain with 
invariant measure $\nu$ and suppose $M_0$ is drawn from $\nu$.
 For $t> 0, \tau > 0$, we pose
\beq
\label{eq:rwsto}
X_t = M_0 + 1_{t \geq \tau} (M_1 - M_0).
\eeq
Depending on the value of $t$, $X_t$ is either equal to $M_1$ or to $M_0$. As $\nu$ is the invariant measure of the Markov chain $(M_n)_{n \in \NN}$ and since $M_0$ is drawn from $\nu$ then $M_1$ is drawn from $\nu$. Hence, $X_t$ is drawn from $\nu$ for any $t > 0$. We can thus apply Theorem~\ref{thm:main2} with the stochastic process $(X_t)_{t \geq 0}$ to bound $W_2(\nu, \mu)$.
\begin{corollary} 
\label{cor:rw}
Suppose Assumption~\ref{ass:main} is verified. Let $(M_n)_{n \geq 0}$ be a Markov chain with transition kernel $K$ and suppose the invariant measure $\nu$ of the Markov chain and the stochastic process $(X_t)_{t \geq 0}$ with $X_t$ given by Equation~(\ref{eq:rwsto}) satisfy Assumption~\ref{ass:mainnu}. 
Then, denoting $X_0$ by $X$, we have that, for any $T > 0$, there exists $C > 0$ such that, for any $T > \tau > 0, \step > 0$,
\begin{align*}
(1&-c e^{-\kappa T} )  W_{2,a}(\nu, \mu) \leq \\
&  C \left( \tau \EE[\|b(X)\|^2_{a^{-1}(X)}]^{1/2} + \EE\left[\left\| \frac{1}{\step} \int_{y \in \RR^d} (y-X) K(X,dy)  - b(X) \right\|^2_{a^{-1} (X)}\right]^{1/2}\right) \\
 & + C^2 \sqrt{d} \left(\sqrt{d \tau} +  \EE\left[\left\| \frac{1}{\step} \int_{y \in \RR^d} \frac{(y-X)^{\otimes2}}{2} K(X,dy)  - a(X) \right\|^2_{a^{-1}(X)}\right]^{1/2} \right) \\
 & + C^3 \frac{\log(\tau) d}{\step} \EE\left[\left\|  \int_{y \in \RR^d} (y-X)^{\otimes 3} K(X,dy)\right\|^2_{a^{-1}(X)}\right]^{1/2} \\
 & + \sum_{k=4}^\infty C^k \frac{(d (k-1))^{(k-1)/2}}{(k-1) k! \tau^{(k-3)/2}  \step} \EE\left[\left\|  \int_{y \in \RR^d} (y-X)^{\otimes k} K(X,dy)\right\|^2_{a^{-1}(X)}\right]^{1/2}. 
\end{align*}
\end{corollary}
Let us remark that the quantities appearing in our bound are natural as they appear in standard diffusion approximation results, see e.g. Section 11.2 \citep{SV}. 
Moreover, let us note that the pair $(M_1,M_0)$ is an exchangeable pair if and only if the Markov chain $(M_n)_{n \in \NN}$ is reversible. Thus, in order to use the exchangeable pairs technique, one would need to assume $(M_n)_{n \in \NN}$ to be reversible. In the following, we consider two applications in which we use our result to study the convergence of invariant measures of non-reversible Markov chains. 

\subsection{Invariant measure of random walks on nearest neighbours graphs}
\label{subsec:graph}
Let $X_1, \dots, X_n$ be independent and identically distributed  random variables on $\RR^d$, drawn from a measure $\mu$ with density $f \in C^\infty(\RR^d, \RR)$. 
Let $\mathcal{X}_n$ be the set of points $(X_1,\dots,X_n)$ and let $\radius_{\mathcal{X}_n}$ be a function from $\RR^d$ to $\RR^+$.
We call random neighbourhood graph a graph $G_n$ with vertices $\mathcal{X}_n$ and edges $\{(x,y) \in \mathcal{X}^2 \mid \|x-y\|^2 \leq \radius_{\mathcal{X}_n}(x)\}$. 
These graphs are at the center of many data analysis algorithms 
\citep{eigenmap,labelprop,spectral}. 
However, such algorithms usually rely on properties of a random neighbourhood graph and discard all other information regarding the data. However, one does not know whether all the relevant statistical information regarding the data $\mathcal{X}_n$ is contained in a given random neighbourhood graph. To answer this question, \citep{roadmap} proposed to check whether it is possible to estimate the density $f$ from which the data is drawn using only the structure of a random neighbourhood graph. If one can compute a good estimator of $f$ from a graph $G_n$, then we can expect this graph to contain most of the relevant information regarding the original data $\mathcal{X}_n$. 

As $n$ grows to infinity, it has been shown by \citep{Ting2011} that, if $\radius_{\mathcal{X}_n}$ converges, after a proper rescaling, to a deterministic function $\tilde{\radius} : \RR^d \rightarrow \RR^+$, 
then random walks on the random neighbourhood graphs $G_n$ with radii $\radius_{\mathcal{X}_n}$ converge to a  
diffusion process with infinitesimal generator 
\[
\mathcal{L}_{\tilde{\mu}} = \tilde{\radius}^2 \left( \nabla \log f . \nabla + \frac{1}{2} \Delta \right).
\]
Since the invariant measure $\tilde{\mu}$ of the limiting diffusion process has a density proportional to $\frac{f^2}{\tilde{\radius}^2}$, 
it is possible to derive an estimator of $f$ from an estimator of $\tilde{\mu}$. As random walks on the graphs $G_n$ converge to diffusion processes with invariant measure $\tilde{\mu}$, it seems natural to use the invariant measures of random walks on the graphs $G_n$ to estimate $\tilde{\mu}$. In fact, \citep{Hashimoto} proved that, under technical assumption on $\radius_{\mathcal{X}_n}$, invariant measures of these random walks converge weakly to $\tilde{\mu}$.  
Let us show how our results can be used to quantify this convergence by tackling the specific case of nearest neighbours graphs, which are quite popular in data analysis thanks to their sparsity. 

Nearest neighbours graphs are obtained by picking an integer $k > 0$ and putting an edge between two points $X_i$ and $X_j$ if and only if $X_j$ is one of the $k$-nearest neighbours of $X_i$. Equivalently, a $k$-nearest neighbours graph corresponds to a random neighbourhood graph with radius function 
\[
\radius_{\mathcal{X}_n}(x) = \inf \left\{r \in \RR^+ | \sum_{i =1}^n 1_{\|X_i - x\| \leq r} \geq k \right\}.
\]
If $k$ is correctly chosen, random walks on such graphs are approximation of a diffusion process with infinitesimal generator 
\[
\mathcal{L}_{\tilde{\mu}} = f^{-2/d} (\nabla \log f . \nabla +  \frac{1}{2}\Delta).
\]
This diffusion process admits an invariant measure $\tilde{\mu}$ with a density proportional to $f^{2+2/d}$ and can thus be used to compute an estimator of $f$. 
To avoid boundary issues, let us assume that $\mu$ is supported on the flat torus $\mathcal{T} = (\RR / \ZZ)^d$ with strictly positive density $f \in C^\infty(\mathcal{T}, \RR^+)$. 
For any integer $k \leq n$, we denote by $\pi_{k,n}$ an invariant measure of a random walk 
on the $k$-nearest neighbour graphs with vertices $\mathcal{X}_n$. In Section~\ref{subsec:knn} we prove the following result. 

\begin{proposition}
\label{pro:randomgraph}
There exists $C > 0$ such that, for any positive integers $k,n$, 
\[
\PP\left(W_2(\pi_{k,n}, \tilde{\mu}) \leq C\left(\frac{\sqrt{\log n} n^{1/d}}{k^{1/2 + 1/d}} + \left(\frac{k}{n}\right)^{1/d} \right) \right) \geq 1 - \frac{C}{n}.
\] 
\end{proposition} 
In particular, if $n >> k >> (\log n)^{d/(2+d)}  n^{2/(2+d)}$ then $W_2(\pi_{k,n}, \tilde{\mu})$ converges, in terms of Wasserstein distance of order $2$, to $\tilde{\mu}$. 
However, a couple issues still remain. First, we quantify the convergence of $\pi_{k,n}$ in the space of measures which is not sufficient to obtain its point-wise convergence which would be necessary to evaluate the performance of an estimator of $f$ computed from $\pi_{k,n}$.
Furthermore, our bound is likely to be suboptimal. In our bound, $k$ should be at least of order $\log(n)^{d/(d+2)}n^{2/(2+d)}$ for $W_2(\pi_{k,n}, \tilde{\mu})$ to go to zero. Yet, such a result is 
 counterintuitive as the requirements on $k$ get weaker as the dimension of the data increases even though we would expect the task of estimating $\tilde{\mu}$ to be more complex in higher dimensions. In fact, it is conjectured in \citep{Hashimoto} that it is sufficient for $n >> k >> \log(n)$ for $\pi_{k,n}$ to converge to $\tilde{\mu}$. 

\subsection{Analysis of a one-dimensional scheme for the Langevin Monte-Carlo algorithm}
\label{subsec:LMC}
In Bayesian statistics, one often needs to sample points from a probability measure $\stationary$ on $\RR^\Dim$ with density $f \in C^\infty(\RR^d, \RR)$. 
To solve this task, multiple sampling algorithms based on the Monte-Carlo approach were proposed and analyzed. 
We want to show how our results can be used to study the complexity of a simple Monte-Carlo algorithm. 

Remark that the measure $\stationary$ is a reversible measure for the diffusion process $Y_t$ with infinitesimal generator 
\[
\mathcal{L}_\stationary = - \nabla u . \nabla + \Delta,
\]
where $u = - \log f$. Thus, under mild assumptions on $\stationary$, the measure of $(Y_t)_{t \geq 0}$ converges to $\stationary$ as $t$ goes to infinity. Hence, when $t$ is sufficiently large, the measure of $Y_t$ should be close to the measure of $\mu$, in which case it is possible to obtain samples from a measure close to $\mu$ by sampling the measure of $Y_t$. As it is not possible to have access to the measure of the continuous process $(Y_t)_{t \geq 0}$ in practice, one needs to rely on a discrete approximation of $(Y_t)_{t \geq 0}$. For instance, one can use the 
Euler-Maruyama approximation scheme with timestep $\timestep > 0$ given by a Markov chain $(M_n)_{n \in \NN}$ with $M_0 = 0$ and transitions
\[
M_{n+1} = M_n  - \timestep \nabla u(M_n) + \sqrt{2 \timestep} \mathcal{N}_{n},
\]
where $\mathcal{N}_{1},\dots, \mathcal{N}_{n}$ is a sequence of independent normal random variables with mean $0$ and covariance matrix $I_\Dim$. 
If the timestep $\timestep$ is small enough, one can expect the invariant measure $\pi$ of $(M_n)_{n \in \NN}$ to be close to $\mu$ and, 
for $n$ large enough, the measure of $M_n$ should be close 
to $\pi$ and thus close to $\stationary$. 
This approach to sampling, introduced in \citep{LMC}, is known as the Langevin Monte-Carlo (LMC) algorithm. 

One may then wonder how large $n$ should be 
for a given metric between $\mu$ and $\nu_n$, the measure of $M_n$, to be smaller than a given accuracy threshold $\epsilon > 0$. Answering this question is linked to the choice of the timestep $\timestep$ as this parameter must satisfy some trade-off: large values of $\timestep$ lead to a poor approximation of $\mu$ by $\pi$ 
but the smaller $\timestep$ is, the larger the number of iterations required for $\nu^n$ to be close to $\pi$.
Recently, \citep{DurmusMoulinesW} proved that, if $-u$ is a strictly convex function (i.e. $\stationary$ is a strictly log-concave measure) and $\nabla u$ is Lipschitz continuous, 
then the LMC algorithm can reach an accuracy $\acc$ for the Wasserstein distance of order $2$ 
in no more than $O(\acc^{-2} \Dim \log(d/\acc))$ steps.
Since the complexity of each step of the Euler-Maruyama discretization is of order $d$, the overall complexity of the algorithm is bounded by $O(\acc^{-2} d^{2} \log(d/\acc))$. 
This rate can be improved whenever $\nabla^2 u$ is Lipschitz continuous, in which case one only requires $O(\acc^{-1} \sqrt{\Dim} \log(d/\acc))$ steps to reach an accuracy $\acc$, meaning the complexity of the algorithm is bounded by $O(\acc^{-1} d^{3/2} \log(d/\acc))$.
One may wonder whether other discretization schemes could outperform the Euler-Maruyama scheme. However, the approach used to obtain the previous bounds are specific to the Euler-Maruyama scheme and cannot be used to evaluate the performance of another discretization scheme. Let us see how this can be done using our results. 
For instance, 
let $e_1,\dots, e_d \in \RR^d$ be the canonical basis of $\RR^d$, $(I_n)_{n \in \NN}$ be independent uniform random variables on $\{1,\dots,d\}$,  $(B_n)_{n \in \NN}$ be independent Rademacher random variables and let us consider the discretization scheme $(M_n)_{n \in \NN}$ with transitions 
\beq
\label{eq:schema}
M_{n+1} = M_n + \left(- \timestep \frac{\partial u}{\partial x_{I_n}} (M_n) + \sqrt{2 \timestep} B_n\right) e_{I_n}.
\eeq
Following the computations presented in Section~\ref{subsec:LMCproof}, we obtain the following result. 
\begin{proposition}
\label{pro:LMC}
Let $\stationary$ be a measure of $\RR^d$ with density $f \in C^\infty(\RR^d, \RR)$ and let $u = - \log f$. 
Suppose $\nabla u(0) = 0$ and assume there exists $\rho > 0, \Lip > 0$ such that, for all $i \in \{1,\dots,d\}$ and $x,y \in \RR^d$, 
\beq
\label{eq:assu1}
\left(\frac{\partial u}{\partial x_i}(y) - \frac{\partial u}{\partial x_i}(x)\right) (y_i - x_i)  \leq - \rho (y_i - x_i)^2
\eeq
 and 
\beq
\label{eq:assu2}
\left| \frac{\partial u}{\partial x_i}(y) - \frac{\partial u}{\partial x_i}(x)\right| \leq \Lip \left| y_i - x_i \right|.
\eeq
Let $\timestep > 0$ and let $(M_n)_{n \geq 0}$ be the Markov chain with $M_0 = 0$ and increments given by Equation~(\ref{eq:schema}). 
Then, there exist constants $C_1, C_2 > 0$ depending on $\rho$ and $\Lip$  such that, for any $\epsilon > 0$, if $\timestep = C_1 \epsilon^2 \Dim^{-3}$ and $n = C_2 \timestep^{-1} \Dim \log(\Dim / \epsilon)$, then 
the measure $\nu^n$ of $M_n$ satisfies 
\[
W_2(\nu^n, \mu) \leq \epsilon.
\]
Moreover, if $\stationary$ is the Gaussian measure then the previous result holds true with $\timestep = C_1 \epsilon^2 \Dim^{-1}$. 
\end{proposition}
Since each step of this one-dimensional discretization has a complexity independent of the dimension, the overall complexity of the LMC algorithm with the discretization scheme given by Equation~(\ref{eq:schema}) is bounded by 
$O(\acc^{-2}\Dim^{4} \log(\Dim/\acc))$ and by $O(\acc^{-2}\Dim^2 \log(\Dim/\acc))$ when $\stationary$ is the Gaussian measure. The discrepancy between the Gaussian case and the general case is due to the 
dependency on the dimension of the function $f_k$ defined in Proposition~\ref{pro:curvdim} which we believe to be suboptimal, see Remark~\ref{rem:suboptimal}. Hence we conjecture the correct complexity of the LMC algorithm using the discretization scheme given by Equation~(\ref{eq:schema}) to be bounded by $O(\acc^{-2}\Dim^2 \log(\Dim/\acc))$ for target measures $\mu$ satisfying the assumptions of Proposition~\ref{pro:LMC}. Under this conjecture, this one-dimensional discretization scheme is still outperformed by the Euler-Maruyama scheme which reaches the complexity  $O(\acc^{-1} d^{3/2} \log(d/\acc))$ whenever $\nabla^2 u$ is Lipschitz continuous. As this slightly stronger assumption should be verified in most practical cases, the Euler-Maruyama scheme should be more efficient than our one-dimensional scheme in practice.

%% file: appendix.tex
\section{Bounds for the Wasserstein distance of order $2$: proofs of Theorems~\ref{thm:mainGauss} and \ref{thm:main2}}
\label{sec:approach}
Let us assume Assumptions~\ref{ass:main} and \ref{ass:mainnu} are verified or, if the target measure is the Gaussian measure, that the assumptions of Theorem~\ref{thm:mainGauss} are verified. 
In this Section, we assume that the measure $\nu$ admits a density $h$ with respect to $\mu$ such that $h = \epsilon + f$
 for some $\epsilon > 0$ and $f \in C^\infty_c(E,\RR^+)$. The general results are then obtained thanks to an approximation argument detailed in Section~\ref{sec:approx}. 

For any $t > 0$, we denote by $\nu_t$ the measure with density $P_t h$. 
Since $\nu_t$ converges exponentially fast to $\mu$,  
it is sufficient to control the Wasserstein distance between $\nu_t$ and $\nu$ for a large enough $T > 0$ in order to bound the Wasserstein distance between $\nu$ and $\mu$. 
Such a control can be derived from the following estimate, obtained in \cite{WangOV} Section 3, 
\beq
\label{eq:OV}
\frac{d^+}{dt} W_{2,a}(\nu, \nu_t) \leq \left(\int_E \frac{\|\nabla P_t h(x)\|_{a(x)}^2}{P_t h(x)} d \mu(x) \right)^{1/2}.
\eeq
The quantity in the right-hand side of the inequality is actually the square root of the Fisher information of the measure $\nu_t$ with respect to $\mu$, which we denote by $I_\mu(\nu_t)$. For $t > 0$, we pose $v_t = \log P_t h$. Let us express $I_\mu(\nu_t)$ for $t > 0$ using the stochastic process $(X_t)_{t \geq 0}$. 
 
\begin{proposition} 
\label{pro:Ibound}
For any $t > 0$,
\begin{align*}
I_\mu(\nu_t) & = \EE\bigg[\left<\EE\left[\frac{X_t-X_0}{s}  - b(X_0)\mid X_0 \right], \nabla P_t v_t(X_0)\right> \\
& \qquad + \left<\EE\left[\frac{(X_t-X_0)^{\otimes 2}}{2s}  - a(X_0) \mid X_0 \right], \nabla^2 P_t v_t(X_0)\right>\\
& \qquad + \sum_{k=3}^{\infty}  \left<\EE\left[\frac{(X_t-X_0)^{\otimes k}}{sk!}\mid X_0 \right], \nabla^k P_t v_t(X_0) \right> \bigg].
\end{align*}
\end{proposition}

The last step of the proof of Theorems~\ref{thm:mainGauss} and \ref{thm:main2} consists in exploiting the regularizing properties of the semigroup $(P_t)_{t \geq 0}$ in order to bound the right-hand term of the bound in Proposition~\ref{pro:Ibound} by a quantity involving the moments of $X_t - X_0$ and $\EE[P_t \|\nabla v_t\|^2_a(X_0)]^{1/2} = I_\mu(\nu_t)^{1/2}$. From here, we are able to obtain a bound on $I_\mu(\nu_t)^{1/2}$ which can then be turned into a bound on $W_{2,a}(\mu,\nu)$ thanks to Equation~(\ref{eq:OV}). 

\begin{proof}[Proof of Proposition~\ref{pro:Ibound}]
Let $t > 0$. We have
\begin{align*}
\mathcal{L}_\mu v_t & = b.\nabla v_t + <a, \nabla^2 v_t> \\
& = \frac{1}{P_t h} \left(b . \nabla P_t h + \left<a, \nabla^2 P_t h - \frac{(\nabla P_t h)^{\otimes 2}}{P_t h} \right> \right) \\
& = \frac{1}{P_t h} \left(\mathcal{L}_\mu P_t h - \frac{\|\nabla P_t h\|_{a}^2}{P_t h} \right).
\end{align*}
Therefore, 
\[
I_\mu(\nu_t) = \int_E \frac{\|\nabla P_t h(x)\|_{a(x)}^2}{P_t h(x)} d \mu(x) = \int_{E} \left( \mathcal{L}_\mu P_t h (x) - P_t h (x)\mathcal{L}_\mu v_t (x) \right) d \mu(x). 
\]
Since $h$ is bounded from above and form below, so are $P_t h$ and $v_t = \log(P_t h)$. Thus, by Proposition~\ref{pro:curvdim}, there exists $C > 0$ such that $|\mathcal{L}_\mu P_t h| \leq C (\|b\|_{a^{-1}} + 1)$. Since $\|b\|_{a^{-1}} \in L_1(\mu)$, 
$\mathcal{L}_\mu P_t h \in L_1(\mu)$. Moreover, 
since $\mu$ is an invariant measure of the operator $\mathcal{L}_\mu$, we have 
\[
\int_{E} \mathcal{L}_\mu P_t h (x) d \mu(x) = 0.
\]
Therefore, 
\[
I_\mu(\nu_t) = -\int_{E} P_t h(x) \mathcal{L}_\mu v_t (x) d \mu(x).
\]
Then, by the symmetry of $(P_t)_{t \geq 0}$ with respect to the measure $\mu$, 
\[
I_\mu(\nu_t)= - \int_{E} h(x) P_t \mathcal{L}_\mu v_t (x) d \mu(x) = - \int_{E} P_t \mathcal{L}_\mu v_t (x) d \nu(x).
\]
Finally, since $\mathcal{L}_\mu$ is the infinitesimal generator of the semigroup $(P_t)_{t \geq 0}$, we can permute $P_t$ and $\mathcal{L}_\mu$ to obtain 
\beq
\label{eq:Inu}
I_\mu(\nu_t) = -\int_{E}  \mathcal{L}_\mu P_t v_t (x) d \nu(x).
\eeq

Taking $\step > 0$, we define the operator $\mathcal{L}_\nu$ such that, for any bounded and measurable function $\test$, 
\[
\forall x \in E, \mathcal{L}_\nu \test (x) = \frac{1}{\step} \EE\left[ \test(X_t) - \test(X_0)  |X_0 = x\right].
\]
Since $X_t$ and $X_0$ are drawn from the same law, integrating this operator with respect to $\nu$ gives 
\[
\int_E \mathcal{L}_\nu \test (x) d\nu(x) = \frac{1}{\step} \EE[\test(X_t) - \test(X_0)] = 0.  
\]
Hence, 
we can inject $\int_E \mathcal{L}_\nu P_t v_t (x) d\nu(x)$ in Equation~(\ref{eq:Inu}) to obtain 
\beq
\label{eq:auxmain}
I_\mu(\nu_t) = \int_{E}  (\mathcal{L}_\nu - \mathcal{L}_\mu) P_t v_t (x) d \nu(x)
\eeq
To conclude the proof of the Proposition we need to rewrite $\mathcal{L}_\nu$ in order to be able to compare it to $\mathcal{L}_\mu$. This is done in the following result, proved in Section~\ref{sec:analproof}, by showing that $P_t v_t$ is real analytic on $E$ and by using a Taylor expansion which concludes the proof of Proposition~\ref{pro:Ibound}. 
\begin{lemma}
\label{lem:analytic}
Let $t > 0$ and let $\test \in C^\infty(\RR^d, \RR)$ such that $\|\nabla \test\|_a < M$ for some $M > 0$. We have 
\[
\int_E \mathcal{L}_\nu P_t \test (x) d\nu(x) = \EE\left[\sum_{k=1}^\infty \frac{1}{s k!} \left<\EE[(X_t-X_0)^{\otimes k} \mid X_0], \nabla^k P_t \test(X_0)\right>\right]. 
\]
\end{lemma}
\end{proof}

\subsection{Gaussian case: proof of Theorem~\ref{thm:mainGauss}}
\label{subsec:gauss}
Let $\mu$ be the $d$-dimensional Gaussian measure $\gamma$ where $d\gamma(x)~=~(2\pi)^{-d/2} e^{-\frac{\|x\|^2}{2}} dx$ for any $x \in \RR^d$. The measure $\gamma$ is the reversible measure of the operator $\mathcal{L}_\gamma$ where 
\[
\forall \test \in C_c^\infty(\RR^d, \RR), x \in \RR^d, \mathcal{L}_\gamma \test (x) = -x . \nabla \test(x) + <I_d, \nabla^2 \test(x)>
\] 
whose associated semigroup $(P_t)_{t \geq 0}$ is the Ornstein-Uhlenbeck semigroup.

For any $k \in \NN$ and any $i \in \{1,\dots,d\}^k$, let $H_i \in C^\infty(\RR^d, \RR)$ be the multivariate Hermite polynomial of index $i$, defined for any $x \in \RR^d$ by 
\beq
\label{eq:hermite}
H_i(x) = (-1)^k e^{\frac{\|x\|^2}{2}} \frac{\partial^k}{\partial x_{i_1} \dots \partial x_{i_k}} e^{-\frac{\|x\|^2}{2}}.
\eeq
We have the following result.
\begin{lemma}
\label{lem:ippgauss}
For any bounded function $\test \in C^\infty(\RR^d, \RR)$, any $k \in \NN$ and any $i \in \{1,\dots,d\}^k$, we have 
\[
\frac{\partial^k P_t \test}{\partial x_{i_1} \dots \partial x_{i_k}}(x) = \frac{e^{-t}}{(e^{2t}-1)^{k/2}}  \int_{\RR^d} H_i (y) \test(xe^{-t} + \sqrt{1-e^{-2t}}y) d\gamma(y).
\]
\end{lemma}

Thus, 
since Hermite polynomials form an orthogonal basis of $L^2(\gamma)$ with norms
\[
\forall i \in \{1,\dots,d\}^k, \|H_i\|^2_{\gamma} = \int_{\RR^d} H_i^2 (y) d\gamma(y) = \prod_{j=1}^d \left(\sum_{l=1}^k \delta_{i_l,j}\right)!,
\]
applying the previous Lemma to the vector field $v_t$ yields, for any $x \in \RR^d$,
\begin{align*}
\sum_{j \in \{1,\dots,d\}} & \sum_{k=0}^\infty   \sum_{i\in \{1,\dots,d\}^{k}} \frac{ (e^{2t} - 1)^{k}}{e^{-2t} \|H_i\|^2_\gamma}   (\nabla^{k+1} P_t v_t)_{j,i_1,\dots,i_{k}}^2 (x)\\ 
& =  \sum_{j \in \{1,\dots,d\} } \sum_{k=0}^\infty \sum_{i\in\{1,\dots,d\}^{k}} \left( \int_{\RR^d} \frac{H_i(y)}{\|H_i\|_\gamma} (\nabla v_t (xe^{-t} + \sqrt{1-e^{-2t}}y))_j d\gamma(y) \right)^2 \\
& = \sum_{j \in \{1,\dots,d\}} \int_{\RR^d} (\nabla v_t (xe^{-t} + \sqrt{1-e^{-2t}}y))_j^2 d\gamma(y) \\
& =  \int_{\RR^d} \|\nabla v_t (xe^{-t} + \sqrt{1-e^{-2t}}y)\|^2 d\gamma(y).
\end{align*}
Then, since $\nabla^k P_t v_t = e^{-t} \nabla^{k-1} P_t \nabla v_t$, we have  
\beq
\label{eq:sommeipp}
\sum_{\substack{k \geq 0 \\ j \in \{1,\dots,d\} \\ i\in \{1,\dots,d\}^{k}}} \frac{(e^{2t} - 1)^{k}}{e^{-2t} \|H_i\|^2_\gamma}   (\nabla^{k+1} P_t v_t)_{j,i_1,\dots,i_{k}}^2 (x)= P_t \|\nabla v_t\|^2 (x).
\eeq
Let us pose 
\begin{align*}
S(t) = &  \left\|\EE\left[\frac{X_t-X_0}{s} + X_0\mid X_0\right] \right\|^2  \\
&+ \frac{1}{e^{2t} - 1} \left\|\EE\left[\frac{(X_t-X_0)^{\otimes 2} }{2s} - I_d\mid X_0\right] \right\|^2 \\
& + \sum_{k=3}^{\infty} \frac{1 }{ (s k!)^2 (e^{2t} - 1)^{k-1} } \|\EE[(X_t-X_0)^{\otimes k}\mid X_0]\|_H^2.
\end{align*}
Applying Proposition~\ref{pro:Ibound} and using Cauchy-Schwarz inequality, we obtain 
\[
I_\gamma(\nu_t) \leq e^{-t} \EE[S(t)]^{1/2} \left(\EE\left[\sum_{\substack{k \geq 0 \\ j \in \{1,\dots,d\} \\ i\in \{1,\dots,d\}^{k}}} \frac{e^{2t}(e^{2t} - 1)^{k}}{\|H_i\|^2_\gamma} (\nabla^{k+1} P_t v_t (X_0))_{j,i_1,\dots,i_{k}}^2 \right]\right)^{1/2}.
\]
Thus, by Equation~(\ref{eq:sommeipp}) and since $v_t = \log(P_t h)$, we have
\begin{align*}
I_\gamma(\nu_t) & \leq e^{-t} \EE[S(t)]^{1/2} \EE[P_t \|\nabla v_t(X_0)\|^2]^{1/2}  \\
& = e^{-t}\EE[S(t)]^{1/2}\left(\int_{\RR^d} \frac{\|\nabla P_t h\|^2}{(P_t h)^2}  \, d\nu_t\right)^{1/2} \\
& = e^{-t}\EE[S(t)]^{1/2}\left(\int_{\RR^d} \frac{\|\nabla P_t h\|^2}{P_t h}  \, d\mu\right)^{1/2} \\
& = e^{-t}\EE[S(t)]^{1/2}I_\gamma(\nu_t)^{1/2}.
\end{align*}
Finally, since $I_\gamma(\nu_t)$ is finite, we end up with 
\[
I_\gamma(\nu_t)^{1/2} \leq e^{-t}\EE[S(t)]^{1/2}
\]
and, by Equation~(\ref{eq:OV}), 
\[
W_2(\nu,\gamma) \leq \int_0^\infty I_\gamma(\nu_t)^{1/2} dt \leq \int_0^\infty e^{-t}\EE[S(t)]^{1/2} dt,
\]
concluding the proof of Theorem~\ref{thm:mainGauss}.

\subsection{General case: proof of Theorem~\ref{thm:main2}} 
Let us first apply Proposition~\ref{pro:Ibound} and use Cauchy-Schwarz inequality in order to obtain
\begin{align*}
I_\mu(\nu_t) \leq & \EE\left[\left\|\EE\left[\frac{X_t-X_0}{s} - b(X_0)\mid X_0 \right] \right\|_{a^{-1}(X_0)} \|\nabla P_t v_t(X_0)\|_{a(X_0)}\right]  \\
& + \EE\left[\left\|\EE\left[\frac{(X_t-X_0)^{\otimes 2}}{2s}  - a(X_0) \mid X_0\right]\right\|_{a^{-1}(X_0)} \|\nabla^2 P_t v_t(X_0)\|_{a(X_0)} \right]  \numberthis \label{eq:I}\\
& + \sum_{k=3}^{\infty}  \EE\left[\left\|\EE\left[\frac{(X_t-X_0)^{\otimes k}}{sk!} \mid X_0\right]\right\|_{a^{-1}(X_0)} \|\nabla^k P_t v_t(X_0)\|_{a(X_0)}\right] .
\end{align*}
Our objective is to bound $\|\nabla^k P_t v_t\|_{a}$ by a quantity involving $P_t \|\nabla v_t\|_a$. Since we assumed there exists $\rho \in \RR$ such that 
\beq
\label{eq:CDcondition}
\forall \test \in C^\infty_c(E, \RR), \Gamma_2(\test,\test) \geq \rho \Gamma_1(\test,\test),
\eeq
 we know (see e.g. Theorem 3.2.3 \citep{Markov}) that 
$(P_t)_{t \geq 0}$ verifies the following gradient bound 
\beq
\label{eq:gradbound}
\forall \test \in C^\infty_c(\RR^d), \|\nabla P_t \test\|_a \leq e^{- \rho t} P_t ( \|\nabla \test\|_a),
\eeq
allowing us to bound the first term in Equation~(\ref{eq:I}). 
The proof of Theorem 4.1 of \citep{Stein} makes use of the Gamma operators $(\Gamma_k)_{k \geq 1}$ defined recursively from the $\Gamma_1$ operator by 
\begin{multline}
\label{eq:gammaoperators}
\forall k > 1, f,g \in C^\infty(E, \RR), \Gamma_{k}(f,g) = \\
\frac{1}{2} \left[ \mathcal{L}_\mu (\Gamma_{k-1}(f,g)) - \Gamma_{k-1}(\mathcal{L}_\mu f, g) - \Gamma_{k-1}( f, \mathcal{L}_\mu g) \right]
\end{multline}
to show that, if there exist $\kappa, \sigma > 0$ such that, for any $\test \in C^\infty_c(E, R)$, $\Gamma_3 (\test, \test)\geq \kappa \Gamma_2(\test, \test)$ and $\Gamma_2 (\test, \test) \geq \sigma \|\nabla^2 \test\|_a$, then 
\beq
\label{eq:di}
\forall \test \in C^\infty_c(\RR^d), \|\nabla^2 P_t \test\|_a^2 \leq \frac{\kappa}{\sigma(e^{\kappa t}-1)}  P_t \|\nabla \test\|^2_a.
\eeq
However, such assumptions are usually hard to check in practice. Let us consider a simple one-dimensional example for which $\mathcal{L}_\mu \test = -u' \test' + \test''$. 
In this case, Equation~(\ref{eq:CDcondition}) is verified as long as $u'' \geq \rho$. On the other hand, following the computations of Section 4.4 \citep{Stein}, in order to have $\Gamma_3(\test,\test) \geq 3 c \Gamma_2(\test,\test)$ and $\Gamma_2(\test,\test) \geq c \|\test''\|_a$ for some $c > 0$, one requires
\[
 u^{(4)} - u' u^{(3)} + 2 (u'')^2 - 6 c u'' \geq 0
 \]
 and 
 \[
 3 (u^{(3)})^2 \leq 2(u'' -c ) (u^{(4)} - u' u^{(3)} + 2 (u'')^2 - 6 c u'' ).
 \]
Even in this rather simple case, such conditions can be difficult to check in practice and obtaining on $\|\nabla^k P_t \test\|_a$ for $k > 2$ in a similar manner would require even stronger assumptions. 
Instead, we rely on the following result, proved in Section~\ref{sec:curvdim}, which provides bounds on $\|\nabla^k P_t \test\|_a$ under a simple $CD(\rho, \infty)$ condition. 
\begin{proposition}
\label{pro:curvdim} 
Under Assumption~\ref{ass:main}, if $\test \in C^\infty(E, \RR)$ is a bounded function then 
\[
\forall k > 0, t > 0, \|\nabla^k P_t \test\|_a \leq f_k(t) \sqrt{P_t \|\nabla \test\|^2_a}, 
\]
where 
\[
f_k(t) = \begin{cases} 
 e^{- \rho t \max(1,k/2)} \left( \frac{2\rho d}{e^{2\rho t/(k-1)}-1} \right)^{(k-1)/2}  \textit{ if } \rho \neq 0 \\
\left(\frac{d (k-1)}{t}\right)^{(k-1)/2} \textit{ if } \rho = 0 
\end{cases}.
\]
\end{proposition}

\begin{remark}
\label{rem:suboptimal}
The bounds we obtain are not dimension-independent as one could expect from the Gaussian case or from Equation~(\ref{eq:di}). We believe this dependency to be an artifact of the proof.
\end{remark}

Injecting this Proposition in Equation~(\ref{eq:I}) and using Cauchy-Schwarz inequality and the triangle inequality, we obtain that 
\begin{align*}
I_\mu(\nu_t) & \leq  \left(\EE[S(t)^2] \EE[P_t \|\nabla v_t\|_a^2  (X_0)]\right)^{1/2} \\
& \leq  \left(\EE[S(t)^2]\EE[\|\nabla v_t(X_t)\|_{a(X_t)}^2] \right)^{1/2} \\
& \leq \left(\EE[S(t)^2] I_\mu(\nu_t)\right)^{1/2},
\end{align*}
where 
\begin{align*}
 S(t) = & f_1(t) \left\|\EE\left[\frac{X_t-X_0}{s} - b(X_0) \mid X_0\right] \right\|_{a^{-1}(X_0)}\\
& +  f_2(t) \left\|\EE\left[\frac{(X_t-X_0)^{\otimes 2}}{2s} - a(X_0)\mid X_0\right] \right\|_{a^{-1}(X_0)} \\
& + \sum_{k=3}^{\infty} \frac{f_k(t)}{s k!} \left\|\EE[(X_t-X_0)^{\otimes k}\mid X_0]\right\|_{a^{-1}(X_0)}.
\end{align*} 
Then, by Equation~(\ref{eq:OV}), 
\[
\forall T > 0, W_{2,a}(\nu, \nu_T) \leq \int_0^T e^{-t} \EE[S(t)^2]^{1/2} dt. 
\]
Finally, we can use our assumption on the convergence speed of $\nu_t$ to $\mu$ to obtain that 
\begin{align*}
W_{2,a}(\nu,\mu) & \leq W_{2,a}(\nu, \nu_T) + W_{2,a}(\nu_T, \mu) \\
& \leq  W_{2,a}(\nu, \nu_T) + c e^{-\kappa T} W_{2,a}(\nu, \mu),
\end{align*}
and thus
\[
(1 - c e^{- \kappa T}) W_{2,a}(\nu,\mu) \leq W_{2,a}(\nu, \nu_T) \leq \int_0^T e^{-t} \EE[S(t)^2]^{1/2} dt,
\]
concluding the proof of Theorem~\ref{thm:main2}. 

\section{Gaussian measure and Wasserstein distances of any order: proofs of Theorems~\ref{thm:WpGauss} and \ref{thm:WpGaussexch}}
\label{sec:Wp}
Let $p \geq 1$ and let $\nu$ be a measure on $\RR^d$. 
In this Section, we assume the measure $\nu$ admits a density $h$ with respect to $\gamma$ such that $h = \epsilon +f$ with $\epsilon > 0$ and $f \in C^\infty_c(\RR^d, \RR^+)$. Theorem~\ref{thm:WpGauss} and Theorem~\ref{thm:WpGaussexch} are then obtained through an approximation argument developed in Section~\ref{sec:approxW2}.

In order to bound the $W_p$ distance between $\nu$ and the $d$-dimensional Gaussian measure $\gamma$, it is possible to use Stein kernels to obtain an explicit expression of the score function $v_t = \log(P_t h)$, where $P_t$ is the Ornstein-Uhlenbeck semigroup \citep{Stein}. Then, by \cite{WangOV} Section 3, we have 
\[
\frac{d+}{dt} W_p(\nu, \nu_t) \leq \left( \int_{\RR^d} \|  \nabla v_t \|^p d\nu_t \right)^{1/p},
\]
leading to
\beq
\label{eq:OVp}
W_p(\nu,\gamma) \leq \int_0^\infty \left( \int_{\RR^d} \|  \nabla v_t \|^p d\nu_t \right)^{1/p} dt.
\eeq
In this Section, we use a similar approach without relying on a Stein kernel. Instead, let us consider a stochastic process $(X_t)_{t \geq 0}$ satisfying the assumptions of either Theorem~\ref{thm:WpGauss} or \ref{thm:WpGaussexch}. 
Let $Z$ be a Gaussian random variable independent from the process $(X_t)_{t \geq 0}$ and let $F_t = e^{-t} X_0 + \sqrt{1-e^{-2t}} Z$.
Let us provide 
a version of $\nabla v_t (F_t)$. 
\begin{lemma}
\label{lem:score}
Let $t > 0$. Then, 
\[
\forall x \in \RR^d, \rho_t = \EE\left[e^{-t}X_0 - \frac{e^{-2t}}{\sqrt{1-e^{-2t}}} Z\mid F_t\right]
\] 
is a version of $\nabla v_t (F_t)$. 
\end{lemma}
\begin{proof}
Let $t > 0$.
Integrating by parts with respect to $\gamma$, we have, for any $\test \in \mathcal{C}^\infty_c(\RR^d)$, 
\begin{align*}
\int_{\RR^d} \nabla \test (x) d \nu_t (x) &= \int_{\RR^d} \nabla \test (x) P_t h (x) d \gamma (x)  \\
 & = \int_{\RR^d} \nabla  (\test P_t h) (x) - \test (x) \nabla P_t h (x) d \gamma (x) \\
 & = \int_{\RR^d} x  \test (x) P_t h (x) - \test (x) \nabla P_t h (x) d\gamma (x) \\
 & = \int_{\RR^d} P_t h(x)  \test (x) \left(x - \frac{\nabla P_t h (x)}{P_t h(x)}\right) d\gamma (x) \\
 & = \int_{\RR^d}  \test (x) \left(x - \frac{\nabla P_t h (x)}{P_t h(x)}\right) d\nu_t (x).
\end{align*}
Thus, we have 
\beq
\label{eq:scorecarac}
\forall \test \in \mathcal{C}^\infty_c(\RR^d), \int_{\RR^d} \nabla \test (x) d \nu_t (x) = \int_{\RR^d} \test (x) ( x - \nabla v_t (x)) d \nu_t (x)
\eeq
In fact, this property completely characterizes $\nabla v_t$: if another function $\xi : \RR^d\rightarrow \RR^d$ is such that
\[
\forall \test \in \mathcal{C}^\infty_c(\RR^d),  \int_{\RR^d} \nabla \test (x) d \nu_t (x) =  \int_{\RR^d} \test (x) ( x - \xi (x)) d \nu_t (x),
\]
then  
\[
\forall \test \in \mathcal{C}^\infty_c(\RR^d),  \int_{\RR^d} \test (x) (\nabla v_t (x) - \xi(x)) d\nu_t = 0,
\]
implying that $\xi = \nabla v_t$ almost everywhere with respect to the measure $\nu_t$.  

Now, let $\test \in C^\infty_c(\RR^d, \RR)$. By an integration by parts with respect to the Gaussian measure, we have 
\begin{align*}
\EE[\test(F_t) (F_t - \rho_t)] & = \EE\left[\test(F_t) \left(F_t - e^{-t}X_0 + \frac{e^{-2t}}{\sqrt{1-e^{-2t}}} Z\right)\right] \\
& = \EE\left[\frac{1}{\sqrt{1-e^{-2t}}}\test(F_t) Z\right] \\
& = \EE[\nabla \test(F_t)]
\end{align*}
and thus $\rho_t$ satisfies Equation~(\ref{eq:scorecarac}) which implies it is a version of $\nabla v_t$. 
\end{proof}

Coming back the proofs of Theorem~\ref{thm:mainGauss} and \ref{thm:main2}, we need to estimate the quantity $\EE[\|\rho_t\|^p]$, where $\rho_t$ is defined in the previous Lemma, to be able to bound $W_p(\nu,\gamma)$. 
To do so, suppose there exists a quantity $\tau_t$ defined using $X_0$ and $Z$ and such that $\EE[\tau_t \mid F_t] = 0$ almost surely. Then, by Jensen's inequality, we obtain 
\begin{align*}
\EE[\|\nabla \rho_t\|^p] & = \EE\left[\left\|\EE\left[e^{-t}X_0 - \frac{e^{-2t}}{\sqrt{1-e^{-2t}}} Z\mid F_t \right]\right\|^p\right] \\
& = \EE\left[\left\|\EE\left[\tau + e^{-t}X_0 - \frac{e^{-2t}}{\sqrt{1-e^{-2t}}} Z\mid F_t \right]\right\|^p\right] \\
& \leq \EE\left[\left\|\tau + e^{-t}X_0 - \frac{e^{-2t}}{\sqrt{1-e^{-2t}}} Z\right\|^p\right].
\end{align*}
Thus, if $\tau_t$ is close to $e^{-t}X_0 - \frac{e^{-2t}}{\sqrt{1-e^{-2t}}} Z$ then $\EE[\|\nabla \rho_t\|^p]$ is small and, by Equation~(\ref{eq:OVp}), so is $W_p(\nu, \gamma)$.  

As in the proof of Theorem~\ref{thm:mainGauss}, our computations are going to involve the Hermite polynomials.  
For any $k \in \NN$, let $\mathcal{H}_k : \RR^d \rightarrow (\RR^d)^{\otimes k}$ be the tensor of Hermite polynomials of order $k$ given by 
\[
\forall x \in \RR^d, \forall i \in \{1,\dots, d\}^k, (\mathcal{H}_k(x))_i = H_i(x),
\]
where for any $i \in \{1,\dots, d\}^k$, $H_i$ is the multi-dimensional Hermite polynomial of index $i$ defined in Equation~(\ref{eq:hermite}). 
For any $k \in \NN$, any $M \in (\RR^d)^{\otimes k}$ and any $N \in (\RR^d)^{\otimes k-1}$, let $MN \in \RR^d$ such that 
\[
\forall i \in \{1,\dots,d\}, (MN)_i = \sum_{j \in \{1,\dots,d\}^{k-1}} M_{i, j_1,\dots,j_{k-1}} N_j.
\]
In the following Sections, we rely on the following property verified by Hermite polynomials.
\begin{lemma}
\label{lem:hypercon}
Let $(M_k)_{k \in \NN}$ be such that $\forall k \in \NN, M_k \in (\RR^d)^{\otimes k}$. Then, 
\[
\EE[\|\sum_{k=1}^\infty M_k \mathcal{H}_{k-1}(Z) \|^p]^{2/p} \leq \begin{cases}
\sum_{k=1}^\infty \|M\|^2_H \text{ if $1 \leq p \leq 2$} \\
\sum_{k=1}^\infty (p-1)^{k-1} \|M\|^2_H \text{ if $p > 2$} 
\end{cases}. 
\]
\end{lemma}

\subsection{One-dimensional case: proof of Theorem~\ref{thm:WpGauss}}
\label{sec:WpGauss1}
Let us assume that $d = 1$ and let $(X_t)_{t \geq 0}$ be a stochastic process satisfying the assumptions of Theorem~\ref{thm:WpGauss}. 

\begin{lemma}
\label{lem:tau1}
Let $t,s > 0$.
Taking 
\[
\tau_t = \sum_{k=1}^{\infty} \frac{e^{-kt}}{s k! \sqrt{1-e^{-2t}}^{k-1}} \EE[(X_t-X_0)^k  \mid X_0] \mathcal{H}_{k-1}(Z),
\] 
we have 
\[
\EE[\tau_t \mid F_t] = 0. 
\]
\end{lemma}

\begin{proof}
Let $\test \in \mathcal{C}^\infty_c(\RR, \RR)$. For any $k \in \NN$, we denote by $\test^{(k)}$ the $k$-th derivative of $\test$. 
Let $k \in \NN$. Since $X_0$ and $Z$ are independent, we can use Lemma~\ref{lem:ippgauss} to obtain
\[
\EE[\mathcal{H}_k(Z) \test(F_t)] = \EE[\mathcal{H}_k(Z) \test(e^{-t}X_0 + \sqrt{1 - e^{-2t}}Z)] = (1 - e^{-2t})^{k/2} \EE[\test^{(k)} (F_t)].
\]
Therefore, we have 
\begin{align*}
\EE[\EE[\tau_t \mid F_t] \test(F_t)] & = \EE[\tau_t \test(F_t)]\\
& =  \frac{1}{s} \EE\left[\sum_{k=1}^{\infty} \frac{e^{-kt}}{k!} (X_t-X_0)^k \test^{(k-1)}(F_t)\right].
\end{align*}
Now, let $\Phi$ be a primitive function of $\test$. By Lemma~\ref{lem:analytic}, the function $x \rightarrow \EE[\Phi(F_t)\mid X_0=x] = P_t \Phi (x)$ satisfies 
\begin{align*}
\EE[P_t \Phi(X_t) -P_t \Phi(X_0)]  & =\EE\left[\sum_{k=1}^{\infty} \frac{e^{-kt}}{k!}  (X_t-X_0)^k (P_t \test)^{(k-1)}(X_0)\right]\\
& = \EE\left[\sum_{k=1}^{\infty} \frac{e^{-kt}}{k!} (X_t-X_0)^k \test^{(k-1)}(F_t)\right] \\
& = s \EE[\EE[\tau_t \mid F_t] \test(F_t)].
\end{align*}
Then, since $X_t$ and $X_0$ are both drawn from $\nu$, we have
\[
\EE[\EE[\tau_t \mid F_t] \test(F_t)] = \frac{1}{s}\EE[P_t \Phi(X_t) -P_t \Phi(X_0)]  = 0,
\]
implying that $\EE[\tau_t \mid F_t] = 0$ almost surely. 
\end{proof}
Returning to the proof of Theorem~\ref{thm:WpGauss}, using Lemma~\ref{lem:tau1} along with 
Lemma~\ref{lem:score} and Jensen's inequality, we obtain 
\begin{align*}
\EE[|\rho_t|^p] & = \EE\left[\left|\EE\left[e^{-t} X_0 + \frac{e^{-2t}}{\sqrt{1-e^{-2t}}} Z + \tau_t \mid F_t\right]\right|^p\right] \\
& \leq \EE\left[\left|e^{-t} X_0 + \frac{e^{-2t}}{\sqrt{1-e^{-2t}}} Z + \tau_t\right|^p\right].
\end{align*}
Then, by Lemma~\ref{lem:hypercon}, we have 
\[
\EE[|\rho_t|^p]^{1/p} = \EE_{X_0}[ \EE_Z[|\rho_t|^p] ]^{1/p}\leq e^{-t} \EE[S_p(t)^{p/2}]^{1/p},
\]
where
\begin{align*}
S_p(t) =  & \EE\left[\frac{X_t-X_0}{s} +X_0 \mid X_0\right]^2 \\
& + \frac{ \max(1,p-1)}{ e^{2t} -1}  \EE\left[\frac{(X_t-X_0)^{2}}{2 s}  - 1 \mid X_0 \right]^2 \\
& +  \sum_{k=3}^{\infty} \frac{\max(1,p-1)^{k-1}}{s^2 k k!  (e^{2t}-1)^{k-1} }  \EE\left[(X_t-X_0)^{k}\mid X_0\right]^2.
\end{align*}
Finally, by Equation~(\ref{eq:OVp}), we obtain that 
\[
W_p(\nu,\gamma) \leq \int_0^\infty \EE[|\rho_t|^p]^{1/p} dt \leq \int_0^\infty e^{-t} \EE[S_p(t)^{p/2}]^{1/p} dt,
\]
concluding the proof of Theorem~\ref{thm:WpGauss}.
\subsection{Multi-dimensional case: proof of Theorem~\ref{thm:WpGaussexch}}
\label{sec:WpGaussexch}
Unfortunately, it is not possible to use a multi-dimensional generalization of the random vector $\tau_t$ defined in Lemma~\ref{lem:tau1} as we would only be able to show that,
\[
\forall \test \in C^\infty_c(\RR^d, \RR), \EE[\left<\EE[\tau_t \mid F_t], \nabla \test(F_t)\right>] = 0
\]
which is not sufficient to assert that $\EE[\tau_t \mid F_t] = 0$. Instead, let us assume that the process $(X_t)_{t \geq 0}$ satisfies the assumptions of Theorem~\ref{thm:WpGaussexch}. 

\begin{lemma}
\label{lem:tau2}
Let $s,t > 0$.
The quantity 
\[
\tau_t = \left[\frac{e^{-t}}{s} (X_t-X_0) \left( 1 + \frac{1}{2} \sum_{k=1}^{\infty} \frac{1}{k! (e^{2t}-1)^{k/2}}\left< (X_t-X_0)^{\otimes k}, \mathcal{H}_{k}(Z)\right> \right) \mid X_0, Z \right]
\]
satisfies 
\[
\EE[\tau_t \mid F_t] = 0. 
\]
\end{lemma}

\begin{proof}
For any $\test \in C^\infty_c(\RR^d, \RR)$, we have
\[
\EE[\EE[\tau_t \mid F_t] \test(F_t)] = \EE[\tau_t \test(F_t)].
\]
Hence, by Lemma~\ref{lem:ippgauss}, 
\begin{multline*}
\EE[\EE[\tau_t \mid F_t] \test(F_t)] =\\
\frac{e^{-t}}{s}  \EE\left[(X_t-X_0) \left(\test(F_t) + \frac{1}{2} \sum_{k=1}^{\infty} \frac{e^{-kt}}{k!} \left<(X_t-X_0)^{\otimes k}, \nabla^k \test(F_t) \right> \right)\right].
\end{multline*}
Let $F'_t = e^{-t}X_t + \sqrt{1 - e^{-2t}} Z$. 
According to Lemma~\ref{lem:analytic}, we have 
\begin{align*}
s e^t \EE[\EE[\tau_t \mid F_t] \test(F_t)]  & = \EE\left[(X_t-X_0) \left(\test(F_t) + \frac{\test(F'_t) - \test(F_t)}{2}   \right)\right] \\
& = \frac{1}{2} \EE\left[(X_t-X_0) \left(\test(F_t) + \test(F'_t)   \right)\right].
\end{align*}
Finally, since the pairs $(X_0,X_t)$ and $(X_t,X_0)$ follow the same law, 
\[
\EE\left[ (X_t-X_0)(\test(F_t) + \test(F'_t))\right] = 0
\]
 and thus 
$\EE[\tau_t \mid F_t] = 0$ almost surely.
\end{proof}

Returning to the proof of Theorem~\ref{thm:WpGaussexch}, using the previous result along with 
Lemma~\ref{lem:score} and Jensen's inequality, we obtain 
\[
\EE[\|\rho_t\|^p]^{1/p} \leq \EE\left[\left\|e^{-t} X_0 + \frac{e^{-2t}}{\sqrt{1-e^{-2t}}} Z + \tau_t\right\|^p\right]^{1/p}.
\] 
Then, by Lemma~\ref{lem:hypercon}, 
\[
\EE[\|\rho_t\|^p]^{1/p} \leq e^{-t} \EE[\|S_p(t)\|^{p/2}]^{1/p}, 
\]
where
\begin{align*}
S_p(t)  & = \left\|\EE\left[\frac{X_t-X_0}{s}+X_0 \mid X_0\right]\right\|^2 \\
& +  \frac{\max(1, p-1)}{e^{2t}-1}  \left\|\EE\left[\frac{(X_t-X_0)^{\otimes 2}}{2 s} - I_d \mid X_0 \right]\right\|^2 \\
& +  \sum_{k=3}^{\infty}  \frac{\max(1, p-1)^{k-1}}{4 (s (k-1)!)^2 (e^{2t}-1)^{k-1}}   \left\|\EE[(X_t-X_0)^{\otimes k} \mid X_0] \right\|_H^2.
\end{align*} 
Finally, injecting this bound in Equation~(\ref{eq:OVp}) yields 
\[
W_p(\nu,\gamma) \leq \int_0^\infty e^{-t} \EE[S_p(t)^{p/2}]^{1/p} dt,
\]
concluding the proof.  

\section{Proof of Theorem \ref{thm:CTL}}
\label{subsec:CTLcomp}
Before starting the proof of Theorem~\ref{thm:CTL}, we first need to state a multi-dimensional version of the Rosenthal inequality. 
Let $k > 0, p \geq 2$ and suppose $Y_1,\dots,Y_n$ are independent random variables taking values in $(\RR^d)^{\otimes k}$. Then, by Theorem 2.1 \citep{acosta1981}, there exists $C_p > 0$ such that 
\[
\EE\left[\left|\left\|\sum_{i=1}^n Y_i\right\| - \EE[\|\sum_{i=1}^n Y_i\|]\right|^p\right]^{1/p} \leq C_p \left( \left(\sum_{i=1}^n \EE[\|Y_i\|^2]\right)^{1/2} + \left(\sum_{i=1}^n \EE[\|Y_i\|^p]^{1/p} \right)^{1/p} \right).
\]
Hence, using the triangle inequality and Jensen's inequality, we have 
\begin{align*}
\EE\left[\left\|\sum_{i=1}^n Y_i\right\|^p\right]^{1/p} & \leq \EE\left[\left|\left\|\sum_{i=1}^n Y_i\right\| - \EE[\|\sum_{i=1}^n Y_i\|]\right|^p\right]^{1/p} + \EE[\|\sum_{i=1}^n Y_i\|] \\
&\leq C_p \left(\EE[\|\sum_{i=1}^n Y_i\|] + \left(\sum_{i=1}^n \EE[\|Y_i\|^2]\right)^{1/2} + \left(\sum_{i=1}^n \EE[\|Y_i\|^p]^{1/p} \right)^{1/p} \right).
\end{align*}
Using Jensen's inequality, we have $\EE[\|\sum_{i=1}^n Y_i\|] \leq \EE[\|\sum_{i=1}^n Y_i\|^2]^{1/2}$. Furthermore, since $Y_1,\dots,Y_n$ are independent, we have
\begin{align*}
\EE[\|\sum_{i=1}^n Y_i\|] & \leq \EE[\|\sum_{i=1}^n Y_i\|^2]^{1/2} \\
& \leq \EE[\sum_{1 \leq i,j \leq n} <Y_i, Y_j>]^{1/2} \\
& \leq \left(\sum_{1 \leq i,j \leq n} \EE[<Y_i, Y_j>]\right)^{1/2} \\
& \leq \left(\sum_{1 \leq i,j \leq n, i\neq j} <\EE[Y_i], \EE[Y_j]> + \sum_{i=1}^n \EE[\|Y_i\|^2]\right)^{1/2} \\
& \leq \|\sum_{i = 1}^n \EE[Y_i]\| + \left(\sum_{i=1}^n \EE[\|Y_i\|^2]\right)^{1/2}.
\end{align*}
Therefore, there exists $C'_p > 0$ such that 
\begin{multline}
\label{eq:Rosenthal}
\EE\left[\left\|\sum_{i=1}^n Y_i\right\|^p\right]^{1/p} \leq C'_p\bigg( \|\sum_{i = 1}^n \EE[Y_i]\| + \left(\sum_{i=1}^n \EE[\|Y_i\|^2]\right)^{1/2} \\
+ \left(\sum_{i=1}^n \EE[\|Y_i\|^p]^{1/p} \right)^{1/p} \bigg).
\end{multline}

We are now ready to start the proof of Theorem \ref{thm:CTL}. 
Let $p\geq2$ and let $X_1,\dots,X_n$ be independent random variables such that $\sum_{i=1}^n \EE[\|X_i\|^{p+q}] < \infty$ for some $0\leq q \leq 2$ and let $S_n = \sum_{i=1}^n X_i$. Let $X'_1,\dots,X'_n$ be independent copies of $X_1, \dots, X_n$. For any $t > 0$, we pose $\alpha(t) = e^{2t} - 1$ and
\[
(S_{n})_t = S_n + n^{-1/2}(X'_I - X_I)1_{\|X'_I\|, \|X_I\| \leq \sqrt{n \alpha(t)}},
\]
where $I$ is a uniform random variable taking values in $\{1,\dots,n\}$. 
Let us first show that, for any $t > 0$, the pairs $(S_n, (S_n)_t)$ and $((S_n)_t, S_n)$ follow the same law. Since $X_1,\dots,X_n,X'_1,\dots,X'_n$ are independent, it is sufficient 
to show that $(X_I, X_I + (X'_I-X_I)1_{\|X'_I\|, \|X_I\| \leq \sqrt{n \alpha(t)}})$ and $(X_I + (X'_I-X_I)1_{\|X'_I\|, \|X_I\| \leq \sqrt{n \alpha(t)}}, X_I)$ follow the same law. 
Given $t > 0$ and $E,F$ two open subsets of $\RR^d$, we have 
\begin{multline*} 
\PP(X_I \in E, X_I + (X'_I-X_I)1_{\|X'_I\|, \|X_I\| \leq \sqrt{n \alpha(t)}} \in F) = \\
 \PP(X_I \in E, X_I \in F \mid \max(\|X_I\|, \|X'_I\|) > \sqrt{n \alpha(t)}) \PP(\max(\|X_I\|, \|X'_I\|) > \sqrt{n \alpha(t)}) \\
 + \PP(X_I \in E, X'_I \in F \mid \max(\|X_I\|, \|X'_I\|) \leq \sqrt{n \alpha(t)}) \PP(\max(\|X_I\|, \|X'_I\|) \leq \sqrt{n \alpha(t)}) \\
\end{multline*}
and, since $X_I$ and $X'_I$ are independent and identically distributed, 
\begin{align*}
\PP(& X_I \in E, X_I + (X'_I-X_I)1_{\|X'_I\|, \|X_I\|  \leq \sqrt{n \alpha(t)}} \in F)  \\
 & = \PP(X_I \in E, X_I \in F \mid \max(\|X_I\|, \|X'_I\|) > \sqrt{n \alpha(t)}) \PP(\max(\|X_I\|, \|X'_I\|) > \sqrt{n \alpha(t)}) \\
& \qquad + \PP(X'_I \in E, X_I \in F \mid \max(\|X_I\|, \|X'_I\|) \leq \sqrt{n \alpha(t)}) \PP(\max(\|X_I\|, \|X'_I\|) \leq \sqrt{n \alpha(t)}) \\
& = \PP(X_I + (X'_I-X_I)1_{\|X'_I\|, \|X_I\| \leq \sqrt{n \alpha(t)}} \in E, X_I  \in F).
\end{align*}
Moreover, Equation~\ref{eq:condition} is verified since $\|(S_{n})_t - S_n\| \leq 2 \sqrt{n \alpha(t)}$ for any $t \geq 0$. 
Hence, we can apply Theorem~\ref{thm:WpGaussexch} to the measure $\nu_n$ of $S_n$ using the stochastic process $((S_n)_t)_{t \geq 0}$ and taking $s = \frac{1}{n}$ and apply the triangle inequality to obtain 
\[
W_p(\nu_n, \gamma) \leq \sum_{k=1}^\infty I_k, 
\]
where 
\begin{itemize}
\item $I_1 = \int_0^\infty e^{-t} \EE[\|\EE[n((S_n)_t - S_n)   + S_n\mid S_n]\|^p]^{1/p} dt$;
\item $I_2 = \int_0^\infty \frac{e^{-t} \sqrt{p-1}}{\sqrt{e^{2t}- 1}} \EE\left[\left\|\EE\left[n\frac{((S_n)_t - S_n)^{\otimes 2}}{2}  - I_d \mid S_n\right] \right\|^p\right]^{1/p} dt$;
\item $I_k = \frac{(p-1)^{(k-1)/2}}{2 (k-1)!} \int_0^\infty \frac{e^{-t}}{\sqrt{e^{2t}- 1}^{k-1}} \EE[\|\EE[n((S_n)_t - S_n)^{\otimes k} \mid S_n]\|_H^p]^{1/p} dt$.
\end{itemize}
Let $m, m' \in (0, \min(p+q,4) - 2]$. 
In order to conclude the proof all we have to do is to show there exists $C_p > 0$ depending only on $p$ such that $I_1, I_2$ and $ \sum_{k=3}^\infty I_k$
are bounded by 
\[
C_p n^{-1/2} ( n^{- q/2p} U_p + n^{-m/4} U_2 + U_1),
\]
where
\[
U_1 = \sum_{i=1}^n \begin{cases}
n^{-1/2 -m'/2} |1-m'|^{-1} \|\EE[X_i X_i^T \|X_i\|^{m'}]\| \text{ if $m' < 1$} \\
 n^{-1} \|\EE[X_i X_i^T \|X_i\|]\| \log(n) \text{ if $m' = 1$} \\
 n^{-1} |1-m'|^{-1} d^{1/2 - 1/(2m')} \|\EE[X_i X_i^T \|X_i\|^{m'}]\|^{1/m'}    \text{ if $m' > 1$}
 \end{cases}
\]
and
\[
U_2 = \left(\sum_{i=1}^n \EE[\|X_i\|^{2+m}] \right)^{1/2}
\]
and
\[
U_p =  \left(\sum_{i=1}^n \EE[\|X_i\|^{p+q}] \right)^{1/p}.
\]

In the remainder of this proof, we denote by $C$ a generic constant and by $C_p$ a generic constant depending only on $p$.
For any $t \geq 0$, we have, by definition of $(S_n)_t$,
\[
(S_n)_t - S_n = \frac{1}{\sqrt{n}} (X'_I - X_I) 1_{\|X_I\|, \|X_I'\| \leq \sqrt{n\alpha(t)}}
\]
and, since $I$ and $S_n$ are independent,
\[
\label{eq:prerosenthal}
\EE[n((S_n)_t - S_n)^{\otimes k} \mid S_n] =  n^{-k/2} \EE\left[\sum_{i=1}^n (X'_i - X_i)^{\otimes k} 1_{\|X_i\|, \|X_i'\| \leq \sqrt{n\alpha(t)}} \mid S_n\right]
\]
for any $k \in \NN$. 

\subsection*{Bounding $I_1$}

Taking $k = 1$ in Equation~(\ref{eq:prerosenthal}), we have 
\[
\EE[n((S_n)_t - S_n) + S_n \mid S_n]  
 =   \frac{1}{\sqrt{n}} \EE\left[\sum_{i=1}^n (X'_i - X_i) 1_{\|X_i\|, \|X_i'\| \leq \sqrt{n\alpha(t)}} + X_i  \mid S_n\right].
\]
Let $i \in \{1,\dots,n\}$. Since $X_i'$ is independent from $S_n$, $\EE[X_i' \mid S_n] = \EE[X_i'] = 0$. Hence, 
\begin{align*}
\EE[X_i' 1_{\|X_i\|, \|X_i'\| \leq \sqrt{n\alpha(t)}} \mid S_n] & = \EE[X_i' (1-1_{\max \|X_i\|, \|X_i'\| \geq \sqrt{n\alpha(t)}}) \mid S_n] \\
& =  -\EE[X_i' 1_{\max \|X_i\|, \|X_i'\| \geq \sqrt{n\alpha(t)}} \mid S_n].
\end{align*}
Therefore
\[
\EE[(X_i' - X_i) 1_{\|X_i\|, \|X_i'\| \leq \sqrt{n\alpha(t)}} + X_i \mid S_n]  = \EE[(X_i - X_i') 1_{\max \|X\|, \|X'\| \geq \sqrt{n\alpha(t)}} \mid S_n]
\]
and
\begin{multline*}
\EE[\|\EE[n((S_n)_t - S_n)  + S_n \mid S_n] \|^p]^{1/p} =  \\
n^{-1/2} \EE\left[\left\|\EE\left[\sum_{i=1}^n  (X_i - X_i') 1_{\max \|X_i\|, \|X_i'\| \geq \sqrt{n\alpha(t)}} \mid S_n \right] \right\|^p \right]^{1/p}.
\end{multline*}
Then, by applying Jensen's inequality to get rid of the conditional expectation, we obtain 
\begin{multline*}
\EE[\|\EE[n((S_n)_t - S_n) + S_n\mid S_n] \|^p]^{1/p} \leq \\
n^{-1/2} \EE\left[\left\|\sum_{i=1}^n (X_i - X_i') 1_{\max \|X_i\|, \|X_i'\| \geq \sqrt{n\alpha(t)}}\right\|^p\right]^{1/p}.
\end{multline*}
Now, let us pose 
\[
Y_i = (X_i - X_i') 1_{\max \|X\|, \|X'\| \geq \sqrt{n\alpha(t)}}.
\]
Since the $(X_i)_{1 \leq i \leq n}$ and the $(X'_i)_{1 \leq i \leq n}$ are independent and identically distributed  random variables, so are the $((X_i - X'_i)1_{\max \|X_i\|, \|X_i'\| \geq \sqrt{n\alpha(t)}})_{1\leq i \leq n}$. 
Hence, we can use Equation~(\ref{eq:Rosenthal}) to obtain 
\begin{multline}
\label{eq:I11}
\EE[\|\EE[n((S_n)_t - S_n) + S_n \mid S_n] \|^p]^{1/p}  \leq \\
C_p n^{-1/2} \left(\|\sum_{i=1}^n \EE[Y_i]\| + \left(\sum_{i=1}^n \EE[\|Y_i\|^2] \right)^{1/2} + \left(\sum_{i=1}^n \EE[ \|Y\|^p] \right)^{1/p} \right).
\end{multline}
Let $i \in \{1,\dots,d\}$ and let us pose $Y = Y_i$, $X = X_i$ and $X' = X'_i$. 
First, since $X$ and $X'$ follow the same law, 
\beq 
\label{eq:I14}
\EE[Y] = 0.
\eeq 
On the other hand, we have
\begin{align*}
\EE[\|Y\|^p] & = \EE[\|X - X'\|^p 1_{\max \|X\|, \|X'\| \geq \sqrt{n\alpha(t)}}] \\ 
& \leq \EE[(\|X\| + \|X'\|)^p 1_{\max \|X\|, \|X'\| \geq \sqrt{n\alpha(t)}}] \\
& \leq 2^p \EE[\max(\|X\|, \|X'\|)^p 1_{\max \|X\|, \|X'\| \geq \sqrt{n\alpha(t)}}] \\
& \leq 2^p (n\alpha(t))^{-q/2}\EE[\max(\|X\|, \|X'\|)^{p+q}] \\
& \leq  2^p (n\alpha(t))^{-q/2}\EE[\|X\|^{p+q} + \|X'\|^{p+q}]
\end{align*}
and thus 
\[
\EE[\|Y\|^p] \leq 4^p (n\alpha(t))^{-q/2} \EE[\|X\|^{p+q}].
\]
Then, since $\alpha(t) \geq 2t$, 
\begin{align*}
\int_0^\infty e^{-t} \left(\sum_{i=1}^n \EE[\|Y_i\|^p]\right)^{1/p} dt & \leq 4 n^{-q/2p} U_p  \int_0^\infty e^{-t} \alpha(t)^{-q/2p} dt \\
& \leq 2^{2 - q/2p} \int_0^\infty e^{-t} t^{-q/2p} dt  
\end{align*}
and 
\beq
\label{eq:I1p}
\int_0^\infty e^{-t} \left(\sum_{i=1}^n \EE[\|Y_i\|^p]\right)^{1/p} dt \leq C n^{-q/2p} U_p.
\eeq
Replacing $p$ by $2$ in the previous computations, we also obtain 
\beq
\label{eq:I12}
\int_0^\infty e^{-t} \left(\sum_{i=1}^n \EE[\|Y_i\|^2]\right)^{1/2} dt \leq C n^{-m/4} U_2.
\eeq
Finally, combining Equations~(\ref{eq:I11}),~(\ref{eq:I14}), (\ref{eq:I1p}) and (\ref{eq:I12}), we obtain 
\begin{align*}
I_1 & = \int_0^\infty e^{-t} \EE[\|\EE[n((S_n)_t - S_n) + S_n \mid S_n] \|_p^p]^{1/p} dt \\
& \leq C_p n^{-1/2} \left(n^{-m/4} U_2 + n^{-q/2p} U_p \right).
\end{align*}

\subsection*{Bounding $I_2$}

Taking $k = 2$ in Equation~(\ref{eq:prerosenthal}) gives
\begin{align*}
\EE\left[\frac{n((S_n)_t - S_n)^{\otimes 2}}{2}   - I_d  \mid S_n\right]&= \EE\left[ \sum_{i=1}^n \frac{(X'_i - X_i)^{\otimes 2}}{2 n} 1_{\|X_I\|, \|X_I'\| \leq \sqrt{n\alpha(t)}} - I_d \mid S_n\right] \\
& =  \frac{1}{n} \EE\left[\sum_{i = 1}^n \left(\frac{(X'_i - X_i)^{\otimes 2}}{2} 1_{\|X_i\|, \|X_i'\| \leq \sqrt{n\alpha(t)}} - I_d \right) \mid S_n\right].
\end{align*}
Again, taking
\[
Y_i = \left(\frac{(X'_i - X_i)^{\otimes 2}}{2} 1_{\|X_i\|, \|X'_i\| \leq \sqrt{n\alpha(t)}} - I_d\right)
\]
and using a combination of Jensen's inequality and Equation~(\ref{eq:Rosenthal}), we obtain 
\begin{multline}
\label{eq:I21}
\EE  \left[\left\|\EE\left[n \frac{((S_n)_t - S_n)^{\otimes 2}}{2}  - I_d \mid S_n\right]\right\|^p\right]^{1/p} \leq \\
 C_p n^{-1} \left(\left\|\sum_{i=1}^n \EE\left[Y _i\right]\right\| + \left( \sum_{i=1}^n \EE\left[\|Y_i\|^2\right]\right)^{1/2} + \left(\sum_{i=1}^n \EE\left[\|Y_i\|^p\right] \right)^{1/p} \right).
\end{multline}
Let $i \in \{1,\dots,d\}$ and let us pose $Y = Y_i$, $X = X_i$ and $X' = X'_i$. 
Let us start by bounding $\|\EE[Y]\|$. Since 
$\EE\left[X^{\otimes 2}\right] = \EE\left[X'^{\otimes 2}\right] = I_d$ and $\EE[X' \otimes X] = \EE[X']\otimes \EE[X] = 0$, we have 
\[
\EE\left[\frac{(X'-X)^{\otimes 2}}{2}\right] = I_d.
\]
 Therefore 
\[
\EE\left[Y\right] = \EE\left[\frac{(X' - X)^{\otimes 2}}{2} 1_{\max \|X\|, \|X'\| \geq \sqrt{n\alpha(t)}}\right].
\]
Let $0 \leq l \leq m'$. Letting $Z$ and $Z'$ be two random variables such that $X$, $X'$, $Z$, $Z'$ are independent and identically distributed  and denoting by $C$ a generic positive constant, we have 
\begin{align*}
\|\EE\left[Y \right] \| &  = \left<\EE\left[(X'-X)^{\otimes 2} 1_{\max \|X\|, \|X'\| \geq \sqrt{n \alpha(t)}}\right], \EE\left[(Z'-Z)^{\otimes 2} 1_{\max \|Z\|, \|Z'\| \geq \sqrt{n \alpha(t)}}\right]\right>^{1/2} \\
& =  \EE\left[\left<(X'-X)^{\otimes 2}, (Z'-Z)^{\otimes 2}\right>  1_{\max \|X\|, \|X'\| \geq \sqrt{n \alpha(t)}} 1_{\max \|Z\|, \|Z'\| \geq \sqrt{n \alpha(t)}}\right]^{1/2} \\
& =  \EE\left[\left<X'-X, Z'-Z\right>^2 1_{\max \|X\|, \|X'\| \geq \sqrt{n\alpha(t)}} 1_{\max \|Z\|,\|Z'\| \geq \sqrt{n \alpha(t)}}\right]^{1/2} \\
& \leq C \EE\left[<X, Z>^2 1_{\max \|X\|, \|X'\| \geq \sqrt{n\alpha(t)}} 1_{\max \|Z\|,\|Z'\| \geq \sqrt{n \alpha(t)}}\right]^{1/2} \\
& \leq C(n \alpha(t))^{-l/2} \EE\left[<X, Z>^2 \max(\|X\|,\|X'\|)^l  \max(\|Z\|,\|Z'\|)^l\right]^{1/2} \\
& \leq C(n \alpha(t))^{-l/2} \EE\left[<X, Z>^2 (\|X\|^l+\|X'\|^l)(\|Z\|^l +\|Z'\|^l)\right]^{1/2} \\
& \leq C(n \alpha(t))^{-l/2} \EE\left[<X^{\otimes 2}, Z^{\otimes 2}> (\|X\|^l+\|X'\|^l)(\|Z\|^l +\|Z'\|^l)\right]^{1/2} \\
& \leq  C(n \alpha(t))^{-l/2} \EE\left[<X^{\otimes 2}(\|X\|^l+\|X'\|^l), Z^{\otimes 2}(\|Z\|^l +\|Z'\|^l)> \right]^{1/2} \\
& \leq C (n \alpha(t))^{-l/2} \|\EE[X^{\otimes 2} (\|X\|^l + \|X'\|^l)]\| \\
& \leq C (n \alpha(t))^{-l/2}  (\|\EE[X^{\otimes 2} \|X\|^l]\| + \|\EE[X^{\otimes 2}\|X'\|^l]\|).
\end{align*}
Moreover, since $X$ and $X'$ are independent and identically distributed, 
\begin{align*}
\|\EE[X^{\otimes 2}\|X'\|^l]\| & = d^{1/2} \EE[\|X'\|^l] \\
& = d^{1/2}\EE[\|X\|^l] \\
& \leq d^{-1/2} \EE[\|X\|^2] \EE[\|X\|^l] \\
& \leq d^{-1/2} \EE[\|X\|^{2+l}]^{2/(2+l)} \EE[\|X\|^{2+l}]^{m/(2+l)} \\
&  \leq d^{-1/2} \EE[\|X\|^{2+l}] \\
& \leq d^{-1/2} \sum_{i=1}^d \EE[X_i^2 \|X\|^l] \\
& \leq \sqrt{\sum_{i=1}^d \EE[X_i^2 \|X\|^l]^2}  \\
& \leq \sqrt{\sum_{i,j=1}^d \EE[X_i X_j \|X\|^l]^2} \\
& \leq \|\EE[X^{\otimes 2} \|X\|^l]\|.
\end{align*}
Therefore, there exists $C > 0$ such that, for any $0 \leq l \leq m'$, 
\beq
\label{eq:auxI21}
\|\EE\left[Y \right] \|  \leq C (n \alpha(t))^{-l/2} \|\EE[X^{\otimes 2} \|X\|^l]\| . 
\eeq
Let $0 < t_0 < 1$. Using Equation~(\ref{eq:auxI21}) with $l = 0$ and $ l = m'$, we have  
\begin{multline*}
\int_0^\infty \frac{e^{-t}}{\sqrt{e^{2t} - 1}}   \|\EE\left[Y \right] \| dt     \\
\leq C \left( \int_0^{t_0} \frac{ \sqrt{d} e^{-t}}{\sqrt{\alpha(t)}} dt  +  n^{-m'/2} \|\EE[X^{\otimes 2} \|X\|^{m'}]\| \int_{t_0}^\infty  \frac{e^{-t}}{\alpha(t)^{(m'+1)/2}} dt \right) \\
\end{multline*}
and, since $\alpha(t) \leq 2t$, 
\begin{multline}
\label{eq:integ}
\int_0^\infty \frac{e^{-t}}{\sqrt{e^{2t} - 1}}  \|\EE\left[Y \right] \| dt \leq \\ 
C \left( \sqrt{d t_0}  +  n^{-m'/2} \|\EE[X^{\otimes 2} \|X\|^{m'}]\| \left(\int_{t_0}^1  \frac{1}{(2t)^{(m'+1)/2}} dt + 1 \right)\right).
\end{multline}
Thus, there exists $C'> 0$ such that, if $m' = 1$,
\[
\int_{t_0}^1  \frac{1}{t^{(m'+1)/2}} dt \leq C' \log(t_0)
\]
and, if $m' \neq 1$, 
\[
\int_{t_0}^1  \frac{1}{t^{(m'+1)/2}} dt \leq C' \frac{|1 - t_0^{(1-m')/2}|}{|1-m'|}.
\]
Hence, taking $t_0 = 0$ if $m' < 1$, $1/n$ if $m' = 1$ and $\frac{\|\EE[X^{\otimes 2} \|X\|^{m'}]\|^{2/m'}}{d^{1/m'} n}$ if $m' > 1$, we obtain that 
\begin{multline*}
\int_0^\infty \frac{e^{-t}}{\sqrt{e^{2t}-1}}  \|\EE\left[Y \right] \| dt  \leq  \\
C  \begin{cases}
 n^{-m'/2} |1-m'|^{-1} \|\EE[X^{\otimes 2} \|X\|^{m'}]\|   \text{ if $m' < 1$} \\
 n^{-1/2} \|\EE[X^{\otimes 2} \|X\|]\| \log(n) \text{ if $m' = 1$} \\
  n^{-1/2} |1-m'|^{-1} \|\EE[X^{\otimes 2} \|X\|^{m'}]\|^{1/m'} d^{1/2 - 1/(2m')}  \text{ if $m' > 1$}
 \end{cases}. 
\end{multline*}
Therefore
\beq 
\label{eq:I22}
\int_0^\infty \frac{e^{-t}}{\sqrt{e^{2t}-1}} \|\sum_{i=1}^n \EE[Y_i]\| dt \leq C n^{1/2} U_1. 
\eeq
Let us now deal with the higher moments of $Y$. We have 
\begin{align*}
\EE[\|Y\|^p]^{1/p} & \leq \EE\left[\left\|\frac{(X' - X)^{\otimes 2}}{2} 1_{\|X'\|, \|X\| \leq \sqrt{n\alpha(t)}}\right\|^p\right]^{1/p} + \|I_d\|  \\
& \leq \frac{1}{2} \EE[\|X' - X\|^{2p}  1_{\|X'\|, \|X\| \leq \sqrt{n\alpha(t)}}]^{1/p} + d^{1/2}  \\
& \leq \EE[\|X\|^{2p}  1_{\|X'\|, \|X\| \leq \sqrt{n\alpha(t)}}]^{1/p} + \EE[\|X\|^2]^{1/2}\\
& \leq \EE[\|X\|^{2p} 1_{\|X\| \leq \sqrt{n\alpha(t)}}]^{1/p} + \EE[\|X\|^p]^{1/p} 
\end{align*}
leading to 
\[
\EE[\|Y\|^p]^{1/p} \leq n^{1/2-q/2p} (\alpha(t)^{1/2-q/2p} + 1) \EE[\|X\|^{p+q}]^{1/p}
\]
and, since $\alpha(t) \geq 2t$, 
\beq
\label{eq:I23}
\int_0^\infty \frac{e^{-t}}{\sqrt{e^{2t} - 1}} \left(\sum_{i=1}^p \EE[\|Y_i\|^p]^{1/p} \right) dt \leq C n^{1/2-q/2p} U_p.
\eeq
And, by similar computations,
\beq
\label{eq:I24}
\int_0^\infty \frac{e^{-t}}{\sqrt{e^{2t} - 1}} \left(\sum_{i=1}^p \EE[\|Y_i\|^2]^{1/2} \right) dt \leq C n^{1/2-m/4} U_2.
\eeq
Combining Equations~(\ref{eq:I21}),(\ref{eq:I22}),(\ref{eq:I23}) and (\ref{eq:I24}), we have 
\begin{align*}
I_2 & =   \int_0^\infty \frac{e^{-t} \sqrt{p-1}}{\sqrt{e^{2t}- 1}} \EE\left[\left\|\EE\left[n\frac{((S_n)_t - S_n)^{\otimes 2}}{2}  - I_d \mid S_n\right] \right\|^p\right]^{1/p} dt \\
 & \leq  C_p n^{-1/2} \left(U_1 +  n^{-m/4} U_2 + n^{-q/2p} U_p\right).
\end{align*}
\subsection*{Bounding $\sum_{k=3}^\infty I_k$}
Consider some $k>2$. 
Let us first remark that 
\beq
\label{eq:Hnorm}
\forall M \in (\RR^d)^{\otimes k}, \|M\|_H \leq \max_{i \in \{1,\dots,d\}^{k-1}} \|H_i\|_\gamma \|M\| \leq \sqrt{(k-1)!} \|M\|.
\eeq
Now, let 
\[
 Y_i = \EE[(X_i' - X_i)^{\otimes k}  1_{\|X\|, \|X'\| \leq \sqrt{n\alpha(t)}} \mid S_n].
\]
Combining Equation~(\ref{eq:prerosenthal}), Jensen's inequality to get rid of the conditional expectation and Equation~(\ref{eq:Rosenthal}), we obtain 
\begin{multline*}
\label{eq:I31}
\EE[\|\EE[n ((S_n)_t - S_n)^{\otimes k} \mid S_n]\|^p]^{1/p} \leq \\
C_p n^{-k/2} \left(\| \sum_{i=1}^n \EE[Y_i]\| +  \left( \sum_{i=1}^n \EE[\|Y_i\|^2] \right)^{1/2} + \left(\sum_{i=1}^n \EE[\|Y_i\|^p]\right)^{1/p} \right).
\end{multline*}
Let $i \in \{1,\dots,d\}$ and let us pose $Y= Y_i$, $X = X_i$ and $X' = X'_i$. 
First, we have 
\begin{align*}
\EE[\|Y\|^p] & \leq \EE[\|X' - X\|^{kp}  1_{\|X\|, \|X'\| \leq \sqrt{n\alpha(t)}}] \\
&\leq 2^{kp} \EE[\|X\|^{kp} 1_{\|X\| \leq \sqrt{n\alpha(t)}}] \\
& \leq 2^{kp} (n\alpha(t))^{((k-1)p-q)/2} \EE[\|X\|^{p+q}].
\end{align*}
Thus, since $\alpha(t) \geq 2t$, 
\beq
\label{eq:I32}
\int_0^\infty \frac{e^{-kt}}{\sqrt{1-e^{-2t}}^{k-1}} \left(\sum_{i=1}^n \EE[\|Y_i\|^p]\right)^{1/p} dt \leq C 2^{k} n^{(k-1)/2-q/2p} U_p
\eeq
and, similarly, 
\beq
\label{eq:I33}
\int_0^\infty \frac{e^{-kt}}{\sqrt{1-e^{-2t}}^{k-1}} \left(\sum_{i=1}^n \EE[\|Y_i\|^2]\right)^{1/2} dt \leq C 2^{k} n^{(k-1)/2-m/4} U_2. 
\eeq
Let us now bound $\|\EE[Y]\|$. Since $X'$ and $X$ are independent and identically distributed, $\EE[Y]= 0$ for odd values of $k$. Let us now consider an even integer $k > 2$ and let $0 \leq l \leq m'$. 
Let us denote by $Z$ and $Z'$ two random variables such that $X,X', Z,Z'$ are independent and identically distributed. We have 
\begin{align*}
\|\EE[Y]\| & = \EE\left[<X'-X,Z'-Z>^k 1_{\|X\|,\|X'\| \leq \sqrt{n \alpha(t)}} 1_{\|Z\|,\|Z'\| \leq \sqrt{n \alpha(t)}}\right]^{1/2} \\
& \leq 2^k \EE\left[<X,Z>^k 1_{\|X\|,\|Z\| \leq \sqrt{n \alpha(t)}} \right]^{1/2} \\
& \leq 2^k\EE\left[<X,Z>^2 \|X\|^{k-2} \|Z\|^{k-2} 1_{\|X\|, \|Z\| \leq \sqrt{n \alpha(t)}}\right]^{1/2} \\
& \leq 2^k(n\alpha(t))^{ (k - l - 2)/2}\EE\left[<X,Z>^2 \|X\|^l \|Z\|^l \right]^{1/2} \\
& \leq 2^k(n\alpha(t))^{ (k - l - 2)/2} \|\EE[X^{\otimes 2} \|X\|^l]\|,
\end{align*}
and thus 
\[
n^{-k/2} \frac{e^{-t}}{\sqrt{e^{2t}-1}^{(k-1)/2}}  \|\EE[Y]\|  \leq 2^k  e^{-t} n^{-l/2 - 1} \alpha(t)^{-(l+1)/2} \|\EE[X^{\otimes 2} \|X\|^l]\|.
\]
Then, using the same integration procedure we used to derive Equation~(\ref{eq:I22}), we obtain that 
\[
\label{eq:I34}
n^{-k/2} \int_0^\infty \frac{e^{-kt}}{\sqrt{1-e^{-2t}}^{k-1}}  \|\sum_{i=1}^n \EE[Y_i]\| dt \leq C n^{-1/2} U_1.
\]
Combining Equations~(\ref{eq:Hnorm}), (\ref{eq:I32}),(\ref{eq:I33}) and (\ref{eq:I34}), we finally obtain 
\begin{align*}
I_k & = \frac{(p-1)^{(k-1)/2}}{2 (k-1)!} \int_0^\infty \frac{e^{-kt}}{\sqrt{1-e^{-2t}}^{k-1}} \EE[\|\EE[n((S_n)_t - S_n)^{\otimes k} \mid S_n]\|_H^p]^{1/p} dt \\
& \leq \frac{(p-1)^{(k-1)/2}}{2 \sqrt{(k-1)!}} \int_0^\infty \frac{e^{-kt}}{\sqrt{1-e^{-2t}}^{k-1}} \EE[\|\EE[n((S_n)_t - S_n)^{\otimes k} \mid S_n]\|^p]^{1/p} dt \\
& \leq C_p n^{-1/2} \frac{ 2^{k-1} (p-1)^{(k-1)/2}}{\sqrt{(k-1)!}} \left( U_1 + n^{-m/4} U_2 + n^{-q/2p} U_p \right). 
 \end{align*}
Finally, since $\sum_{k=3}^\infty \frac{2^k (p-1)^{k/2}}{\sqrt{(k-1)!}} < \infty$, 
\[
\sum_{k=3}^\infty I_k \leq  C_p n^{-1/2} \left( U_1 + n^{-m/4} U_2 + n^{-q/2p} U_p \right). 
\]

\section{Proof of Corollary~\ref{cor:rw}}
\label{sec:corrw}
Suppose Assumption~\ref{ass:main} is verified. 
For $t < \tau$, let 
\[
S(t) = f_1(t) \|b(X)\|_{a^{-1}(X)} + f_2(t) \sqrt{d}
\]
and, for $t \geq \tau$, let 
\begin{align*}
S(t) = & f_1(t) \left\| \frac{1}{\step} \int_{y \in \RR^d} (y-X) K(X,dy)  - b(X) \right\|_{a^{-1} (X)} \\
& + f_2(t) \left\|  \frac{1}{\step} \int_{y \in \RR^d} \frac{(y-X)^{\otimes2}}{2} K(X,dy)  - a(X) \right\|_{a^{-1}(X)} \\
& + \sum_{k=3}^\infty \frac{f_k(t)}{s k!} \left\|\int_{y \in \RR^d} (y-X)^{\otimes k} K(X,dy)\right\|_{a^{-1}(X)}
\end{align*}

By Equation~(\ref{eq:fkbound}), there exists $C> 0$ such that 
\[
\forall k = 1,2, \int_{0}^\tau f_k(t) dt \leq C^k \frac{(d (k-1))^{(k-1)/2}}{(k-1)} \tau^{(3-k)/2}
\]
and
\[
\forall k \geq 1, \int_{\tau}^T f_k(t) dt \leq C^k \frac{(d (k-1))^{(k-1)/2}}{(k-1) \tau^{(k-3)/2}}.
\]
Therefore, by the triangle inequality, 
\[
\int_0^\tau \EE[S(t)^2]^{1/2} dt \leq C \tau \EE[\|b(X)\|^2_{a^{-1}(X)}]^{1/2} + C^2 d \sqrt{\tau}
\]
and
\begin{align*}
\int_\tau^T \EE[S(t)^2]^{1/2} dt \leq & C T \EE\left[\left\| \frac{1}{\step} \int_{y \in \RR^d} (y-X) K(X,dy)  - b(X) \right\|^2_{a^{-1} (X)}\right]^{1/2} \\
& + 2 C^2 \sqrt{T d }  \EE\left[\left\| \frac{1}{\step} \int_{y \in \RR^d} \frac{(y-X)^{\otimes2}}{2} K(X,dy)  - a(X) \right\|^2_{a^{-1}(X)}\right]^{1/2} \\
& + C^3 \frac{\log(\tau) d}{3 \step}   \EE\left[\left\|  \int_{y \in \RR^d} (y-X)^{\otimes 3} K(X,dy)\right\|^2_{a^{-1}(X)}\right]^{1/2} \\
& + \sum_{k=4}^\infty C^k \frac{2 (d (k-1))^{(k-1)/2}}{(k-1) k! \tau^{(k-3)/2}  \step} \EE\left[\left\|  \int_{y \in \RR^d} (y-X)^{\otimes k} K(X,dy)\right\|^2_{a^{-1}(X)}\right]^{1/2}.
\end{align*}
Finally, applying Theorem~\ref{thm:main2} using the stochastic process given by Equation~(\ref{eq:rwsto}), we obtain 
\[
(1-c e^{-\kappa T}) W_2(\pi,\mu) \leq \int_0^T \EE[S(t)^2]^{1/2} dt,
\]
completing the proof of Corollary~\ref{cor:rw}. 
\section{Proof of Proposition~\ref{pro:randomgraph}}
\label{subsec:knn}
Let $\mathcal{T} = (\RR/\ZZ)^d$  be the $d$-dimensional flat torus and let $\mu$ be a measure supported on $\mathcal{T}$ with strictly positive density $f \in C^\infty(\mathcal{T},\RR)$. 
While $\mathcal{T}$ is not an open set of $\RR^d$, the arguments used in the proof of Theorem~\ref{thm:main2} still hold. 
Let $\tilde{\mu}$ be the measure with density $\tilde{f} = C f^{2+2/d}$, where $C > 0$ is a renormalization factor. Let us denote the Lebesgues measure by $\lambda$ and let $\nabla .$ be the divergence operator. For any two functions $\phi, \psi \in C^\infty(\mathcal{T}, \RR)$, we have, using an integration by parts with respect to the Lebesgues measure, 
\begin{align*}
\int_T \phi f^{-2/d} (\nabla \log f . \nabla \psi & + \frac{1}{2} \Delta \psi ) d\tilde{\mu} \\
& = C \int_T \phi  (\nabla \log f . \nabla \psi  + \frac{1}{2} \Delta \psi ) f^2 d \lambda \\
& = \frac{C}{2} \int_T \phi  (\nabla \log f^2 . \nabla \psi  + \Delta \psi ) f^2 d \lambda \\
& = \frac{C}{2} \int_T \phi  (\nabla f^2 . \nabla \psi  + f^2 \Delta \psi )  d \lambda \\
& = \frac{C}{2} \int_T \phi  \nabla .(f^2 \nabla \psi) d \lambda \\
& = - \frac{C}{2} \int_T f^2 \nabla \phi  . \nabla \psi d \lambda \\
& = \frac{C}{2} \int_T \psi  \nabla .(f^2 \nabla \phi) d \lambda \\
& = \frac{C}{2} \int_T \psi  (\nabla \log f^2 . \nabla \phi  + \Delta \phi ) f^2 d \lambda \\
& = \int_T \psi f^{-2/d} (\nabla \log f . \nabla \phi  + \frac{1}{2}\Delta \phi ) d\tilde{\mu}.
\end{align*}
The measure $\tilde{\mu}$ is thus the reversible probability measure of the operator 
\[
\mathcal{L}_{\tilde{\mu}} = f^{-2/d}  (\nabla \log f . \nabla  + \frac{\Delta}{2}).
\] 
As $\mathcal{T}$ is compact and $f$ has bounded derivatives of all orders and is strictly positive, $f^{-2/d} \nabla \log f$ and $f^{-2/d} I_d$ admit bounded derivatives of all orders and $f^{-2/d} I_d$ is strictly positive-definite on all of $\mathcal{T}$ and thus items $(i)$ to $(v)$ of Assumption~\ref{ass:main} are verified. 
Finally, thanks to Corollary 2.2 \cite{Fwang}, the Markov semigroup associated to $\mathcal{L}_{\tilde{\mu}}$ verifies property $(vi)$ of Assumption~\ref{ass:main}. 

Let $k,n \in \NN$ such that $k < n$, let $(X_1, \dots, X_n)$ be independent random variables drawn from the measure $\mu$ and let $\pi_{k,n}$ be the invariant measure of the random walk on the $k$-nearest neighbor graphs built on the point cloud $(X_1,\dots,X_n)$. For $x \in \mathcal{T}$ and $r > 0$, we denote by $B(x,r)$ the ball of radius $r$ centred in $x$. Finally, let us  pose 
\[
\forall x \in \mathcal{T}, \radius_{\mathcal{X}_n}(x) = \inf \left\{s \in \RR^+ | \sum_{i =1}^n 1_{\|X_i - x\| \leq s} \geq k \right\}.
\]
In the remainder of this Section, we denote by $C$ a generic constant depending only on $f$ and $d$. 
Since Assumption~\ref{ass:main} is verified and since $\mathcal{T}$ is compact, we can apply Corollary~\ref{cor:rw} to the measures $\pi_{k,n}$ and $\tilde{\mu}$ with $T = 1$ and 
\[
\tau = \step = \left(\frac{k}{n} \right)^{2/d} \frac{\int_{\|x\|\leq 1} x_1^2 dx}{\left(\int_{\|x\|\leq 1} 1 dx\right)^{1+2/d}},
\]
we obtain that 
\begin{multline}
\label{eq:mainknn}
W_{2,a} (\pi_{k,n}, \tilde{\mu})  \leq C (s + \sqrt{s} + I_1 + I_2 + s^{-1} \log(s) I_3)  \\
+ \sum_{m=4}^\infty \frac{C^m (m-1)^{(m-1)/2}}{m! s^{(m-1)/2}} I_m,
\end{multline}
where
\begin{itemize}
\item $I_1 = \sup_{i \in \{1,\dots,n\}} \left\|\frac{1}{k \step} \sum_{X_j \in \mathcal{B}(X_i,r_{\mathcal{X}_n}(X_i))}  (X_j - X_i) - f^{-2/d} \nabla \log f (X_i) \right\|$;
\item $I_2 = \sup_{i \in \{1,\dots,n\}} \left\|\frac{1}{k \step} \sum_{X_j \in \mathcal{B}(X_i,r_{\mathcal{X}_n}(X_i))}  \frac{(X_j - X_i)^{\otimes 2}}{2} -  \frac{f(X_i)^{-2/d} I_d}{2} \right\|$ and 
\item $\forall m > 2, I_m = \sup_{i \in \{1,\dots,n\}} \left\|\frac{1}{k} \sum_{X_j \in \mathcal{B}(X_i,r_{\mathcal{X}_n}(X_i))}  (X_j - X_i)^{\otimes m} \right\|$.
\end{itemize}

In the remainder of this Section, we show that, with probability greater than $1 - \frac{C}{n}$, 
\begin{enumerate}[label=(\roman*)]
\item  $I_1 \leq C ( \frac{\sqrt{\log n} n^{1/d}}{k^{1/2+1/d}} + \left(\frac{k}{n}\right)^{1/d})$;
\item  $I_2 \leq C\left(\sqrt{\frac{\log n}{k}} + \left(\frac{k}{n}\right)^{2/d}\right)$;
\item $s^{-1} I_3 \leq C \left(\frac{\sqrt{\log n} k^{1/d}}{ n^{1/d}k^{1/2}} + \left(\frac{k}{n}\right)^{2/d}\right)$;
\item $\forall m > 3,  I_m\leq C^m \left(\frac{k}{n}\right)^{m/d}$.
\end{enumerate}
Proposition~\ref{pro:randomgraph} is then obtained by injecting these bounds in Equation~(\ref{eq:mainknn}) and by remarking that, since $T$ is compact and $a$ is smooth, then $W_{2,a} \leq C W_2$. 

For $k > 0, r> 0$, we pose 
$V_k = \int_{\mathcal{B}(0,1)} x_1^k dx$. Let $x \in T$ and let us pose 
$N_r = \sum_{i=1}^n 1_{X_i \in \mathcal{B}(x,r)}$ and $P_r = \mu(B(x,r))$. 
Let $0 \leq \delta< 1$. By the multiplicative Chernoff bound,
\[
P\left(|N_r - n P_r| \geq \delta n P_r\right) \leq  2 e^{-\delta^2 n P_r / 3}.
\]
Thus, taking $\delta = \left(\frac{3\log(2n^2)}{n P_r}\right)^{1/2}$,
we obtain 
\beq
\label{eq:Chernoff}
P\left(|N_r - n P_r| \geq  (n P_r 3 \log(2n^2))^{1/2}\right) \leq \frac{1}{n^2}.
\eeq
Taking $r_M = \left( \frac{2k}{n V_0 \min f} \right)^{1/d}$, we have $P_{r_M} \geq \frac{2k}{n}$ and 
\[
P\left(N_{r_M} \leq 2k  - C \sqrt{ k \log n}\right) \leq \frac{1}{n^2}.
\]
If $\frac{k}{\log(n)}$ is sufficiently large, then $\PP(N_{r_M} \geq k) \geq 1 - \frac{1}{n^2}$ and
 \[
 \PP(\radius_{\mathcal{X}_n}(x) \leq r_M) \geq 1 - \frac{1}{n^2}
\]
 and, by
a union-bound, 
\[
\PP(\sup_{x \in \mathcal{X}_n} \radius_{\mathcal{X}_n}(x) \leq r_M) \geq 1 - \frac{1}{n}.
\]
 Therefore, if $\frac{k}{\log(n)}$ is sufficiently large then, for all $m > 3$, 
\begin{align*}
 I_m & \leq  \frac{1}{k} \sum_{X_j \in \mathcal{B}(X_i,\tilde{r})}  \left\|X_j - X_i\right\|^m \\
& \leq r_M^m
\end{align*}
with probability $1 - \frac{1}{n}$ and inequality $(iv)$ follows. 

Let us now prove inequality $(i)$. Let $x \in T$ and $r = \left( \frac{k}{n V_0 f(x)} \right)^{1/d}$. By a Taylor expansion,
\begin{align*}
\EE[(X_i - x)& 1_{X_i \in \mathcal{B}(x,r)}] \\
& = \int_{\mathcal{B}(x,r)} (y-x) \mu(dy) \\
& =  \int_{\mathcal{B}(x,r)} (y-x)f(y) dy \\
& =  \int_{\mathcal{B}(x,r)} (y-x)f(x) + (y-x)^{\otimes 2} \nabla f(x) + \frac{(y-x)^{\otimes 3} \nabla^2 f (x)}{2} + O(r^4) dy.
\end{align*}
Therefore, by symmetry of $\mathcal{B}(x,r)$, we have
\[
\EE[(X_i - x)1_{X_i \in \mathcal{B}(x,r)}] =  V_2 r^{d+2} \nabla f(x) + O(r^{d+4})
\]
and, by definition of $s$,
\beq
\label{eq:auxknn1}
\EE[(X_i - x)1_{X_i \in \mathcal{B}(x,r)}]= \frac{k s }{n f^{2/d}(x)} \nabla \log f (x) + O\left(s \left(\frac{k}{n} \right)^{2/d} \right).
\eeq
Let $b_1 = \frac{1}{k s} \sum_{i=1}^n (X_i - x )1_{X_i \in \mathcal{B}(x,r)}$. 
Since 
\[
\|(X_i - x )1_{X_i \in \mathcal{B}(x,r)} \| \leq r \leq C \left(\frac{k}{n}\right)^{1/d}
\]
and 
\[
\EE[\|X_i - x \|^21_{X_i \in \mathcal{B}(x,r)} ] \leq r^2 P_r \leq r^2 P_{r_M} \leq C \left( \frac{k}{n} \right)^{1+2/d}, 
\]
applying Bernstein's inequality, we obtain that there exists $C > 0$ such that 
\[
\forall t > 0, P(\left\|ks b_1 - ks \EE[b_1]\|_\infty \geq t \right) \leq 2 d e^{- \frac{t^2}{ C (k (k/n)^{2/d} + (k/n)^{1/d}t)}}.
\]
Thus, if $\frac{k}{\log(n)}$ is large enough, taking $t = \sqrt{2 C \log(2 d n^2) k} \left(\frac{k}{n}\right)^{1/d}$, we have 
\[
P\left(\left\|ks b_1 - n \EE[(X_i - x)1_{X_i \in \mathcal{B}(x,r)}] \right\|_\infty \geq C \left( \frac{k}{n} \right)^{1/d} \sqrt{k \log n}\right) \leq \frac{1}{n^2}
\]
or 
\[
P\left(\left\|b_1 - \frac{n}{ks} \EE[(X_i - x)1_{X_i \in \mathcal{B}(x,r)}] \right\|_\infty \geq C \left(\frac{n^{1/d} \sqrt{\log n}}{k^{1/2 + 1/d}}\right) \right) \leq \frac{1}{n^2}.
\]
Thus, by Equation~(\ref{eq:auxknn1}), 
\beq
\label{eq:auxknn2}
P\left(\left\|b_1 - f^{-2/d}(x) \nabla \log f (x) \right\|_\infty \geq C \left( \frac{n^{1/d} \sqrt{\log n}}{k^{1/2 + 1/d}} + \left(\frac{k}{n} \right)^{2/d} \right) \right) \leq \frac{1}{n^2}. 
\eeq
Now, since $\|\nabla^2 f\|$ is bounded on $T$, 
\begin{align*}
|P_r - V_0 r^{d}| = \left|\int_{B(x,r)} f(y) - f(x) dy \right | & \leq \int_{B(x,r)} r^2 \max_{y \in T} \|\nabla^2 f(y)\| dy \\
&\leq V_2 r^{d+2} \max_{y \in T} \|\nabla^2 f(y)\|,
\end{align*}
thus
\[
\left|P_r - \frac{k}{n}\right| \leq C\left(\frac{k}{n}\right)^{1+2/d}
\]
and, by Equation~(\ref{eq:Chernoff}), 
\beq
\label{eq:radiusconc}
\PP\left(|N_r - k| \leq C \left(\sqrt{k \log n}+\frac{k^{1+2/d}}{n^{2/d}}\right)\right) \geq 1 - \frac{1}{n^2}.
\eeq
Taking $b_2 = \frac{1}{ks} \sum_{X_i \in \mathcal{B}(x,\radius_{\mathcal{X}_n}(x))} X_i - x$, we have 
\[
\|b_2 - b_1\| \leq \frac{r_M}{ks} |N_r - N_{\radius_{\mathcal{X}_n}(x)}| = \frac{r_M}{ks} |N_r - k| \leq C \frac{n^{1/d}}{k^{1+1/d}} |N_r - k|.
\]
Combining this bound with Equation~(\ref{eq:radiusconc}), we obtain  
\beq
\label{eq:auxknn3}
P\left(\|b_1 - b_2\| \leq C \left( \frac{ n^{1/d} \sqrt{\log n}}{k^{1/2 + 1/d}} + \left( \frac{k}{n} \right)^{1/d} \right)\right) \geq 1 - \frac{1}{n^2}.
\eeq
Combining Equation~(\ref{eq:auxknn2}) and (\ref{eq:auxknn3}), we have, with probability $1 - \frac{2}{n^2}$, 
\begin{align*}
\bigg\|\frac{1}{k \step} \sum_{X_i \in \mathcal{B}(x,r_{\mathcal{X}_n}(x))}  (X_i - x) & - f^{-2/d} \nabla \log f \bigg\| \\
& = \left\|b_2  - f^{-2/d}\nabla \log f \right\| \\
& \leq \left\|b_1-f^{-2/d}\nabla \log f \right\| + C \left( \frac{n^{1/d} \sqrt{\log n} }{k^{1/2+1/d}} + \left( \frac{k}{n} \right)^{1/d} \right) \\
& \leq  C \left( \frac{n^{1/d} \sqrt{\log n} }{k^{1/2+1/d}} + \left(\frac{k}{n}\right)^{1/d}\right).
\end{align*}
Inequality $(i)$ is fianlly obtained by using a union-bound. 

Let us derive inequalities $(ii)$ and $(iii)$ through similar computations. First, using a Taylor expansion, we obtain  
\[ 
\EE[(X_i - x)^{\otimes 2} 1_{X_i \in \mathcal{B}(x,r)}] = V_2 r^{d+2} f(x) I_d + O(r^{d+4}) 
\]
and 
\[
\EE[(X_i - x)^{\otimes 3} 1_{X_i \in \mathcal{B}(x,r)}] = O(r^{d+4}) 
\] 
Letting $a_1 = \frac{1}{2ks} \sum_{i=1}^n (X_i - x)^{\otimes 2} 1_{X_i \in \mathcal{B}(x,r)}$ and 
$c_1 = \frac{1}{ks} \sum_{i=1}^n (X_i - x)^{\otimes 3} 1_{X_i \in \mathcal{B}(x,r)}$, we have, by Bernstein's inequality and if $\frac{k}{n}$ is sufficiently small, 
\[
P\left(\left\|a_1 - \frac{1}{2f^{2/d}(x)} I_d \right\| \geq C \left( \sqrt{\frac{ \log n}{k}} + \left(\frac{k}{n} \right)^{2/d} \right) \right) \leq \frac{1}{n^2} 
\]
and 
\[
P\left(\left\|c_1 \right\| \geq C \left( \frac{ \sqrt{k^{1/d} \log n}}{n^{1/d} k^{1/2}} + \left(\frac{k}{n} \right)^{2/d} \right) \right) \leq \frac{1}{n^2}.
\]
Then, letting $a_2 = \frac{1}{2ks} \sum_{i=1}^n (X_i - x)^{\otimes 2} 1_{X_i \in \mathcal{B}(x,r_{\mathcal{X}_n}(x))}$ and 
$c_2 = \frac{1}{ks} \sum_{i=1}^n (X_i - x)^{\otimes 3} 1_{X_i \in \mathcal{B}(x,r_{\mathcal{X}_n}(x))}$ and using Equation~(\ref{eq:radiusconc}) once more, we obtain, 
\[
P\left(\|a_2 - a_1\| \leq  C\left( \sqrt{\frac{ \log n}{k}} + \left(\frac{k}{n} \right)^{2/d} \right)\right) \geq 1 - \frac{1}{n^2}
\]
and
\[
P\left(\|c_2 - c_1\| \leq C \left(\frac{\sqrt{\log n} k^{1/d}}{ n^{1/d}k^{1/2}} + \left(\frac{k}{n}\right)^{2/d}\right)\right) \geq 1 - \frac{1}{n^2}
\]
From here, we derive 
\[
 \left\|\frac{1}{k \step} \sum_{X_i \in \mathcal{B}(x,r_{\mathcal{X}_n}(x))}  \frac{(X_i - x)^{\otimes 2}}{2} -  \frac{f(x)^{-2/d} I_d}{2} \right\|  \leq 
 C\left( \sqrt{\frac{ \log n}{k}} + \left(\frac{k}{n} \right)^{2/d} \right)
\]
and 
\[
 \left\|\frac{1}{k \step} \sum_{X_i \in \mathcal{B}(x,r_{\mathcal{X}_n}(x))}  (X_i - x)^{\otimes 3} \right\|  \leq 
 C \left(\frac{\sqrt{\log n} k^{1/d}}{ n^{1/d}k^{1/2}} + \left(\frac{k}{n}\right)^{2/d}\right).
\]
Inequalities $(ii)$ and $(iii)$ are finally obtained by a union-bound inequality. 
\section{Proof of Proposition~\ref{pro:LMC}}
\label{subsec:LMCproof}

Let us assume that the assumptions of Proposition~\ref{pro:LMC} are verified. 
Let $\timestep > 0$ and let $(M_n)_{n \geq 0}$ be the Markov chain with $M_0 = 0$ and increments given by 
\[
M_{n+1} = M_n + \left(- \timestep \frac{\partial u}{\partial x_{I_n}} (M_n) + \sqrt{2 \timestep} B_n\right) e_{I_n},
\]
where $e_1,\dots, e_d \in \RR^d$ is the canonical basis of $\RR^d$, $(I_n)_{n \in \NN}$ are independent uniform random variables on $\{1,\dots,d\}$ and  $(B_n)_{n \in \NN}$ are independent Rademacher random variables. 
Let us denote by $\pi_\timestep$ the invariant measure of the Markov chain $(M^n)_{n \geq 0}$ and let 
$n \in \NN$. By the triangle inequality, we have 
\[
W_2(\nu_n, \mu) \leq W_2(\nu_n, \pi_\timestep) + W_2(\pi_\timestep, \stationary). 
\]

Before bounding these two terms, let us prove a few results on $\pi_\timestep$. 
Let $X$ be a random variable drawn from $\pi_{\timestep}$ and let 
\[
\increment = \left(- \timestep \frac{\partial u}{\partial x_I} (X) + \sqrt{2 \timestep} B \right) e_I,
\]
where $I$ is a uniform random variable on $\{1,\dots,d\}$ and $B$ is a Rademacher random variable. Let us start by bounding the second moment of $X$. 
Since $\pi_\timestep$ is the invariant measure of $(M^n)_{n \geq 0}$, $X$ and $X + \increment$ follow the same law and 
\begin{align*}
0 & =\EE[ \|X + \increment \|^2] - \EE[ \|X \|^2]  \\
& = \EE[ 2 \left<X, \increment\right> + \|\increment\|^2 ] \\
& = \EE\left[ - 2 \timestep X_I \frac{\partial u}{\partial x_I} (X) + \left(\timestep \frac{\partial u}{\partial x_I} (X)\right)^2 \right] + 2 \timestep \\
& =  \frac{1}{d}\EE\left[- 2 \timestep <X, \nabla u(X)> + \timestep^2  \|\nabla(u)(X) \|^2\right] + 2 \timestep.
\end{align*}
Thus, by Equations~(\ref{eq:assu1}) and (\ref{eq:assu2}), we have 
\[
\frac{1}{\Dim} \EE[ - 2 \timestep \rho \|X\|^2 + \timestep^2 \Lip^2 \|X\|^2 ] + 2  \timestep \leq 0,
\]
from which we deduce that
\beq
\label{eq:secondmom}
\EE[ \|X \|^2] \leq \frac{2 \Dim }{2  \rho   - \Lip^2 \timestep} \leq \frac{\Dim}{\rho} + O(\Dim \timestep).
\eeq
In addition to this bound, we can obtain a bound $\|X\|_\infty$. Indeed, for any $n \geq 0$ and any $i \in \{1,\dots, d\}$, we have
\begin{align*}
|M^{n+1}_i| & \leq |(M_n)_i - \timestep \frac{\partial u}{\partial x_i}((M_n)_i)|  + \sqrt{2 \timestep} \\
& \leq \left(((M_n)_i)^2 - 2 (M_n)_i\timestep \frac{\partial u}{\partial x_i}((M_n)_i) +  \timestep^2 \left(\frac{\partial u}{\partial x_i}((M_n)_i)\right)^2 \right)^{1/2}+  \sqrt{2 \timestep} \\
& \leq |(M_n)_i| \sqrt{1 - 2 \timestep \rho + \timestep^2 \Lip^2} +  \sqrt{2 \timestep}. 
\end{align*}
Thus, since $\pi_\timestep$ is the invariant measure of $(M_n)_{n \geq 0}$, 
\[
\|X\|_\infty \leq \frac{\sqrt{2 \timestep}}{1 - \sqrt{ 1 - 2 \timestep \rho + \timestep^2 \Lip^2}} 
\]
and there exists $\timestep_0 > 0$ such that
\beq 
\label{eq:maxbound}
\|X\|_\infty \leq \frac{2}{ \rho \sqrt{\timestep}} 
\eeq
as long as $\timestep < \timestep_0$. From now on, we assume that $\timestep < \timestep_0$.

Let us bound $W_2(\nu_n, \pi_\timestep)$. 
For $x \in \RR^d$, we denote by $\nu_x$ the measure of 
\[
M^x= x + \left(- \timestep \frac{\partial u}{\partial x_{I}} (x) + \sqrt{2 \timestep} B\right) e_{I},
\]
where $B$ is a Bernoulli random variable. 
Let $x ,y \in \RR^d$. For any $i \in \{1,\dots,d\}$, we have
\begin{align*}
\EE[(M^x - M^y)_i^2] & = \frac{\Dim - 1}{\Dim} (x_i - y_i)^2  + \frac{1}{\Dim} \left(x_i - y_i - \timestep \left( \frac{\partial u}{\partial x_{i}} (x) - \frac{\partial u}{\partial x_{i}} (y) \right) \right)^2 \\
& = (x_i - y_i)^2 - \frac{2 \timestep}{\Dim} (x_i - y_i) \left(\frac{\partial u}{\partial x_{i}} (x) - \frac{\partial u}{\partial x_{i}} (y)\right) + \frac{\timestep^2}{\Dim} \left( \frac{\partial u}{\partial x_{i}} (x) - \frac{\partial u}{\partial x_{i}} (y) \right)^2 \\
& \leq \left(1 + \frac{\Lip^2 \timestep^2  - 2 \rho \timestep}{\Dim} \right) (x_i - y_i)^2.
\end{align*}
Hence, by definition of the Wasserstein distance, 
\[
W^2_2(\nu_x, \nu_y) \leq \EE[\|M^x - M^y\|^2]  \leq \left(1 + \frac{ \Lip^2 \timestep^2- 2 \rho \timestep}{\Dim} \right) \|x - y\|^2.
\]
Thus, applying Corollary 21 \cite{Ollivier} for the Wasserstein distance of order $2$ and using Equation~(\ref{eq:secondmom}), 
\begin{align*}
W_2(\nu_n, \pi_\timestep) & \leq  \left(1 + \frac{\Lip^2 \timestep^2  - 2 \rho \timestep}{\Dim}\right)^{n/2} W_2(\nu_0, \pi_\timestep) \\
& \leq \left(1 + \frac{\Lip^2 \timestep^2  - 2 \rho \timestep}{\Dim}\right)^{n/2} \EE[\|X\|^2] \\
& \leq \left(1 + \frac{\Lip^2 \timestep^2  - 2 \rho \timestep}{\Dim}\right)^{n/2} \frac{2 \Dim}{2 \rho  - \Lip^2 \timestep}.
\end{align*}
Hence, if 
\[
n \leq \frac{2 \log\left(\frac{ \epsilon (2\rho - L^2 \timestep)}{2 d}\right)}{\log\left(1 + \frac{L^2 \timestep^2 - 2 \rho h}{d}\right)} = O(h^{-1}d \log(d/\epsilon)),
\]
then 
\beq 
\label{eq:lmcfirst}
W_2(\nu_n, \pi_\timestep) \leq \frac{\epsilon}{2}.
\eeq

Let us now bound $W_2(\pi_\timestep, \stationary)$. By Equation~(\ref{eq:assu1}), $\mu$ is a log-concave measure, implying that Assumption~\ref{ass:main} holds. On the other hand since $\EE[ \|X \|^2] < \infty$ and by Remark~\ref{rem:condition}, we can apply Corollary~\ref{cor:rw} with $T = \infty$, $\step = \frac{\timestep}{\Dim}$ and $\tau = d \timestep$ to obtain that there exists $C >0$ depending on $\rho$ such that 
\begin{multline}
\label{eq:mainMCMC}
W_2(\pi_\timestep, \stationary) \leq C \bigg(d \timestep \EE[\|\nabla u(X)\|^2]^{1/2} + d^{3/2} \sqrt{\timestep} + I_1 + \sqrt{d} I_2 + d \log(d \timestep)  I_3 \\
+ \sum_{k=4}^\infty \frac{C^k d ((k-1))^{(k-1)/2}}{\timestep^{(k-3)/2} k!} I_k \bigg),
\end{multline}
where
\begin{itemize}
\item $I_1 = \EE[\|\EE[\frac{\xi d}{\timestep}-\nabla u(X)\mid X]\|^2]^{1/2}$;
\item $I_2 = \EE[\|\EE[\frac{\xi^{\otimes 2} d}{2\timestep} - I_d\mid X]\|^2]^{1/2}$ and
\item $\forall k > 3, I_k = \EE[\|\EE[\frac{\xi^{\otimes k} d}{\timestep}\mid X]\|^2]^{1/2}$.
\end{itemize}
In the remainder of this proof, we are going to prove the following inequalities. 
\begin{enumerate}[label=(\roman*)]
\item $ \EE[\|\nabla u(X)\|^2]^{1/2}  \leq \frac{\Lip \Dim}{\rho} + O(dh)$;
\item $I_1 = 0$;
\item $I_2 \leq \frac{\sqrt{2} \Lip^2 (\Dim \timestep)^{1/2}}{\rho^{3/2}} + O(\Dim^{1/2} \timestep) $;
\item $I_3 \leq  (2  \Lip + 4 \rho ^{-2} \Lip^3  ) \rho^{-1/2} d^{1/2} \timestep + O(d^{1/2} \timestep^{3/2})$; 
\item $\forall k > 3, I_k \leq \left( 2\Lip \rho^{-1} + \sqrt{2}  \right)^{k} d^{1/2} \timestep^{k/2-1}$.
\end{enumerate}
Then, combining these bounds with Equation~(\ref{eq:mainMCMC}), there exists $C > 0$ such that 
\[
W_2(\pi_\timestep, \stationary) \leq C(\Dim^{3/2} \sqrt{\timestep}).
\]
Moreover, whenever $\stationary$ is the Gaussian measure, we can derive a more refined version of Corollary~\ref{cor:rw} using Theorem~\ref{thm:mainGauss} in which taking $\tau = \timestep$ gives
\[
W_2(\pi_\timestep, \stationary) \leq C( \sqrt{\Dim \timestep}).
\]
Overall, there exists $C > 0$, depending on $\rho$ and $\Lip$ such that, if 
\[
\timestep = C \epsilon^{2} \begin{cases} 
\Dim^{-1} \text{ if $\stationary$ is the Gaussian measure} \\
\Dim^{-3} \text{ otherwise} 
\end{cases},
\]
then 
\beq 
\label{eq:lmcsec}
W_2(\pi_\timestep, \stationary) \leq \frac{\epsilon}{2}. 
\eeq
Proposition~\ref{pro:LMC} is then obtained by combining Equations~(\ref{eq:lmcfirst}) and (\ref{eq:lmcsec}).
In order to conclude this proof, let us prove equalities (i)-(v). 

By Equations~(\ref{eq:assu2}) and (\ref{eq:secondmom}), we have 
\begin{align*}
 \EE[\|\nabla u(X)\|^2]^{1/2} & \leq \Lip \EE[\|X\|^2]^{1/2} \\
 & \leq \frac{\Lip \Dim}{\rho} + O(dh)
\end{align*}
and so inequality (i) is verified. 
Furthermore, equality (ii) is true by construction of $\xi$. 
Let us prove (iii). By Equation~(\ref{eq:assu2}), we have 
\begin{align*}
I_2^2 & = \EE\left[\left\|\EE\left[\frac{\xi^{\otimes 2} d}{2\timestep} - I_d\mid X\right]\right\|^2\right]  \\
& = \timestep^2 \EE\left[\sum_{i=1}^d (\nabla u (X))_i^{4} \right] \\
&  \leq \timestep^2 \Lip^4 \EE\left[\sum_{i=1}^d X_i^4\right]. 
\end{align*}
Hence, by  Equation~(\ref{eq:maxbound}) and Equation~(\ref{eq:secondmom}), 
\[
I_2^2 \leq \frac{4 \Lip^4 \timestep}{\rho^2} \EE[\|X\|^2] \leq \frac{2 \Lip^4 \Dim \timestep}{\rho^3} + O(\Dim \timestep^2).
\]
Let us now deal with (iv). We have 
\begin{align*}
I_3 &= \EE\left[\left\|\EE\left[\frac{\xi^{\otimes 3} d} {\timestep}\mid X\right]\right\|^2\right]^{1/2} \\
& = \EE\left[\sum_{i=1}^d \left(2 \timestep \frac{\partial u}{\partial(x_i)}(X) + \timestep^2 \left(\frac{\partial u}{\partial(x_i)}(X)\right)^3 \right)^2 \right]^{1/2}.
\end{align*}
Thus, by Equation~(\ref{eq:assu2}), 
\[
I_3 \leq \EE\left[\sum_{i=1}^d \left(2 \timestep \Lip |X_i| + \timestep^2 \Lip^3 |X_i|^3 \right)^2 \right]^{1/2}
\]
and, by Equation~(\ref{eq:maxbound}), 
\begin{align*}
I_3 & \leq \EE\left[\sum_{i=1}^d (2  \Lip + 4 \rho ^{-2} \Lip^3  )^2 \timestep^2 X_i^2 \right]^{1/2} \\
& \leq  (2  \Lip + 4 \rho ^{-2} \Lip^3  ) \timestep \EE[\|X\|^2]^{1/2}.
\end{align*}
And finally, by Equation~(\ref{eq:secondmom}), 
\[
I_3 \leq (2  \Lip + 4 \rho ^{-2} \Lip^3  ) \rho^{-1/2} d^{1/2} \timestep + O(d^{1/2} \timestep^{3/2}). 
\]
Finally, we have 
\begin{align*}
 I_k^2 &= \EE\left[\left\|\EE\left[\frac{\xi^{\otimes k} d}{\timestep}\mid X\right]\right\|^2\right] \\
 & = \frac{1}{\timestep^2} \EE_X\left[\sum_{i=1}^d \EE_B\left[\left(  -(\timestep \nabla u (X))_i + \sqrt{2 \timestep} B \right)^{k}\right]^2 \right] \\
 & \leq \frac{1}{\timestep^2} \sum_{i=1}^d \EE\left[ \left(  \Lip \timestep |X_i| + \sqrt{2 \timestep} \right)^{2k}\right] \\
 & \leq   \left( 2\Lip \rho^{-1} + \sqrt{2}  \right)^{2k} d \timestep^{k-2}.
\end{align*}
concluding the proof. 

\section{Technical results}
\label{sec:technical}
In this Section, we provide the proofs of the intermediary results used to derive Theorems~\ref{thm:mainGauss},\ref{thm:main2},\ref{thm:WpGauss} and \ref{thm:WpGaussexch}.
\subsection{Proof of Proposition~\ref{pro:curvdim}}
\label{sec:curvdim}

We are going to prove Proposition~\ref{pro:curvdim} by induction for the case $\rho \neq 0$, the case $\rho = 0$ can be obtained in a similar manner. First, by Equation~(\ref{eq:gradbound}), the result is true for $k = 1$. 
Now, let $k \in \NN$ and suppose that, 
\[
\forall t > 0,  \|\nabla^k P_t \test\|_a \leq e^{-\rho t \max(1, k/2)} \left( \frac{2 \rho d}{e^{(2 \rho t)/(k-1)} - 1}\right)^{(k-1)/2} (P_t \|\nabla \test\|_a^2)^{1/2}.
\]
Let $x \in \RR^d$ and 
let $(e_1,\dots,e_d)$ be an orthonormal basis of $\RR^d$ with respect to the $a(x)$-scalar product $<.,.>_{a(x)}$. 
We have 
\begin{align*}
\|\nabla^{k+1} P_t \test(x)\|_{a(x)}^2 & =  \sum_{i=1}^d \sum_{j \in \{1,\dots,d\}^{k}}  \left<\nabla^{k+1} P_t \test(x), e_1 \otimes e_{j_1} \otimes \dots \otimes e_{j_k} \right>_{a(x)}^2 \|\\
& = \sum_{i=1}^d \sum_{j \in \{1,\dots,d\}^{k}}   \left<\nabla^{k} < \nabla P_t \test(x),  e_1>_{a(x)}, e_{j_1} \otimes \dots \otimes e_{j_k} \right>_{a(x)}^2 \|\\
& = \sum_{i=1}^d \|\nabla^{k} <\nabla P_t \test(x), e_i>_{a(x)} \test\|_{a(x)}^2 \\
& = \sum_{i=1}^d  \lim_{\epsilon \rightarrow 0} \|\nabla^k (P_t \test (x + \epsilon a(x) e_i) - P_t\test (x) )\|_{a(x)}^2. 
\end{align*}
and thus
\[
\label{eq:curvdimaux}
\|\nabla^{k+1} P_t \test(x)\|_{a(x)}^2 
= \sum_{i=1}^d  \lim_{\epsilon \rightarrow 0} \|\nabla^k (P_t \test (x + \epsilon a(x) e_i) - P_t\test (x) )\|_{a(x)}^2.
\]
Let $\epsilon > 0$ and let $(X_t)_{t \geq 0}$ and $(\tilde{X}^\epsilon_t)_{t \geq 0}$ be two diffusion processes with infinitesimal generator $\mathcal{L}_\mu$ and started respectively at $x$ and $x+\epsilon a e_1$. Letting  $\psi_\epsilon : y \rightarrow \EE[\test(\tilde{X}^\epsilon_t) \mid X_t =y]$, we have
\begin{align*}
P_t \test (x + \epsilon a(x) e_i) - P_t\test (x) & = \EE[\test(\tilde{X}^\epsilon_t) - \test(X_t)] \\
& = \EE[ \EE[\test(\tilde{X}^\epsilon_t) \mid X_t] - \test(X_t)] \\
& = P_t(\psi_\epsilon - \test)(x).
\end{align*}
Then, using our induction hypothesis and the Markov property of the semigroup $(P_t)_{t \geq 0}$, we have that
\begin{multline*}
\left\|\nabla^k P_t(\psi_\epsilon - \test)(x) \right\|^2_{a(x)} = \left\|\nabla^k P_{t (k-1)/k} P_{t/k} (\psi_\epsilon - \test)(x) \right\|^2_{a(x)} \\
\leq  e^{- \rho t \max(2,k) \frac{k-1}{k}} \left( \frac{2 \rho d}{e^{2 \rho  t/{k}} - 1} \right)^{k-1} P_{t\frac{k-1}{k}} \left\|\nabla P_{t/k}(\psi_\epsilon - \test)(x) \right\|^2_{a(x)}.
\end{multline*}
Then, by Theorem 4.7.2~\cite{Markov} and Jensen's inequality, 
\begin{align*}
\left\|\nabla^k P_t (\psi - \test)(x)\right\|^2_{a(x)} & \leq 
 e^{-\rho t (k-1)} d^{k-1} \left( \frac{2 \rho }{e^{2 \rho t/k} - 1} \right)^{k} P_t \left|\psi - \test \right|^2(x) \\
 & \leq e^{-\rho t (k-1)} d^{k-1} \left( \frac{2 \rho }{e^{2 \rho t/k} - 1} \right)^{k} \EE\left[|\test(\tilde{X}^\epsilon_t) - \test(X_t)|^2\right].
\end{align*}
By Theorem 3.2.4 \citep{Markov} and Theorem 2.2 \citep{duality}, we can take $\tilde{X}^\epsilon_t$ such that, $d_a(\tilde{X}^\epsilon_t,X_t) \leq d_a(\tilde{X}^\epsilon_0, X_0) e^{-\rho t}$ almost surely. 
Using such $\tilde{X}^\epsilon_t$, we have 
\begin{align*}
\EE\left[\lim_{\epsilon \rightarrow 0} \left|\frac{ \test(\tilde{X}^\epsilon_t) -  \test(X_t) }{\epsilon} \right| \right] 
& \leq \EE\left[\lim_{\epsilon \rightarrow 0} \left|\frac{<\tilde{X}^\epsilon_t - X_t,\nabla \test(X_t) >}{\epsilon}  \right| \right] \\
& \leq \EE\left[\lim_{\epsilon \rightarrow 0} \left|\frac{\|\tilde{X}^\epsilon_t - X_t\|_{a^{-1}(X_t)}}{\epsilon} \|\nabla \test(X_t)\|_{a(X_t)} \right| \right] \\
& \leq \EE\left[\lim_{\epsilon \rightarrow 0} \left|\frac{d_a(\tilde{X}^\epsilon_t,X_t)}{\epsilon} \|\nabla \test(X_t)\|_{a(X_t)} \right| \right] \\
& \leq e^{-\rho t} \EE\left[\lim_{\epsilon \rightarrow 0} \left|\frac{d_a(\tilde{X}^\epsilon_0, X_0)}{\epsilon} \|\nabla \test(X_t)\|_{a(X_t)} \right| \right] \\
& \leq e^{-\rho t} \EE\left[\|\nabla \test(X_t)\|_{a(X_t)} \right].
\end{align*}
Since a similar result holds for all $(e_i)_{i \in \{1,\dots,d\}}$, combining this bound with Equation~(\ref{eq:curvdimaux}) yields 
\[
\|\nabla^{k+1} P_t \test\|^2_a \leq
e^{-\rho t (k+1)} \left( \frac{2 \rho d }{e^{2 \rho d t/k} - 1} \right)^{k} \left( \sqrt{P_t \|\nabla \test\|_a^2} \right)^2,
\]
concluding the proof. 

\subsection{Proof of Lemma~\ref{lem:analytic}}
\label{sec:analproof}
We prove the result under the assumptions corresponding to the general case. The result can be obtained similarly under the Gaussian assumptions. 

Before proving the Lemma, let us prove a crude bound on the functions $(f_k)_{k \geq 1}$ which will prove useful many times in the proofs. 
For $t> 0, k\in \NN$, let us recall that
\[
f_k(t) = \begin{cases} 
 e^{- \rho t \max(1,k/2)} \left( \frac{2\rho d}{e^{2\rho t/(k-1)}-1} \right)^{(k-1)/2}  \textit{ if } \rho \neq 0 \\
\left(\frac{d (k-1)}{t}\right)^{(k-1)/2} \textit{ if } \rho = 0 
\end{cases}.
\]
First, if $\rho > 0$, we have 
\[
\frac{2 \rho}{e^{2 \rho t / (k-1)} - 1} \leq \frac{k-1}{t}.
\]
On the other hand, if $\rho < 0$, 
\[
\frac{2 \rho}{e^{2 \rho t / (k-1)} - 1} = e^{2 |\rho| t / (k-1)} \frac{2 |\rho|}{e^{2 |\rho| t / (k-1)} - 1} \leq e^{2 |\rho| t / (k-1)}\frac{k-1}{t}.
\]
Thus, taking $D = \max(1, e^{- \rho T})$, the functions $f_k$ verify
 \beq
 \label{eq:fkbound}
\forall k \geq 1, f_k(t) \leq D^{\max(1, k/2) + 2/(k-1)} \left(\frac{d (k-1)}{t}\right)^{(k-1)/2} \leq  D^{2k} \left(\frac{d (k-1)}{t}\right)^{(k-1)/2} .
\eeq

Let $\test$ be a bounded and measurable function on $E$, let $t > 0$ and let $k$ be a strictly positive integer. By Proposition~\ref{pro:curvdim}, we have 
\[
\|\nabla^k P_t \test\|_a^2 \leq f_k(t) P_t \|\nabla \test\|_a^2
\]
and, since $\|\nabla \test\| \leq M$, 
\[
\|\nabla^k P_t \test\|_a^2 \leq M^2 f_k(t). 
\]
Then, by Equation~(\ref{eq:fkbound}), there exists $C >0$ such that 
\beq
\label{eq:derivativebornee}
\|\nabla^k P_t \test\|_a^2 \leq C^k (k-1)^{(k-1)/2}.
\eeq

Now, let $x,x' \in E$ and $l \in \NN$. Since $E$ is convex, we can use a Taylor expansion with remainder to obtain that there exists $\xi$ on the segment $(x,x')$ such that
\begin{multline*}
P_t \test(x') - P_t \test(x) = \\
\sum_{k=1}^{l-1} \frac{1}{k!} <(x'-x)^{\otimes k}, \nabla^k P_t \test(x) > + \frac{1}{l!} <(x'-x)^{\otimes l}, \nabla^l P_t \test (\xi)>. 
\end{multline*}
Applying Cauchy-Schwarz inequality, 
\begin{multline*}
\left| P_t \test(x') - P_t \test(x) - \sum_{k=1}^{l-1} \frac{1}{k!} <(x'-x)^{\otimes k}, \nabla^k P_t\test  (x) > \right| \leq \\
 \frac{1}{l!} \|x'-x\|^l_{a^{-1}(\xi)} \|\nabla^l P_t \test(\xi)\|_{a(\xi)}.
\end{multline*}
Since $a^{-1}$ is continuous, it is bounded on the segment $[x,x']$ so there exists $C' > 0$ such that 
\[
\|x' - x\|_{a^{-1}(\xi)} \leq \|x'-x\| \|a^{-1}(\xi)\|^{1/2} \leq C'.
\]
Therefore, by Equation~(\ref{eq:derivativebornee}), we obtain that 
\[
\left| P_t \test(x') - P_t \test(x) - \sum_{k=1}^l \frac{1}{k!} <(x'-x)^{\otimes k}, \nabla^k P_t \test (x)>\right| \leq (C C')^l \frac{(l-1)^{(l-1)/2}}{l!}. 
\]
By Stirling's formula, the right-hand term of the previous equation converges to zero when $l$ goes to infinity, thus
\[
\int_E \mathcal{L}_\nu P_t \test (x) d\nu(x) = \EE\left[\sum_{k=1}^\infty \frac{1}{s k!} \left<(X_t-X_0)^{\otimes k} , \nabla^k P_t \test(X_0)\right>\right].
\]
Finally, since, by assumption, 
\[
\EE\left[\sum_{k=1}^\infty \frac{1}{s k!} \|X_t-X_0\|_{a^{-1}(X_0)}^{k} \|\nabla^k P_t \test(X_0)\|\right] < \infty, 
\]
we can take the conditional expectation with respect to $X_0$ and conclude the proof.
\subsection{Proof of Lemma~\ref{lem:ippgauss}}
Let $x \in \RR^d$ and let $\test  \in C^\infty(\RR^d, \RR)$ be a bounded function.
The function $P_t \test$ admits the following representation (see e.g. \citep{Markov}, Equation 2.7.3)
\beq
\label{eq:OU}
P_t \test (x) = \int_{\RR^d} \test(xe^{-t} + \sqrt{1-e^{-2t}}y) d\gamma(y).
\eeq
Thus, we have 
\beq 
\label{eq:OUderiv}
\frac{\partial^k P_t \test}{\partial x_{i_1} \dots \partial x_{i_k}}(x) = e^{-kt} \int_{\RR^d} \frac{\partial^k \test}{\partial x_{i_1} \dots \partial x_{i_k}}(xe^{-t} + \sqrt{1-e^{-2t}}y) d\gamma(y)
\eeq
By definition, the Hermite polynomials satisfy the following recurrence relationship 
\[
\forall k \in \NN, i \in \{1,\dots,d\}^k, y \in \RR^d, H_{i_1,\dots,i_k} (y) = y_{i_k} H_{i_1,\dots,i_{k-1}} (y) - \frac{\partial H_{i_1,\dots,i_{k-1}}}{\partial y_{i_k}} (y).
\] 
Thus, letting $\psi(y) = \test(xe^{-t} + \sqrt{1- e^{-2t}}y)$ and taking $i \in \{1,\dots,d\}^k, j \in \{1,\dots,d\}$, integrating by parts with respect to $\gamma$ yields 
 \begin{align*}
\int_{\RR^d} H_i(y) \frac{\partial \test}{\partial x_j}(e^{-t} x + & \sqrt{1-e^{-2t}} y)  d\gamma(y)  \\
& = (e^{-2t} - 1)^{-1/2} \int_{\RR^d} H_i(y) \frac{\partial \psi}{\partial y_j}(y)  d\gamma(y) \\
 &= (e^{-2t} - 1)^{-1/2} \int_{\RR^d} \left( \frac{\partial H_i  \psi}{\partial y_j}(y)  - \frac{\partial H_i}{\partial y_j}(y) \psi(y) \right) d\gamma(y) \\
& =  (e^{-2t} - 1)^{-1/2} \int_{\RR^d} \left( y_j H_i (y)  \psi(y)  - \frac{\partial H_i}{\partial y_j}(y) \psi(y) \right) d\gamma(y) \\
 & =  (e^{-2t} - 1)^{-1/2} \int_{\RR^d} H_{i_1,\dots,i_k,j} (y) \psi(y) d\gamma(y) \\
 & = (e^{-2t} - 1)^{-1/2} \int_{\RR^d} H_{i_1,\dots,i_k,j} (y) \test(e^{-t} x + \sqrt{1-e^{-2t}} y) d\gamma(y).
 \end{align*}
Hence, taking $k \in \NN$ and performing $k-1$ successive integrations by parts on Equation~(\ref{eq:OUderiv}), we obtain 
\[
\frac{\partial^k P_t \test}{\partial x_{i_1} \dots \partial x_{i_k}}(x) = \frac{e^{-t}}{(e^{2t}-1)^{k/2}}  \int_{\RR^d} H_i (y) \test(xe^{-t} + \sqrt{1-e^{-2t}}y) d\gamma(y).
\]
\subsection{Proof of Lemma~\ref{lem:hypercon}}
\label{sec:hypercon}
Let $(M_k)_{k \in \NN}$ be such that $M_k \in (\RR^d)^{\otimes k}$ for any $k \in \NN$. 
Let us start with the case $1 \leq p < 2$. We have, by Jensen's inequality,
\[
\EE[\|\sum_{k=1}^\infty M_k \mathcal{H}_{k-1}(Z)\|^p]^{1/p} \leq \EE[\|\sum_{k=1}^\infty M_k \mathcal{H}_{k-1}(Z)\|^2]^{1/2}.
\]
Then, 
\begin{align*}
\EE[& \|\sum_{k=1}^\infty M_k \mathcal{H}_{k-1}(Z)\|^2] \\
& = \EE\left[\sum_{k=1}^\infty \sum_{j=1}^d \left(\sum_{i \in \{1,\dots,d\}^{k-1}} (M_k)_{j,i_1,\dots,i_{k-1}} H_i(Z) \right)^2\right] \\
& = \sum_{k,k'=1}^\infty \sum_{j=1}^d  \sum_{\substack{i \in \{1,\dots,d\}^{k-1} \\ {l \in \{1,\dots,d\}^{k'-1}}}}  (M_k)_{j,i_1,\dots,i_{k-1}} (M_{k'})_{j,l_1,\dots,l_{k'-1}} \EE[ H_i(Z) H_l(Z)].
\end{align*}
Finally, since, $\EE[H_i(Z) H_l(Z)]$ is equal to zero if $i \neq l$, we have 
\begin{align*}
\EE[\|\sum_{k=1}^\infty M_k \mathcal{H}_{k-1}(Z)\|^2] & = \sum_{k=1}^\infty \sum_{j=1}^d  \sum_{i \in \{1,\dots,d\}^{k-1}} (M_k)_{j,i_1,\dots,i_{k-1}}^2 \EE[H_i(Z)^2] \\
& = \sum_{k=1}^\infty \|M_k\|_{H}^2. 
\end{align*}

Let us now consider the case $p > 2$. Let us pose $t  = \log(\sqrt{p-1})$. 
Since the Ornstein-Uhlenbeck semigroup $(P_t)_{t \geq 0}$ is hypercontractive (see e.g. Theorem 5.2.3 \citep{Markov}), we have 
\[
\forall \test \in L^2(\gamma), \EE[|P_t \test(Z)|^p]^{1/p} \leq \EE[\test(Z)^2]^{1/2}. 
\]
This inequality can be extended to vector-valued functions $\test$, in which case we have
\[
\forall \test, \|\test\|\in L^2(\gamma), \EE[\|P_t \test(Z)\|^p]^{1/p} \leq \EE[(P_t \|\test(Z)\|)^p]^{1/p} \leq \EE[\|\test(Z)\|^2]^{1/2}.
\]
For any $k \in \NN$, the entries of $\mathcal{H}_{k-1}$ are Hermite polynomials and thus eigenvectors of $P_t$ with eigenvalue $e^{-(k-1)t} = (p-1)^{-(k-1)/2}$. 
Therefore 
\begin{align*}
\EE[\|\sum_{k=1}^\infty M_k \mathcal{H}_{k-1}(Z)\|^p]^{1/p} & = \EE[\|\sum_{k=1}^\infty (p-1)^{(k-1)/2} M_k P_t \mathcal{H}_{k-1}(Z)\|^p]^{1/p} \\
& =  \EE[\|P_t \sum_{k=1}^\infty (p-1)^{(k-1)/2} M_k  \mathcal{H}_{k-1}(Z)\|^p]^{1/p} \\
& \leq \EE[\|\sum_{k=1}^\infty (p-1)^{(k-1)/2} M_k  \mathcal{H}_{k-1}(Z)\|^2]^{1/2} \\
& \leq \left(\sum_{k=1}(p-1)^{k-1}  \|M_k\|_{H}^2\right)^{1/2},
\end{align*}
concluding the proof. 
\section{Approximation arguments}
\label{sec:approx}
In this Section, we present the approximation necessary to conclude the proofs of Theorems~\ref{thm:mainGauss},  \ref{thm:WpGauss}, \ref{thm:WpGaussexch} and \ref{thm:main2}.
\subsection{Gaussian case}
\label{sec:approxW2}
Let us first detail the approximation argument required to conclude the proof of 
Theorem \ref{thm:WpGaussexch}. Theorems~\ref{thm:mainGauss} and \ref{thm:WpGauss} can be obtained through similar computations. 

Suppose the measure $\nu$ and the stochastic process $(X_t)_{t \geq 0}$ satisfy the assumptions of Theorem~\ref{thm:WpGaussexch}.  
Let $s > 0$ and let
\begin{align*}
S_p(t)  & = \left\|\EE\left[\frac{X_t-X_0}{s}+X_0 \mid X_0\right]\right\|^2 \\
& +  \frac{ \max(1, p-1)}{e^{2t}-1}  \left\|\EE\left[\frac{(X_t-X_0)^{\otimes 2}}{2 s} - I_d \mid X_0 \right]\right\|^2 \\
& +  \sum_{k=3}^{\infty}  \frac{\max(1,p-1)^{k-1}}{4 (s (k-1)!)^2 (e^{2t}-1)^{k-1}}   \left\|\EE[(X_t-X_0)^{\otimes k} \mid X_0] \right\|_H^2.
\end{align*}

Let $R > 0, \epsilon_1 = R^{-1}$ and $0<\epsilon_2<1$ . For any $t > 0$, let $X_t^R$ be the orthogonal projection of $X_t$ on $\mathcal{B}(0,R)$, the ball of radius $R$ centred in $0$. Let $Z$ be a standard normal random variable, $N$ be a random variable taking values in the ball of radius $1$ with smooth density and let $I$ be a Bernoulli random variable with parameter $\epsilon_1$ such that $(X_t)_{t \geq 0}, Z, N$ and $I$ are independent. Finally, we pose $ U = \epsilon_1 N$. 
For any $t > 0$, we pose 
\[
\tilde{X}_t = I Z + (1-I) ( U +  X^R_t 1_{\epsilon_2 \leq t \leq 1/\epsilon_2} + X^R_0 1_{\epsilon_2 < t < \epsilon_2^{-1}}). 
\]
Let $\tilde{\nu}_R$ be the measure of $\tilde{X}_0$. This measure admits a density $h$ with respect to the measure $\gamma$ such that $h = \epsilon_2 + f$ with $f \in C^\infty_c(\RR^d, \RR^+)$. Furthermore, for any $t > 0$,  $(\tilde{X}_0,\tilde{X}_t)$ and  $(\tilde{X}_t,\tilde{X}_0)$ follow the same law. Therefore, we can apply the computations of Section~\ref{sec:WpGaussexch} and use the triangle inequality to obtain 
\begin{multline}
\label{eq:approxGaussmain}
W_p(\tilde{\nu}_R,\gamma) \leq \epsilon_1 \int_0^\infty e^{-t} \EE[S_Z(t)^{p/2}]^{1/p} dt + (\int_0^{\epsilon_2} + \int_{1/\epsilon_2}^\infty) e^{-t} \EE[\tilde{S}_{p,1}(t)^{p/2}]^{1/p} dt  \\
+ \int_{\epsilon_2}^{1/\epsilon_2}   e^{-t} \EE[\tilde{S}_{p,2}(t)^{p/2}]^{1/p} dt,
\end{multline}
where
\[
S_Z(t) =  \|Z\|^2 + \frac{  d(\max(1, p-1)) }{e^{2t} -1}
\]
and
\[
\tilde{S}_{p,1}(t) = \|X_0^R + U\|^2 +  \frac{d(\max(1, p-1))}{e^{2t} -1}
\]
and
\begin{align*}
\tilde{S}_{p,2}(t) =  & \left\|\EE\left[\frac{X^R_t-X^R_0}{s} + (X^R_0 +\epsilon_1  U) \mid X^R_0 + U\right]\right\|^2 \\
& + \frac{\max(1, p-1)}{e^{2t} -1}  \left\|\EE\left[\frac{(X^R_t-X^R_0)^{\otimes 2}}{2 s}  - I_d \mid X^R_0 + U \right]\right\|^2 \\
& +  \sum_{k=3}^{\infty} \frac{\max(1, p-1)^{k-1}}{4((k-1)! s)^2 (e^{2t}-1)^{k-1} }  \left\|\EE[(X^R_t-X^R_0)^{\otimes k}\mid X^R_0 + U]\right\|_H^2.
\end{align*}
First, since $Z$ admits a finite moment of order $p$, there exists $C > 0$ such that 
\beq
\label{eq:SZ}
\int_0^\infty e^{-t} \EE[S_Z(t)^{p/2}]^{1/p} dt \leq C. 
\eeq
Then, since $X_0^R$ is the orthogonal projection of $X_0$ on $\mathcal{B}(0,R)$, 
\[
\|X_0^R + U\| \leq \|X_0^R \| + U \leq \|X_0\| + \epsilon_1
\]
 and, since $\nu$ admits a finite moment of order $p$, we have that there exists $C > 0$ such that 
\beq
\label{eq:approxGaussaux1}
\EE[\tilde{S}_{p,1}(t)^{p/2}]^{1/p} \leq C\left( 1 + \epsilon_1+ \frac{1}{\sqrt{e^{2t} -1} }\right). 
\eeq
Now, let 
\begin{align*}
\tilde{S}_{p,3}(t) =  & \left\|\EE\left[\frac{X^R_t-X^R_0}{s} + X^R_0  \mid X^R_0 + U \right]\right\|^2 \\
& + \frac{\max(1, p-1)}{e^{2t} -1}  \left\|\EE\left[\frac{(X^R_t-X^R_0)^{\otimes 2}}{2 s}  - I_d \mid X^R_0 + U \right]\right\|^2 \\
& +  \sum_{k=3}^{\infty} \frac{(\max(1, p-1))^{k-1}}{4( (k-1)! s)^2 (e^{2t}-1)^{k-1} }  \left\|\EE[(X^R_t-X^R_0)^{\otimes k}\mid X^R_0 + U ]\right\|_H^2.
\end{align*}
By the triangle inequality, we have 
\[
\EE[\tilde{S}_{p,2}^{p/2}]^{1/p} \leq \EE[\tilde{S}_{p,3}^{p/2}]^{1/p} + \epsilon_1 \EE[\|U\|^{p}]^{1/p}
\]
and, since $\|U\| \leq \epsilon_1$, 
\beq
\label{eq:approxGaussaux0}
\EE[\tilde{S}_{p,2}^{p/2}]^{1/p} - \EE[\tilde{S}_{p,3}^{p/2}]^{1/p} \leq \epsilon_1. 
\eeq
Now, let 
\begin{align*}
\tilde{S}_{p,4}(t) =  & \left\|\EE\left[\frac{X^R_t-X^R_0}{s} + X^R_0  \mid X_0 \right]\right\|^2 \\
& + \frac{\max(1, p-1)}{e^{2t} -1}  \left\|\EE\left[\frac{(X^R_t-X^R_0)^{\otimes 2}}{2 s}  - I_d \mid X_0  \right]\right\|^2 \\
& +  \sum_{k=3}^{\infty} \frac{(\max(1, p-1))^{k-1}}{4( (k-1)! s)^2 (e^{2t}-1)^{k-1} }  \left\|\EE[(X^R_t-X^R_0)^{\otimes k}\mid X_0 ]\right\|_H^2.
\end{align*}
Since $(X^R_t)_{t \geq 0}$ and $U$ are independent and since $X^R_0$ is $X_0$-measurable, we have 
\[
\EE[\tilde{S}_{p,3}(t)^{p/2}]^{1/p} = \EE[\EE[\tilde{S}_{p,4}(t)\mid X^R_0 + U]^{p/2}]^{1/p}.
\]
Thus, applying Jensen's inequality yields 
\beq
\label{eq:approxGaussaux2}
\EE[\tilde{S}_{p,3}(t)^{p/2}]^{1/p} \leq \EE[\tilde{S}_{p,4}(t) ^{p/2}]^{1/p}.
\eeq
From here, we have 
\begin{align*}
\tilde{S}_{p,4}(t) & - S_p(t) - \left(\|X_0^R\|^2 + \frac{d \max(1, p-1)}{e^{2t} - 1} \right)  1_{X_0 \notin \mathcal{B}(0,R)} \\
&  \leq  \left(\sum_{k=1}^\infty \frac{(\max(1, p-1))^{k-1}}{s^2 (k-1)!} \EE[\|X^R_t - X_0^R\|^{2k} (1_{X_t \notin \mathcal{B}(0,R)} + 1_{X_0 \notin \mathcal{B}(0,R)}) \mid X_0]\right) \\
& \leq \EE\left[\sum_{k=1}^\infty \frac{(\max(1, p-1))^{k-1}}{s^2 (k-1)! (e^{2t}-1)^{k-1}}\|X^R_t - X_0^R\|^{2k} (1_{X_t \notin \mathcal{B}(0,R)} + 1_{X_0 \notin \mathcal{B}(0,R)}) \mid X_0 \right] \\
& \leq \frac{1}{s} \EE\left[\|X^R_t - X^R_0\|^2 e^{\frac{(\max(1, p-1)) \|X^R_t - X^R_0\|^2}{e^{2t} -1}}  (1_{X_t \notin \mathcal{B}(0,R)} + 1_{X_0 \notin \mathcal{B}(0,R)}) \mid X_0 \right]. 
\end{align*}
Thus, by the triangle inequality and by Jensen's inequality, 
\begin{align*}
\EE[\tilde{S}_{p,4}& (t)^{p/2}]^{1/p} \leq  \EE[S_{p}(t)^{p/2}]^{1/p} + \EE[\|X_0^R\|^{p} 1_{X_0 \notin \mathcal{B}(0,R)}]^{1/p} \\
&+ \left(P(X_0 \notin \mathcal{B}(0,R))\frac{d \max(1, p-1)}{e^{2t} - 1}\right)^{1/2} \\& 
+ \frac{1}{s} \EE\left[\|X^R_t - X^R_0\|^p e^{\frac{p \max(1, p-1) \|X^R_t - X^R_0\|^2}{2(e^{2t} -1)}}  (1_{X_t \notin \mathcal{B}(0,R)} + 1_{X_0 \notin \mathcal{B}(0,R)})\right]^{1/p}.
\end{align*}
Then, since $X^R_t$ is the orthogonal projection of $X_t$ on the convex set $\mathcal{B}(0,R)$, we have that $\|X_0^R\| \leq \|X_0\|$ and $\|X^R_t - X^R_0\| \leq \|X_t - X_0\|$. Thus,  
\begin{align*}
\EE[\tilde{S}_{p,4}& (t)^{p/2}]^{1/p} \leq \EE[S_{p}(t)^{p/2}]^{1/p} + \EE[\|X_0\|^{p} 1_{X_0 \notin \mathcal{B}(0,R)}]^{1/p}  \\ 
&+ \left(P(X_0 \notin \mathcal{B}(0,R))\frac{d \max(1, p-1)}{e^{2t} - 1}\right)^{1/2}  \\
&+ \frac{1}{s} \EE\left[\|X_t - X_0\|^p e^{\frac{p \max(1, p-1) \|X_t - X_0\|^p}{2(e^{2t} -1})}  (1_{X_t \notin \mathcal{B}(0,R)} + 1_{X_0 \notin \mathcal{B}(0,R)})\right]^{1/p}.
\end{align*}
By our assumption, there exists $\xi,M > 0$, depending on $\epsilon_2$, such that, for any $t \in [\epsilon_2, \epsilon_2^{-1}]$, 
\[
\EE\left[\|X_t - X_0\|^{p(1+\xi)} e^{\frac{ (1+\xi) p \max(1, p-1) \|X_t - X_0\|^2}{2(e^{2t}-1})}\right] \leq M. 
\]
Hence, using H\"older's inequality, we obtain that there exists $M', \xi' > 0$ such that 
\begin{multline*}
\EE[\tilde{S}_{p,4}(t)^{p/2}]^{1/p} \leq \EE[S_{p}(t)^{p/2}]^{1/p} + \EE[\|X_0\|^{p} 1_{X_0 \notin \mathcal{B}(0,R)}]^{1/p} \\
 + \left(P(X_0 \notin \mathcal{B}(0,R))\frac{d \max(1, p-1)}{e^{2t} - 1}\right)^{1/2}  
 + M' P(X_0 \notin \mathcal{B}(0,R)^{\xi'}.
\end{multline*}
Combining this bound with Equations~(\ref{eq:approxGaussmain}), (\ref{eq:SZ}), (\ref{eq:approxGaussaux1}), (\ref{eq:approxGaussaux0}) and (\ref{eq:approxGaussaux2}), we obtain that there exists $C > 0$ and $C_1(\epsilon_2),C_2(\epsilon_2) > 0$ such that 
\begin{multline*}
W_p(\tilde{\nu}_R,\gamma) \leq \int_0^\infty \EE[S_p(t)^{p/2}]^{1/p} \\
+ C(\epsilon_2 + e^{-1/\epsilon_2} + \epsilon_1 +  \EE[\|X_0\|^{p} 1_{X_0 \notin \mathcal{B}(0,R)}]^{1/p}  +  C_1(\epsilon_2)P(X_0 \notin \mathcal{B}(0,R))^{C_2(\epsilon_2)})).
\end{multline*}
Since $X_0$ has a finite moment of order $p$ and since $\epsilon_1 = R^{-1}$, letting $R$ go to infinity and $\epsilon_2$ go to zero yields 
\[
\lim_{R \rightarrow \infty} W_p(\tilde{\nu}_R,\gamma) \leq \int_0^\infty \EE[S_{p}(t)^{p/2}]^{1/p} dt.
\]
On the other hand, 
when $R$ goes to infinity, we have that $\tilde{\nu}_R$ converge weakly to $\nu$ and the $p$-moment of $\tilde{\nu}_R$ converges to the $p$-moment of $\nu$. Thus, by Theorem~6.9 \citep{Villani}, $W_p(\tilde{\nu}_R, \nu)$ converges to zero as $R$ goes to infinity. Therefore, 
\[
W_p(\nu,\gamma) \leq \lim_{R \rightarrow \infty} (W_p(\tilde{\nu}_R,\gamma) + W_p(\tilde{\nu}_R, \nu)) \leq \int_0^\infty \EE[S_{p}(t)^{p/2}]^{1/p} dt,
\]
concluding the proof. 
\subsection{General case}
\label{sec:approxWp}
Let us provide the approximation argument necessary to obtain Theorem~\ref{thm:main2}. 
Suppose Assumptions~\ref{ass:main} and \ref{ass:mainnu} are satisfied and let $T, s > 0$. 
We pose
\begin{align*}
 S(t) = & f_1(t)\left\|\EE\left[\frac{X_t-X_0}{s} - b(X_0) \mid X_0\right] \right\|_{a^{-1}(X_0)}\\
& +  f_2(t) \left\|\EE\left[\frac{(X_t-X_0)^{\otimes 2}}{2s} - a(X_0)\mid X_0\right] \right\|_{a^{-1}(X_0)}\\
& + \sum_{k=3}^{\infty} \frac{f_k(t)}{s k!} \left\|\EE[(X_t-X_0)^{\otimes k}\mid X_0]\right\|_{a^{-1}(X_0)},
\end{align*} 
where, for any $k \in \NN$, 
\[
f_k(t) = \begin{cases} 
 e^{- \rho t \max(1,k/2)} \left( \frac{2\rho d}{e^{2\rho t/(k-1)}-1} \right)^{(k-1)/2}  \textit{ if } \rho \neq 0 \\
\left(\frac{d (k-1)}{t}\right)^{(k-1)/2} \textit{ if } \rho = 0 
\end{cases}.
\]

Let $h$ be a density function with respect to the measure $\mu$ such that $\epsilon \leq h \leq \epsilon^{-1}$ for some $\epsilon > 0$. 
By Equation 5.6.2 \citep{Markov}, we have that 
\[
\forall x,y \in E, P_t \log h (x) \leq \log P_t h (y) + \frac{\rho d_a(x,y)^2}{2 (e^{2 \rho t} - 1)}
\]
and, by the triangle inequality, 
\[
\forall x,y \in E, P_t \log h (x) \leq  \log P_t h (y) + \frac{\rho (d_a(x,0) + d_a(0,y))^2}{2 (e^{2 \rho t} - 1)}. 
\]
Thus, we have 
\[
\forall x \in E, P_t \log h(x) \leq \int_E \left( \log P_t h (y) + \frac{\rho (d_a(x,0) + d_a(0,y))^2}{2 (e^{2 \rho t} - 1)} \right) d\mu(y).
\]
Then, applying Jensen's inequality yields 
\[
\int_E \log P_t h  d\mu \leq \log \int_E P_t h d \mu \leq 0
\]
and we obtain that 
\beq
\label{eq:logapprox}
\forall x \in E, P_t \log h(x) \leq  \frac{\rho }{2 (e^{2 \rho t} - 1)}  \int_E (d_a(x,0) +  d_a(0,y))^2 d\mu(y). 
\eeq

Now, let us consider a family $(K_n)_{n \in \NN}$ of compact sets such that for any $n \in \NN$, $K_n \subset K_{n+1} \subset E$ and $\cup_{n \in \NN} K_n = E$.  Let $0 < \epsilon_1, \epsilon_2, \epsilon_3 < 1$. Let $Z$ be a random variable drawn from $\mu$, $N$ be a random variable in the ball of radius $1$ with smooth density  and let $I$ be a Bernoulli random variable with parameter $\epsilon_2$ such that $N,I,Z, (X_t)_{t \geq 0}$ are independent. Finally, we pose $U = \epsilon_1 N$. 
For $t \geq 0$, we pose 
\[
\tilde{X}_t = IZ + (1-I) ( U + X_t 1_{X_t \in K_n, t > \epsilon_3} + X_0 1_{X_0 \in K_n, t \leq \epsilon_3} )
\]
and we denote the measure of $\tilde{X}_0$ by $\tilde{\nu}_{n, \epsilon_1, \epsilon_2}$. For any $x \in E, \test \in c^\infty_c(E)$ and $t> 0$, let 
\[
(\mathcal{L}_{\tilde{\nu}_{n, \epsilon_1, \epsilon_2}})_t \test (x)=  \frac{1}{s}\EE[I \mathcal{L}_\mu \test (\tilde{X}_t) + (1-I)( \test(\tilde{X}_t) - \test(\tilde{X}_0))  \mid \tilde{X}_0 = x].
\]
Let $t > 0$. Since $\tilde{X}_t$ and $\tilde{X}_0$ follow the same law and since $\EE[\mathcal{L}_\mu \test (Z)]= 0$, we have
\[
\EE[(\mathcal{L}_{\tilde{\nu}_{n, \epsilon_1, \epsilon_2}})_t \test (\tilde{X}_0)] = 0. 
\]
Rewriting $(\mathcal{L}_{\tilde{\nu}_{n, \epsilon_1, \epsilon_2}})_t$, we also have, for $t > \epsilon_3$, 
\begin{align*}
\EE[(\mathcal{L}_{\tilde{\nu}_{n, \epsilon_1, \epsilon_2}})_t \test (\tilde{X}_0^K)] & =  \EE[I \mathcal{L}_\mu \test(Z) ] + \EE\left[(1-I) (\test(X_t + U) - \test(X_0 + U))1_{ X_0,X_t \in K_n}\right] \\
& \qquad + \EE[(1-I) 1_{X_0 \notin K_n, X_t \in K_n}(\test(X_t + U) - \test(U))] \\
& \qquad  + \EE[(1-I) 1_{X_t \notin K_n, X_0 \in K_n}(\test(U) - \test(X_0 + U))]. 
\end{align*}
Then, for any bounded real analytic function $\test$ with bounded derivatives of all orders, 
\begin{align*}
\EE[(& \mathcal{L}_{\tilde{\nu}_{n, \epsilon_1, \epsilon_2}})_t \test (\tilde{X}_0)]  \\
& =  \EE[ I \mathcal{L}_\mu \test(Z) ] +  \EE\left[1_{X_0, X_t \in K_n} (1-I)  \sum_{k=1}^\infty \frac{1}{sk!} \left< (X_t - X_0)^{\otimes k}, \nabla^k \test(X_0 + U) \right >\right] \\
& \qquad + \EE[(1-I) 1_{X_0 \notin K_n, X_t \in K_n}(\test(X_t + U) - \test(U))] \\
& \qquad  + \EE[(1-I) 1_{X_t \notin K_n, X_0 \in K_n}(\test(U) - \test(X_0 + U))]. 
\end{align*}

Since $K_n$ is a compact set, there exists a compact set $K_n' \subset E$ and $e(n) > 0$ such that, if $\epsilon_1 < e(n)$, then $(X_0 + U)1_{X_0 \in K_n} \in K_n'$. 
Let us now assume that $\epsilon_1 < e(n)$. 
Then, $\tilde{\nu}_{n, \epsilon_1, \epsilon_2}$ admits a measure $h$ with respect to $\mu$ such that $h = \epsilon + f$ with $\epsilon > 0$ and $f \in C^\infty_c(E,\RR)$. 
Thus, for $t > 0$, we can follow the computations of the proof of Proposition~\ref{pro:Ibound} to obtain 
\[
I((\tilde{\nu}_{n, \epsilon_1, \epsilon_2})_t)_\mu = \int_E ((\mathcal{L}_{\tilde{\nu}_{n, \epsilon_1, \epsilon_2}})_t - \mathcal{L}_\mu) P_t \log P_t h (x) d \tilde{\tilde{\nu}_{n, \epsilon_1, \epsilon_2}(x)}.
\]
Then, using Proposition~\ref{pro:curvdim} along with Cauchy-Schwarz inequality, we obtain 
\[
I((\tilde{\nu}_{n, \epsilon_1, \epsilon_2})_t)_\mu \leq  \begin{cases}
  \EE[(1-I) F(t)^2]^{1/2} I((\tilde{\nu}_{n, \epsilon_1, \epsilon_2})_t)_\mu^{1/2} \text{ if $t \leq \epsilon_3$} \\
  \EE[(1-I) \tilde{S}(t)^2]^{1/2} I((\tilde{\nu}_{n, \epsilon_1, \epsilon_2})_t)_\mu^{1/2} + E(t) \text{ if $t > \epsilon_3$}
 \end{cases}
\]
where 
\[
F(t) = f_1(t) \|b(U + X_0 1_{X_0 \in K_n})\|_{a^{-1}(U + X_0 1_{X_0 \in K_n})} + f_2(t) \sqrt{d}
\]
and
\begin{multline*}
E(t) = \EE[1_{X_0 \notin K_n, X_t \in K_n}(P_t \log P_t h (X_t + U) - P_t \log P_t h(U))] \\
 + \EE[1_{X_t \notin K_n, X_0 \in K_n}(P_t \log P_t h(U) - P_t \log P_t h(X_0 + U))]
\end{multline*}
and
\begin{align*}
\tilde{S}&(t)  = f_1(t) \|b(U)\|_{a^{-1}(U)} 1_{X_0 \notin K_n} + f_2(t) \sqrt{d} 1_{X_0 \notin K_n} \\
&  f_1(t) \left\|\EE\left[\frac{X_t-X_0}{s} 1_{X_t \in K_n}  - b(X_0 + U) \mid X_0 + U\right]  \right\|_{a^{-1}(X_0 + U)} 1_{X_0 \in K_n} \\
& +  f_2(t) \left\|\EE\left[\frac{(X_t-X_0)^{\otimes 2}}{2s} 1_{ X_t \in K_n} - a(X_0 + U)\mid X_0 + U\right] \right\|_{a^{-1}(X_0 + U)}1_{X_0 \in K_n}  \\
& + \sum_{k=3}^{\infty} \frac{f_k(t)}{s k!} \left\|\EE[(X_t-X_0)^{\otimes k} 1_{ X_t \in K_n}\mid X_0 + U]\right\|_{a^{-1}(X_0 + U)} 1_{X_0 \in K_n} ,
\end{align*}
with the functions $f_k$ defined as in Proposition~\ref{pro:curvdim}. 
Finally, by Equation~(\ref{eq:OV}) and by our assumptions, we obtain that 
\beq
\label{eq:approxImain}
(1 - c e^{- \kappa T}) W_{2,a}(\tilde{\nu}_{n, \epsilon_1, \epsilon_2}, \mu) \leq \int_0^{\epsilon_3} \EE[F(t)^2]^{1/2} dt + \int_{\epsilon_3}^T \tilde{S}(t)^{1/2} + E(t)^{1/2} dt. 
\eeq

Let us start by bounding $E(t)$. Let $ t\geq \epsilon_3$, by Equation~(\ref{eq:logapprox}), we have  
\begin{align*}
E(t) & \leq  \frac{\rho}{2 (e^{2 \rho t} - 1)}  \EE[(d_a(X_t + U,0) + d_a(0,Z))^2 1_{X_0 \notin K_n, X_t \in K_n}  \\
& + (d_a(X_0 + U,0) + d_a(0,Z))^2 1_{X_t \notin K_n, X_0 \in K_n} \\
& + (d_a(U,0) + d_a(0,Z))^2 (1_{X_0 \notin K_n, X_t \in K_n} + 1_{X_t \notin K_n, X_0 \in K_n})]
\end{align*}
and thus 
\begin{multline*}
E(t)^{1/2} \leq \frac{\sqrt{\rho}}{\sqrt{2 (e^{2 \rho t} - 1)}} \bigg(2 \EE[d_a(0,Z)^2]^{1/2} P(X_0 \notin K_n) + 2 \EE[d_a(U,0)^2]^{1/2}  \\
 + \EE[d_a(X_t, 0)^2 1_{X_0 \notin K_n}]^{1/2} + \EE[d_a(X_0, 0)^2 1_{X_t \notin K_n}]^{1/2}\bigg).
\end{multline*}
Since $0 \in K_n$ and $\epsilon_1 \leq e(n)$, $U$ is supported on a compact set on which $\|a^{-1}\|$ is bounded by $C > 0$. Hence, by definition of the metric $d_a$ and by Cauchy-Schwarz inequality, 
\[
d_a(0,U) \leq \int_0^1  \|U\|_{a^{-1}(t U)} dt \leq \int_0^1  \|U\| \|a^{-1}(t U)\| dt \leq C \|U\| \leq C \epsilon_1. 
\]
On the other hand, since $d_a(.,0) \in L_2(\mu)$, we know there exists $C > 0$ such that 
\[
\EE[d_a(0,Z)^2]^{1/2} \leq C. 
\]
Therefore, there exists $C(\epsilon_3) > 0$ such that 
\[
E(t)^{1/2} \leq C(\epsilon_3) \bigg(P(X_0 \notin K_n)^{1/2} 
+ \epsilon_1 + \EE[d_a(X_t, 0)^2 1_{X_0 \notin K_n}]^{1/2} + \EE[ d_a(X_0, 0)^2 1_{X_t \notin K_n}]^{1/2}\bigg).
\]
Now, by assumption, there exists $\xi(\epsilon_3), M(\epsilon_3) > 0$ such that 
\[
\EE[d_a(X_t, X_0)^{2 + \xi(\epsilon_3)}]^{1/(2 + \xi(\epsilon_3))} \leq M(\epsilon_3). 
\]
Thus, we can use the triangle inequality and H\"older's inequality to obtain that there exists $\xi(\epsilon_3)', M(\epsilon_3)' > 0$ such that
\begin{align*}
\EE[& d_a(X_t, 0)^2  1_{X_0 \notin K_n}]^{1/2} \\
& \leq \EE[d_a(X_t, X_0)^2 1_{X_0 \notin K_n}]^{1/2} + \EE[d_a(X_0, 0)^2 1_{X_0 \notin K_n}]^{1/2}  \\
&  \leq M(\epsilon_3)' P(X_0 \notin K_n)^{\xi(\epsilon_3)'}  + \EE[d_a(X_0, 0)^2 1_{X_0 \notin K_n}]^{1/2}  \\
\end{align*}
Therefore, we obtained there exists $C(\epsilon_3), \xi'(\epsilon_3) > 0$ such that 
\beq
\label{eq:approx2}
E(t)^{1/2} \leq C(\epsilon_3) \bigg( \epsilon_1 + \EE[d_a(X_0,0)^2 1_{X_0 \notin K_n} ]^{1/2}
  +  P(X_0 \notin K_n)^{\xi'(\epsilon_3)}\bigg).
\eeq

In the following, we denote by $C(n)$ a generic constant depending only on $n$. 
Since $b$ is continuous on $E$ and since the coefficients of $a^{-1}$ are in $C^\infty(E, \RR)$, we have that for any $x,y \in K_n'$, $\|b(x)\| \leq C(n), \|b(y) - b(x)\| \leq C(n) \|y-x\|$ and for any $x \in K_n'$, $i,j \in \{1,\dots,d\}, \| a^{-1}_{i,j}(x) \| \leq C(n)$ and $\|\nabla a^{-1}_{i,j}(x)\| \leq C(n)$. Therefore, for any $k \in \NN$, any $x \in K_n'$ and any $i,j \in \{1,\dots,d\}^k, \|\nabla \Pi_{l=1}^k a^{-1}_{i_l,j_l}(x)\| \leq C(n)^k$. 
Let $x,y \in K_n'$,  $k \in \NN$ and $u \in (\RR^d)^{\otimes k}$. Using Cauchy-Schwarz inequality we obtain that 
\begin{align*}
\|u\|^2_{a^{-1}(y)} - \|u\|^2_{a^{-1}(x)} & = \sum_{i,j \in \{1,\dots,d\}^k} u_i u_j (\Pi_{l=1}^k a^{-1}_{i_l,j_l}(y) - \Pi_{l=1}^k a^{-1}_{i_l,j_l}(x)) \\
& \leq \|u\|^2 \left(\sum_{i,j \in \{1,\dots,d\}^k} \left(\Pi_{l=1}^k a^{-1}_{i_l,j_l}(y) - \Pi_{l=1}^k a^{-1}_{i_l,j_l}(x)\right)^2 \right)^{1/2}
\end{align*}
and thus,
\beq
\label{eq:approxaux1}
\|u\|^2_{a^{-1}(y)} - \|u\|^2_{a^{-1}(x)} \leq C(n)^k \|u\|^2 \|y-x\|. 
\eeq

Let us bound $F(t)$.
First, since $\epsilon_1 < e(n)$, there exists $C> 0$ such that 
\[
 \|b(U)\|_{a^{-1}(U)} \leq C.
\]
Now, using the triangle inequality and Cauchy-Schwarz inequality,  
\begin{align*}
 \EE[\|b(& U + X_0 1_{X_0 \in K_n})\|_{a^{-1}(U + X_0 1_{X_0 \in K_n})}^2]^{1/2}  \\
 & \leq  \EE[\|b(X_0 1_{X_0 \in K_n})\|_{a^{-1}(U + X_0 1_{X_0 \in K_n})}^2]^{1/2}  \\
 & \qquad + \EE[\|b(X_0 1_{X_0 \in K_n}) - b(U + X_0 1_{X_0 \in K_n})\|_{a^{-1}(U + X_0 1_{X_0 \in K_n})}^2]^{1/2}  \\
 & \leq \EE[\|b(X_0 1_{X_0 \in K_n})\|_{a^{-1}(U + X_0 1_{X_0 \in K_n})}^2]^{1/2} \\
 & \qquad +  \EE[\|a^{-1}(U + X_0 1_{X_0 \in K_n})\| \|b(X_0 1_{X_0 \in K_n}) - b(U + X_0 1_{X_0 \in K_n})\|^2]^{1/2}\\
 & \leq \EE[\|b(X_0 1_{X_0 \in K_n})\|_{a^{-1}(U + X_0 1_{X_0 \in K_n})}^2]^{1/2} \\
& \qquad  +  C(n) \EE[\|b(X_0 1_{X_0 \in K_n}) - b(U + X_0 1_{X_0 \in K_n})\|^2]^{1/2}\\
 & \leq \EE[\|b(X_0 1_{X_0 \in K_n})\|_{a^{-1}(U + X_0 1_{X_0 \in K_n})}^2]^{1/2} + \epsilon_1  C(n) .
\end{align*}
Then, by Equation~(\ref{eq:approxaux1}),
\[
\EE[\|b(X_0 1_{X_0 \in K_n})\|_{a^{-1}(U + X_0 1_{X_0 \in K_n})}^2]^{1/2} \leq \EE[\|b(X_0 1_{X_0 \in K_n})\|_{a^{-1}(X_0 1_{X_0 \in K_n})}^2]^{1/2} +\sqrt{\epsilon_1} C(n).
\]
Therefore, 
\[
F(t) \leq f_1(t) (\EE[\|b(X_0)\|_{a^{-1}(X_0)}^2]^{1/2} + C + \sqrt{\epsilon_1} C(n)) + f_2(t) \sqrt{d}.
\]
Since $\|b\|_{a^{-1}} \in L_2(\nu)$ and by Equation~(\ref{eq:fkbound}), we obtain that there exists $C > 0$ such that 
\beq
\label{eq:approx1}
\int_0^{\epsilon_3} \EE[F(t)^2]^{1/2} dt \leq C\sqrt{\epsilon_3} \left( 1 +  C(n) \sqrt{\epsilon_1}   \right). 
\eeq

Finally, let us bound $\EE[\tilde{S}(t)^2]^{1/2}$ for $t \in [\epsilon_3, T]$. 
Let us first pose 
\begin{align*}
R_1 & =  f_1(t) \left\|\EE\left[\frac{X_t-X_0}{s} 1_{X_t \in K_n}  - b(X_0) \mid X_0 + U\right]  \right\|_{a^{-1}(X_0 + U)} 1_{X_0 \in K_n} \\
& +  f_2(t) \left\|\EE\left[\frac{(X_t-X_0)^{\otimes 2}}{2s} 1_{ X_t \in K_n} - a(X_0)\mid X_0 + U\right] \right\|_{a^{-1}(X_0 + U)}1_{X_0 \in K_n}  \\
& + \sum_{k=3}^{\infty} \frac{f_k(t)}{s k!} \left\|\EE[(X_t-X_0)^{\otimes k} 1_{ X_t \in K_n}\mid X_0 + U]\right\|_{a^{-1}(X_0 + U)} 1_{X_0 \in K_n}.
\end{align*}
By the triangle inequality and by Jensen's inequality, we have 
\begin{align*}
\EE[\tilde{S}(t)^2]^{1/2} & - \EE[R_1^2]^{1/2} \leq f_1(t) \EE[\|b(U)\|_{a^{-1}(U)}^2 1_{X_0 \notin K_n}]^{1/2}  +f_2(t) \sqrt{d} P(X_0 \notin K_n)^{1/2} \\
& + f_1(t) \EE[\| b(X_0 + U) - b(X_0) \|^2_{a^{-1}(X_0 + U)} 1_{X_0 \in K_n}]^{1/2} \\
& + f_2(t) \EE[\|a(X_0 + U) - a(X_0) \|^2_{a^{-1}(X_0 + U)} 1_{X_0 \in K_n}]^{1/2}. 
\end{align*}
Since $\epsilon_1 < e(n)$, there exists $C > 0$ such that $\|b(U)\|_{a^{-1}(U)} \leq C$ and 
\begin{multline*}
\EE[\tilde{S}(t)^2]^{1/2} - \EE[R_1^2]^{1/2} \leq f_1(t) \left(\epsilon_1 C(n) + C P(X_0 \notin K_n)^{1/2} \right) \\
 + f_2(t) \left(C(n) + \sqrt{d} P(X_0 \notin K_n)^{1/2} \right). 
\end{multline*}
Therefore,there exists $C,C(\epsilon_3) > 0$ such that 
\beq
\label{eq:Sbound1}
\EE[\tilde{S}(t)^2]^{1/2} - \EE[R_1^2]^{1/2} \leq  C(\epsilon_3)(C(n) \epsilon_1  + C P(X_0 \notin K_n)^{1/2}). 
\eeq

Now, let 
\begin{align*}
R_2 & =  f_1(t) \left\|\EE\left[\frac{X_t-X_0}{s} 1_{X_t \in K_n}  - b(X_0) \mid X_0\right]  \right\|_{a^{-1}(X_0 + U)} 1_{X_0 \in K_n}\\
& +  f_2(t) \left\|\EE\left[\frac{(X_t-X_0)^{\otimes 2}}{2s} 1_{ X_t \in K_n} - a(X_0)\mid X_0\right] \right\|_{a^{-1}(X_0 + U)}1_{X_0 \in K_n} \\
& + \sum_{k=3}^{\infty} \frac{f_k(t)}{s k!} \left\|\EE[(X_t-X_0)^{\otimes k} 1_{ X_t \in K_n}\mid X_0]\right\|_{a^{-1}(X_0 + U)} 1_{X_0 \in K_n}.
\end{align*}
Since $U$ and the process $(X_t)_{t \geq 0}$ are independent, we have, by Jensen's inequality,
\beq
\label{eq:Sbound3}
\EE[R_1^2]^{1/2} \leq \EE[R_2^2]^{1/2}. 
\eeq

Now, let us pose 
\begin{align*}
R_3 & =  f_1(t) \left\|\EE\left[\frac{X_t-X_0}{s} 1_{X_t \in K_n}  - b(X_0) \mid X_0\right]  \right\|_{a^{-1}(X_0)} 1_{X_0 \in K_n}\\
& +  f_2(t) \left\|\EE\left[\frac{(X_t-X_0)^{\otimes 2}}{2s} 1_{ X_t \in K_n} - a(X_0)\mid X_0\right] \right\|_{a^{-1}(X_0)}1_{X_0 \in K_n} \\
& + \sum_{k=3}^{\infty} \frac{f_k(t)}{s k!} \left\|\EE[(X_t-X_0)^{\otimes k} 1_{ X_t \in K_n}\mid X_0 ]\right\|_{a^{-1}(X_0)} 1_{X_0 \in K_n}.
\end{align*}
By Equation~(\ref{eq:approxaux1}) and by the triangle inequality, 
\begin{multline*}
\EE[R_2^2]^{1/2} - \EE[R_3^2]^{1/2} \leq \epsilon_1^{1/2} C(n) \bigg(f_1(t) \EE[ \|b(X_0)\|^2 1_{X_0 \in K_n}]^{1/2} \\ 
+ f_2(t) \EE[\|a(X_0)\|^2 1_{X_0 \in K_n}]^{1/2}  
+ \frac{1}{s} \EE\left[\left(\sum_{k=1}^\infty \frac{f_k(t) C(n)^k}{k!} \|X_t - X_0\|^{2k} 1_{X_0, X_t \in K_n} \right)^2 \right]^{1/2}\bigg).
\end{multline*}
We have $\|X_t - X_0\| 1_{X_0, X_t \in K_n} \leq 2 D$, where $D$ is the diameter of the compact set $K_n$. Thus, by Equation~(\ref{eq:fkbound}) and, since $t \geq \epsilon_3$, there exists $C(n, \epsilon_3) > 0$ such that 
\beq
\label{eq:Sbound2}
\EE[R_2^2]^{1/2} - \EE[R_3^2]^{1/2} \leq C(n, \epsilon_3) \epsilon_1^{1/2}.
\eeq

Finally, we have 
\[
\EE[R_3^2]^{1/2} - \EE[S(t)^2]^{1/2} \leq \EE\left[ \left(\sum_{k=1}^\infty \frac{f_k(t)}{k!} \|X_t - X_0\|_{a^{-1}(X_0)}^k  \right)^2 1_{X_t \notin K_n}\right]^{1/2}.
\]
By our assumption, there exists $\xi(\epsilon_3), M(\epsilon_3) > 0$ such that 
\[
\EE\left[ \left(\sum_{k=1}^\infty \frac{f_k(t)}{k!} \|X_t - X_0\|_{a^{-1}(X_0)}^k  \right)^{2 + \xi(\epsilon_3)} \right]^{1/(2 + \xi(\epsilon_3))} \leq M(\epsilon_3). 
\]
Thus, by H\"older inequality, there exists $M'(\epsilon_3)$ and $\xi'(\epsilon_3)$ such that 
\beq
\label{eq:Sbound4}
\EE[R_3^2]^{1/2} - \EE[S(t)^2]^{1/2} \leq M'(\epsilon_3) P(X_t \notin K_n)^{\xi'(\epsilon_3)} .
\eeq
By combining Equations~(\ref{eq:Sbound1}), (\ref{eq:Sbound3}), (\ref{eq:Sbound2}) and (\ref{eq:Sbound4}), we finally obtain that there exists $C(n, \epsilon_3), C(\epsilon_3)$ and $\xi(\epsilon_3) > 0$ such that 
\beq
\label{eq:Sboundfinal}
\EE[\tilde{S}(t)^2]^{1/2} \leq \EE[S(t)^2]^{1/2} + C(n,\epsilon_3) \epsilon_1^{1/2} + C(\epsilon_3) P(X_0 \notin K_n)^{\xi(\epsilon_3)}. 
\eeq

Now, injecting in Equation~(\ref{eq:approxImain}) the bounds obtained in Equations~(\ref{eq:approx2}), (\ref{eq:approx1}) and (\ref{eq:Sboundfinal}), we obtain that  
\begin{multline*}
(1 - e^{- c \kappa T}) W_{2,a}(\tilde{\nu}_{n, \epsilon_1, \epsilon_2}, \mu) \leq   \int_0^T \EE[S(t)^2]^{1/2} dt + C \epsilon_3^{1/2} (1 + C(n) \epsilon_1^{1/2}  )   \\ 
+ 
  C(n,\epsilon_3) \epsilon_1^{1/2} + C(\epsilon_3)( \EE[d_a(X_0,0)^2 1_{X_0 \notin K_n}]^{1/2} 
  +  P(X_0 \notin K_n)^{\xi(\epsilon_3)}) .
\end{multline*}
Since $d_a(.,0) \in L_2(\nu)$, letting $\epsilon_1$ go to zero, $n$ go to infinity and $\epsilon_2, \epsilon_3$ go to zero yields 
\[
\lim\limits_{ \epsilon_2 \rightarrow 0} \lim\limits_{ n \rightarrow \infty} \lim\limits_{ \epsilon_1 \rightarrow 0} W_{2,a}(\tilde{\nu}_{n, \epsilon_1, \epsilon_2}, \mu) \leq \int_0^T \EE[S(t)^2]^{1/2} dt
\]
and we conclude the proof by remarking that, since $d_a(.,0) \in L_2(\nu) \cap L_2(\mu)$, then by Theorem 6.9 \citep{Villani}, 
\[
\lim\limits_{ \epsilon_2\rightarrow 0} \lim\limits_{ n \rightarrow \infty} \lim\limits_{ \epsilon_1 \rightarrow 0} W_{2,a}(\tilde{\nu}_{n, \epsilon_1, \epsilon_2}, \nu) = 0.
\]